\documentclass[a4j,10pt,showkeys]{article}

\usepackage{amsmath,amsthm,amssymb}
\usepackage{geometry}
\usepackage{hyperref}
\usepackage[capitalize]{cleveref}
\usepackage{cancel}
\usepackage{mathtools}
\usepackage{bm}
\usepackage{color}
\usepackage{mathrsfs}
\usepackage[dvipsnames]{xcolor}
\hypersetup{
    colorlinks=true,
    citecolor=MidnightBlue,
    linkcolor=Red,
    urlcolor=OliveGreen,
}
\usepackage{tikz}
\usetikzlibrary{calc}
\tikzset{set label/.style={fill=white}}

\newtheorem{theo}{Theorem}[section]
\newtheorem{lemm}[theo]{Lemma}
\newtheorem{prop}[theo]{Proposition}

\theoremstyle{definition}
\newtheorem{defi}[theo]{Definition}
\newtheorem{rem}[theo]{Remark}
\newtheorem{assum}{Assumption}

\newcommand{\bE}{\mathbb{E}}
\newcommand{\bF}{\mathbb{F}}

\newcommand{\bH}{\mathbb{H}}

\newcommand{\bN}{\mathbb{N}}

\newcommand{\bP}{\mathbb{P}}

\newcommand{\bR}{\mathbb{R}}
\newcommand{\bS}{\mathbb{S}}

\newcommand{\cA}{\mathcal{A}}
\newcommand{\cB}{\mathcal{B}}

\newcommand{\cD}{\mathcal{D}}
\newcommand{\cE}{\mathcal{E}}
\newcommand{\cF}{\mathcal{F}}
\newcommand{\cG}{\mathcal{G}}
\newcommand{\cH}{\mathcal{H}}

\newcommand{\cL}{\mathcal{L}}

\newcommand{\cN}{\mathcal{N}}

\newcommand{\cP}{\mathcal{P}}
\newcommand{\cQ}{\mathcal{Q}}

\newcommand{\cU}{\mathcal{U}}
\newcommand{\cV}{\mathcal{V}}

\newcommand{\sP}{\mathscr{P}}
\newcommand{\sQ}{\mathscr{Q}}

\newcommand{\fg}{\mathfrak{g}}
\newcommand{\fh}{\mathfrak{h}}
\newcommand{\fp}{\mathfrak{p}}
\newcommand{\fq}{\mathfrak{q}}
\newcommand{\fm}{\mathfrak{m}}

\newcommand{\1}{\mbox{\rm{1}}\hspace{-0.25em}\mbox{\rm{l}}} %% indicator function
\newcommand{\ep}{\varepsilon} %% varepsilon
\newcommand{\vth}{\vartheta} %% vartheta
\newcommand{\diff}{\mathrm{d}} %% differential
\newcommand{\dmu}{\,\mu(\diff\theta)} %% integral by the measure $\mu$
\newcommand{\dmumu}{\,\mu(\diff\theta_1)\mu(\diff\theta_2)} %% integral by the measure $\mu^{\otimes2}$
\newcommand{\V}{\mathrm{V}} %% Volterra
\DeclareMathOperator*{\esssup}{ess\,sup} %% essential supremum

\providecommand{\keywords}[1]{\textbf{Keywords:} #1}
\makeatletter

\@addtoreset{equation}{section}
\makeatother
\def\widebar{\accentset{{\cc@style\underline{\mskip10mu}}}}
\numberwithin{equation}{section}
\allowdisplaybreaks

\title{Global maximum principle for optimal control of stochastic Volterra equations with singular kernels: An infinite dimensional approach}

\author{
Yushi Hamaguchi\footnote{Department of Mathematics, Kyoto University, Kyoto 606-8502, Japan. Email: \href{mailto:hmgch2950@gmail.com}{hmgch2950@gmail.com}. The author was supported by JSPS KAKENHI Grant Number 22K13958.}
}

\begin{document}
\maketitle

%% Abstract

\begin{abstract}
In this paper, we consider optimal control problems of stochastic Volterra equations (SVEs) with singular kernels, where the control domain is not necessarily convex. We establish a global maximum principle by means of the spike variation technique. To do so, we first show a Taylor type expansion of the controlled SVE with respect to the spike variation, where the convergence rates of the remainder terms are characterized by the singularity of the kernels. Next, assuming additional structure conditions for the kernels, we convert the variational SVEs appearing in the expansion to their infinite dimensional lifts. Then, we derive first and second order adjoint equations in form of infinite dimensional backward stochastic evolution equations (BSEEs) on weighted $L^2$ spaces. Moreover, we show the well-posedness of the new class of BSEEs on weighted $L^2$ spaces in a general setting.
\end{abstract}

%% Keywords

\keywords
global maximum principle; stochastic Volterra equation; infinite dimensional lift; backward stochastic evolution equation.

%% MSC

\textbf{2020 Mathematics Subject Classification}: 93E20; 49K45; 60H20; 45D05.

%93E20 Optimal stochastic control
%49K45 Optimality conditions for problems involving randomness
%60H20 Stochastic integral equations
%45D05 Volterra integral equations

%45G05 Singular nonlinear integral equations
%60H15 Stochastic partial differential equations (aspects of stochastic analysis)
%60G22 Fractional processes, including fractional Brownian motion
%45A05 Linear integral equations
%93B52 Feedback control
%49N15 Duality theory (optimization)
%45B05 Fredholm integral equations
%34A08 Fractional ordinary differential equations and fractional differential inclusions
%26A33 Fractional derivatives and integrals

%% Contents

\tableofcontents

%%%%%%%%%%%%%%%%%%%%%%%%%%%%%%%%%%
%%%%%%%%%%%%%%%%%%%%%%%%%%%%%%%%%%
\section{Introduction}\label{intro}
%%%%%%%%%%%%%%%%%%%%%%%%%%%%%%%%%%
%%%%%%%%%%%%%%%%%%%%%%%%%%%%%%%%%%

Let $W$ be a one-dimensional Brownian motion on a complete probability space $(\Omega,\cF,\bP)$, and let $\bF=(\cF_t)_{t\geq0}$ be the augmentation of the filtration generated by $W$. Fix $T>0$. Let $\cU$ be a set of progressively measurable processes $u=(u_t)_{t\in[0,T]}$ taking values in a topological space $U$. For each $u\in\cU$, consider the following ($n$-dimensional) \emph{controlled stochastic Volterra equation} (SVE):
\begin{equation}\label{intro_eq_cSVE}
	X^u_t=\xi_t+\int^t_0K_b(t-s)b(s,u_s,X^u_s)\,\diff s+\int^t_0K_\sigma(t-s)\sigma(s,u_s,X^u_s)\,\diff W_s,\ t\in[0,T].
\end{equation}
Here, $b,\sigma:\Omega\times[0,T]\times U\times\bR^n\to\bR^n$ are progressively measurable maps, $K_b,K_\sigma:(0,T)\to\bR^{n\times n}$ are matrix valued deterministic functions called kernels, and $\xi:\Omega\times[0,T]\to\bR^n$ is a given adapted process which represents the forcing term. Each $u\in\cU$ is referred to as a control process, and the solution $X^u$ of \eqref{intro_eq_cSVE} is referred to as the state process associated with $u$. Consider a cost functional of the form
\begin{equation}\label{intro_eq_cost}
	J(u):=\bE\left[h(X^u_T)+\int^T_0f(t,u_t,X^u_t)\,\diff t\right],
\end{equation}
where $h:\Omega\times\bR^n\to\bR$ and $f:\Omega\times[0,T]\times U\times\bR^n\to\bR$ stand for the terminal and running costs, respectively. We are concerned with the minimization problem for $J(u)$ over all control processes $u\in\cU$. Specifically, we are interested in the characterization of the (open-loop) optimal control $\hat{u}\in\cU$, which is a control process such that $J(\hat{u})=\inf_{u\in\cU}J(u)$.

If $K_b(t)=K_\sigma(t)=I_{n\times n}$ and if $\xi_t=\xi_0$ for any $t\in[0,T]$, the controlled SVE \eqref{intro_eq_cSVE} reduces to the following standard controlled stochastic differential equation (SDE):
\begin{equation}\label{intro_eq_cSDE}
	\diff X^u_t=b(t,u_t,X^u_t)\,\diff t+\sigma(t,u_t,X^u_t)\,\diff W_t,\ \ t\in[0,T],\ \ X^u_0=\xi_0.
\end{equation}
More generally, the SVE \eqref{intro_eq_cSVE} includes a class of (time-) fractional SDEs of the form
\begin{equation}\label{intro_eq_cFSDE}
	\partial^\beta_t(X^u_t-\xi_t)=b(t,u_t,X^u_t)+\partial^\gamma_t\int^t_0\sigma(s,u_s,X^u_s)\,\diff W_s,\ \ t\in[0,T],\ \ X^u_0=\xi_0.
\end{equation}
Here, for $\alpha\in(0,1)$, $\partial^\alpha_t$ is the Caputo fractional derivative of order $\alpha$ and is defined by
\begin{equation*}
	\partial^\alpha_tf(t):=\frac{1}{\Gamma(1-\alpha)}\frac{\diff}{\diff t}\int^t_0(t-s)^{-\alpha}(f(s)-f(0))\,\diff s
\end{equation*}
for suitable function $f$ such that the integral in the right-hand side is (weakly) differentiable in $t$, where $\Gamma(\lambda):=\int^\infty_0t^{\lambda-1}e^{-t}\,\diff t$, $\lambda>0$, is the Gamma function. It is known that (see \cite{ChKiKi15,LiRoSi18,LoRo19}), when the parameters $\beta,\gamma\in(0,1)$ satisfy $\beta-\gamma>-\frac{1}{2}$, the fractional SDE \eqref{intro_eq_cFSDE} is equivalent to
\begin{equation*}
	X^u_t=\xi_t+\frac{1}{\Gamma(\beta)}\int^t_0(t-s)^{\beta-1}b(s,u_s,X^u_s)\,\diff s+\frac{1}{\Gamma(1+\beta-\gamma)}\int^t_0(t-s)^{\beta-\gamma}\sigma(s,u_s,X^u_s)\,\diff W_s,\ \ t\in[0,T].
\end{equation*}
This equation is a special case of the controlled SVE \eqref{intro_eq_cSVE}, where the kernels $K_b$ and $K_\sigma$ are of the form $K_b(t)=\frac{1}{\Gamma(\beta)}t^{\beta-1}I_{n\times n}$ and $K_\sigma(t)=\frac{1}{\Gamma(1+\beta-\gamma)}t^{\beta-\gamma}I_{n\times n}$, respectively. The fractional SDE captures the memory effect and heredity, and is useful to describe anomalous diffusions exhibiting subdiffusive behavior \cite{ChKiKi15,LiRoSi18,LoRo19}. Fractional calculus has attracted lots of attentions in several fields of applied mathematics such as physics, chemistry, engineering, and mathematical finance. Fractional optimal control problems have been extensively studied in the deterministic setting; see for example \cite{Ag04,BeBo20,Go20,JePe09,Ka14,LiYo20} and references cited therein. It is important to consider the stochastic version of the fractional optimal control when the effect of randomness to the fractional controlled system is crucial. This motivates us to investigate the \emph{stochastic fractional control} or more generally the \emph{stochastic Volterra control with singular kernels}; we say that the kernels $K_b$ and $K_\sigma$ in the controlled SVE \eqref{intro_eq_cSVE} are singular if they are unbounded on a neighbourhood of zero. Notice that the weakest conditions for the singular kernels which make the Lebesgue and stochastic integrals in \eqref{intro_eq_cSVE} well-defined are that $K_b\in L^1(0,T;\bR^{n\times n})$ and $K_\sigma\in L^2(0,T;\bR^{n\times n})$. Typical examples of such singular kernels are the \emph{fractional kernels} of the form $K_b(t)=\frac{1}{\Gamma(\beta_b)}t^{\beta_b-1}I_{n\times n}$ and $K_\sigma(t)=\frac{1}{\Gamma(\beta_\sigma)}t^{\beta_\sigma-1}I_{n\times n}$ with $\beta_b\in(0,1)$ and $\beta_\sigma\in(\frac{1}{2},1)$, which appear in the integral form of the fractional SDE \eqref{intro_eq_cFSDE}.

In this paper, we establish a \emph{global maximum principle} for the stochastic Volterra control problem \eqref{intro_eq_cSVE}--\eqref{intro_eq_cost} with singular kernels and general (not necessarily convex) control domain $U$. The maximum principle is one of the most important principles in (stochastic) control theory. The global/local maximum principle provides a necessary condition for optimality in terms of the global/local maximization of the associated (risk adjusted) Hamiltonian. Now let us briefly review the well-established theory of the maximum principle for the control problem of SDE \eqref{intro_eq_cSDE}--\eqref{intro_eq_cost}, which corresponds to \eqref{intro_eq_cSVE} with $K_b(t)=K_\sigma(t)=I_{n\times n}$ and $\xi_t=\xi_0$. In this SDE setting, Peng \cite{Pe90} established a ``global'' maximum principle, which states that, under suitable assumptions, every optimal control $\hat{u}\in\cU$ of the problem \eqref{intro_eq_cSDE}--\eqref{intro_eq_cost} must satisfy the following variational inequality:
\begin{equation}\label{intro_gMP-SDE}
\begin{split}
	&H(t,\hat{u}_t,\hat{X}_t,\hat{p}_t,\hat{q}_t)-H(t,v,\hat{X}_t,\hat{p}_t,\hat{q}_t)-\frac{1}{2}\big\langle\hat{P}_t\big\{\sigma(t,\hat{u}_t,\hat{X}_t)-\sigma(t,v,\hat{X}_t)\big\},\sigma(t,\hat{u}_t,\hat{X}_t)-\sigma(t,v,\hat{X}_t)\big\rangle\geq0\\
	&\hspace{6cm}\text{for all $v\in U$ a.s.\ for a.e.\ $t\in[0,T]$.}
\end{split}
\end{equation}
Here, $H:\Omega\times[0,T]\times U\times\bR^n\times\bR^n\times\bR^n\to\bR$ is the \emph{Hamiltonian} defined by
\begin{equation}\label{intro_eq_Hamiltonian}
	H(t,u,x,p,q):=\langle p,b(t,u,x)\rangle+\langle q,\sigma(t,u,x)\rangle-f(t,u,x),
\end{equation}
$\hat{X}=X^{\hat{u}}$ is the corresponding (optimal) state process, $\hat{p}$ and $\hat{q}$ are the corresponding \emph{first order adjoint processes}, and $\hat{P}$, together with an auxiliary process $\hat{Q}$, are the corresponding \emph{second order adjoint processes}. The pair $(\hat{p},\hat{q})$ is defined as the adapted solution of the following backward stochastic differential equation (BSDE) on $\bR^n$:
\begin{equation}\label{intro_adeq1-SDE}
	\begin{dcases}
	\diff\hat{p}_t=-H_x(t,\hat{u}_t,\hat{X}_t,\hat{p}_t,\hat{q}_t)^\top\,\diff t+\hat{q}_t\,\diff W_t,\ \ t\in[0,T],\\
	\hat{p}_T=-h_x(\hat{X}_T)^\top.
	\end{dcases}
\end{equation}
The pair $(\hat{P},\hat{Q})$ is defined as the adapted solution of the following BSDE on $\bR^{n\times n}$:
\begin{equation}\label{intro_adeq2-SDE}
	\begin{dcases}
	\diff\hat{P}_t=-\Big\{b_x(t,\hat{u}_t,\hat{X}_t)^\top\hat{P}_t+\hat{P}_tb_x(t,\hat{u}_t,\hat{X}_t)+\sigma_x(t,\hat{u}_t,\hat{X}_t)^\top\hat{P}_t\sigma_x(t,\hat{u}_t,\hat{X}_t)\\
	\hspace{3cm}+\sigma_x(t,\hat{u}_t,\hat{X}_t)^\top\hat{Q}_t+\hat{Q}_t\sigma_x(t,\hat{u}_t,\hat{X}_t)+H_{xx}(t,\hat{u}_t,\hat{X}_t,\hat{p}_t,\hat{q}_t)\Big\}\,\diff t\\
	\hspace{1cm}+\hat{Q}_t\,\diff W_t,\ \ t\in[0,T],\\
	\hat{P}_T=-h_{xx}(\hat{X}_T).
	\end{dcases}
\end{equation}
Here and elsewhere, $\varphi_x$ and $\varphi_{xx}$ denote the derivative and Hessian matrix of a function $\varphi$ with respect to the state variable $x$. For the theory and applications of BSDEs, see the textbook \cite{Zh17}. The BSDEs \eqref{intro_adeq1-SDE} and \eqref{intro_adeq2-SDE} are called the \emph{first order adjoint equation} and \emph{second order adjoint equation}, respectively, associated with the optimal control $\hat{u}$. The first order adjoint processes $(\hat{p},\hat{q})$ can be seen as an analogue of the generalized momenta in the classical mechanics; see \cite[Chapter 5]{YoZh99}. In the variational inequality \eqref{intro_gMP-SDE}, the process $\hat{P}$ plays a role of the risk adjustment for the Hamiltonian maximization, which is related to the second order effect of the control on the diffusion coefficient $\sigma(t,u,x)$ and reflects the controller's risk-averse or risk-seeking attitude; see \cite[Chapter 3]{YoZh99}. In some special cases, the global maximum principle \eqref{intro_gMP-SDE} implies the following consequences.
\begin{itemize}
\item
If the diffusion coefficient $\sigma$ does not depend on the control variable $u$, then the global maximum principle \eqref{intro_gMP-SDE} reduces to the following necessary optimality condition:
\begin{equation*}
	H(t,\hat{u}_t,\hat{X}_t,\hat{p}_t,\hat{q}_t)=\max_{v\in U}H(t,v,\hat{X}_t,\hat{p}_t,\hat{q}_t)\ \ \text{a.s.\ for a.e.\ $t\in[0,T]$.}
\end{equation*}
\item
If the control domain $U$ is a convex body of a Euclidean space and if $b,\sigma,f$ are continuously differentiable with respect to the control variable $u$, then \eqref{intro_gMP-SDE} implies the following weaker version of necessary optimality condition:
\begin{equation*}
	H_u(t,\hat{u}_t,\hat{X}_t,\hat{p}_t,\hat{q}_t)(v-\hat{u}_t)\leq0\ \ \text{for all $v\in U$ a.s.\ for a.e.\ $t\in[0,T]$,}
\end{equation*}
where $H_u$ denotes the derivative of the Hamiltonian $H$ with respect to the control variable $u$. The above is referred to as the ``local'' maximum principle.
\end{itemize}
In the above two special cases, the contribution of the risk adjustment term $\hat{P}$ disappears, and we do not need to introduce the second order adjoint equation \eqref{intro_adeq2-SDE}. Under the convexity assumption on $U$ as in the second bullet point above, the local maximum principle (which is a weaker necessary optimality condition than the global maximum principle) can be directly obtained via the analysis of the first order perturbation of the system in terms of the \emph{convex variation} of the control, which is a variation of the given (optimal) control $\hat{u}$ of the form $\hat{u}+\ep(v-\hat{u})$ with small parameter $\ep\in(0,1)$ and arbitrary control process $v\in\cU$. However, this technique can not derive the global maximum principle and is not applicable when the control domain $U$ is not convex (such as $U=\{0,1\}$). In order to derive the global maximum principle allowing the control domain to be non-convex, one should consider the perturbation of the system in terms of the \emph{spike variation} of the control, which is of the form $\hat{u}\1_{[0,T]\setminus[\tau,\tau+\ep]}+v\1_{[\tau,\tau+\ep]}$ with arbitrary $\tau\in[0,T)$, $v\in\cU$ and small parameter $\ep\in(0,T-\tau]$. For more detailed theory, interpretations and history of the maximum principle for SDEs, see \cite[Chapters 3 and 5]{YoZh99}.

In the literature of deterministic fractional/Volterra optimal control, global maximum principles were established by several authors, for example in \cite{JePe09,Ka14,LiYo20}. A local maximum principle for (non-convolution type) SVEs was first investigated by Yong \cite{Yo06} and rigorously proved by his subsequent paper \cite{Yo08}. He assumed that the kernels are regular (that is, bounded) and that the control domain $U$ is a convex subset of a Euclidean space. In this setting, utilizing the convex variation technique, he obtained the corresponding first order adjoint equation in form of a (Type-II) backward stochastic Volterra integral equation (BSVIE), and established a local maximum principle. After that, generalizations and variants of \cite{Yo06,Yo08} have been established in \cite{AgOk15,ShWaYo15,ShWaYo13,ShWeXi20} for regular kernel cases and in \cite{Ha21,Ha23,WaZh24} for singular kernel cases, among others. These works rely on the convexity assumption on the control domain $U$ and consider only the local maximum principle by means of the convex variation technique. The analysis of the global maximum principle for the stochastic Volterra control problem is quite limited. To the best of the author's knowledge, there are only two essential works \cite{Wa20,WaYo23} in this direction so far (see also \cite{MeShWaZh25} for the application of the technique established in \cite{Wa20,WaYo23} to control problems of SDEs with delay). In \cite{Wa20,WaYo23}, the authors considered stochastic Volterra control problems with (non-convolution type) regular kernels and general (not necessarily convex) control domain. Wang \cite{Wa20} derived a global maximum principle by means of a first order adjoint process solving a Type-II BSVIE and two implicit operator-valued second order adjoint processes. After that, the second order adjoint equation in form of a system of non-standard BSVIEs was derived by Wang and Yong \cite{WaYo23}. However, their works heavily rely on the regularity of the kernels, and their generalization to the singular kernel case is quite non-trivial. In particular, they are not applicable to the stochastic fractional control problem described by \eqref{intro_eq_cFSDE}--\eqref{intro_eq_cost}.

In this paper, we generalize Peng's global maximum principle \eqref{intro_gMP-SDE} to the stochastic Volterra control problem \eqref{intro_eq_cSVE}--\eqref{intro_eq_cost} with singular kernels and general control domain. We utilize the spike variation technique, together with the technique of \emph{infinite dimensional lifts} of SVEs in the spirit of the author's previous works \cite{Ha24,Ha25}. The method in the present paper is completely different from the aforementioned works on maximum principle for fractional optimal control and stochastic Volterra control problems. We provide a new infinite dimensional perspective on the global maximum principle for the stochastic Volterra control problem, allowing the kernels to be singular. Noteworthy is that we do not rely on the analysis of fractional derivatives or BSVIEs.

The main contributions of this paper are the following three points:
\begin{itemize}
\item[(i)]
We first show novel estimates for the Taylor type expansion of the controlled SVE \eqref{intro_eq_cSVE} with respect to the spike variation, where the convergence rates of the remainder terms are characterized by the singularity of the kernels. See \cref{Taylor_prop_SVE-expansion}.
\item[(ii)]
Under additional structure conditions on the kernels, we derive new types of first and second order adjoint equations in form of \emph{infinite dimensional backward stochastic evolution equations (BSEEs) on weighted $L^2$ spaces}, and obtain a global maximum principle. See \cref{MP}.
\item[(iii)]
We prove the well-posedness of a general class of BSEEs on weighted $L^2$ spaces, which is also new and includes the first and second order adjoint equations as special cases. See \cref{MP-adeq} and \cref{BSEE}.
\end{itemize}
The main result is \cref{MP_theo_MP}, that is summarized in the contribution (ii) above. The contributions (i) and (iii) are secondary to the purpose of establishing the global maximum principle for stochastic Volterra control problems with singular kernels, but they provide important and novel findings by their own rights beyond stochastic Volterra control theory. Here, let us make further explanations on the above contributions.

\underline{On the contribution (i).}
In order to derive the global/local maximum principles, the important first step is to investigate the effect of the variation of the control to the state process. If one considers the convex variation $\hat{u}+\ep(v-\hat{u})$ (assuming that the control domain $U$ is convex), the perturbation $X^{\hat{u}+\ep(v-\hat{u})}$ from the optimal state process $\hat{X}:=X^{\hat{u}}$ can be estimated by the standard variational argument; see \cite{Yo08}. This is related to the regularity in the control variable, but the regularity in time is not crucial. Hence, the singularity of the kernels does not matter for deriving the local maximum principle via the convex variation, and this is the reason why parallel arguments of \cite{Yo08} (for the regular kernel cases) are successfully applicable even in the singular kernel cases in \cite{Ha21,Ha23,WaZh24}. However, this is not the case when one wants to derive the global maximum principle via the spike variation $\hat{u}\1_{[0,T]\setminus[\tau,\tau+\ep]}+v\1_{[\tau,\tau+\ep]}$ under general control domain $U$. Unlike the convex variation case, the regularity in time of the controlled system is crucial in the estimate of the perturbation $X^{\hat{u}{\scriptsize\1}_{[0,T]\setminus[\tau,\tau+\ep]}+v{\scriptsize\1}_{[\tau,\tau+\ep]}}$ from $\hat{X}$. Intuitively speaking, in the deterministic ordinary differential equation case, the estimate of the spike variation is related to the time-regularity of the Lebesgue integral of the form $\int^t_0\{\cdots\}\1_{[\tau,\tau+\ep]}(s)\,\diff s$. This is of order $O(\ep)$, and hence the first order variation is sufficient to obtain the order $o(\ep)$ for the remainder term. In the SDE case (with controlled diffusion), in addition to the Lebesgue integral term, we need to estimate the time-regularity of the stochastic integral of the form $\int^t_0\{\cdots\}\1_{[\tau,\tau+\ep]}(s)\,\diff W_s$. The latter is of order $O(\ep^{1/2})$ in $L^p$, and this is the reason why the second order variation is necessary (and sufficient) in this SDE case to obtain the order $o(\ep)$ for the remainder term; see \cite{Pe90} and \cite[Chapter 3]{YoZh99}. Analogously, in the case of stochastic Volterra control problem, the time-regularity of the Lebesgue Volterra integral $\int^t_0K_b(t-s)\{\cdots\}\1_{[\tau,\tau+\ep]}(s)\,\diff s$ and the stochastic Volterra integral $\int^t_0K_\sigma(t-s)\{\cdots\}\1_{[\tau,\tau+\ep]}(s)\,\diff W_s$ are crucial. On the one hand, if the kernels $K_b$ and $K_\sigma$ are sufficiently regular, it turns out that these Volterra type integrals have the same orders as that without kernels; see \cite{Wa20,WaYo23}. On the other hand, if the kernels are singular, the Volterra integrals have lower time-regularity than the standard Lebesgue and stochastic integrals. Consequently, variational estimates with respect to the spike variation are quite sensitive in the singular kernel cases. With these observations in mind, we provide variational estimates in terms of the singularity of the kernels $K_b$ and $K_\sigma$. In particular, it turns out that the remainder term by the second order expansion becomes of order $o(\ep)$ if $K_b\in L^{3/2}(0,T;\bR^{n\times n})$ and $K_\sigma\in L^6(0,T;\bR^{n\times n})$; see \cref{Taylor_prop_SVE-expansion}. This is the first important contribution of the present paper and is crucial for deriving the global maximum principle.

\underline{On the contribution (ii).}
By the Taylor type expansion in (i), one gets a first order expansion of the cost functional in form of $J(\hat{u}\1_{[0,T]\setminus[\tau,\tau+\ep]}+v\1_{[\tau,\tau+\ep]})=J(\hat{u})+J^{1,2,v,\tau,\ep}+o(\ep)$, where $J^{1,2,v,\tau,\ep}$ is defined as a linear-quadratic functional of the solutions $X^{1,v,\tau,\ep}$ and $X^{2,v,\tau,\ep}$ to the first and second order variational (linear) SVEs appearing in the expansion. From this, we immediately obtain a necessary condition $\liminf_{\ep\downarrow0}\frac{1}{\ep}J^{1,2,v,\tau,\ep}\geq0$ for $\hat{u}$ to be optimal. However, this is still implicit and is not easily applicable, since the term $J^{1,2,v,\tau,\ep}$ involves $X^{1,v,\tau,\ep}$ and $X^{2,v,\tau,\ep}$, which implicitly depend on $v$. The next important step to derive the global maximum principle is to introduce precise forms of adjoint equations that enable us to get rid of $X^{1,v,\tau,\ep}$ and $X^{2,v,\tau,\ep}$ from the expression of $J^{1,2,v,\tau,\ep}$ via a kind of duality formula. In the SDE setting, $X^{1,v,\tau,\ep}$ and $X^{2,v,\tau,\ep}$ are It\^{o} processes solving variational SDEs, and hence we can use powerful tools of It\^{o}'s calculus and BSDE theory as demonstrated in \cite{Pe90}. However, this is not the case in the SVE setting. The biggest difficulty in the Volterra case is that the variational processes $X^{1,v,\tau,\ep}$ and $X^{2,v,\tau,\ep}$, solutions to the corresponding variational SVEs, are no longer semimartingales. Hence, we can not use It\^{o}'s calculus or BSDE theory directly. In order to overcome this difficulty, our idea is to utilize the framework of infinite dimensional lifts of SVEs investigated in \cite{Ha24,Ha25}. This makes us possible to express $X^{1,v,\tau,\ep}$ and $X^{2,v,\tau,\ep}$ as super positions of infinitely many semimartingales. By using this idea, we can successfully calculate the term $J^{1,2,v,\tau,\ep}$ and obtain first and second order adjoint equations in the infinite dimensional framework.

Let us briefly summarize the main result (\cref{MP_theo_MP}) related to the contribution (ii). We assume that the kernels $K_b$ and $K_\sigma$ are of the following forms:
\begin{equation}\label{intro_eq_kernel}
	K_b(t)=\int_{\bR_+}e^{-\theta t}M_b(\theta)\dmu\ \ \text{and}\ \ K_\sigma(t)=\int_{\bR_+}e^{-\theta t}M_\sigma(\theta)\dmu,\ \ t\in(0,T),
\end{equation}
for some Borel measure $\mu$ on $\bR_+=[0,\infty)$ and matrix-valued measurable maps $M_b,M_\sigma:\bR_+\to\bR^{n\times n}$ satisfying suitable integrability conditions (see \cref{MP_assum_kernel} for more details). The fractional kernel is an important and typical example satisfying the above structure condition. Under this setting (with suitable conditions), we establish the following global maximum principle: every optimal control $\hat{u}\in\cU$ of the stochastic Volterra control problem \eqref{intro_eq_cSVE}--\eqref{intro_eq_cost} must satisfy the following variational inequality:
\begin{align*}
	&H\big(t,\hat{u}_t,\hat{X}_t,\mu[M_b^\top\hat{p}_t],\mu[M_\sigma^\top\hat{q}_t]\big)-H\big(t,v,\hat{X}_t,\mu[M_b^\top\hat{p}_t],\mu[M_\sigma^\top\hat{q}_t]\big)\\
	&-\frac{1}{2}\Big\langle\mu^{\otimes2}[M_\sigma^\top\hat{P}_tM_\sigma]\big\{\sigma(t,\hat{u}_t,\hat{X}_t)-\sigma(t,v,\hat{X}_t)\big\},\sigma(t,\hat{u}_t,\hat{X}_t)-\sigma(t,v,\hat{X}_t)\Big\rangle\geq0\\
	&\hspace{4cm}\text{for all $v\in U$ a.s.\ for a.e.\ $t\in[0,T]$}.
\end{align*}
Here, $H$ is the Hamiltonian defined by \eqref{intro_eq_Hamiltonian}, which is exactly the same as the one in the classical SDE case. $\hat{X}=X^{\hat{u}}$ is the corresponding optimal state process, that is the solution of the controlled SVE \eqref{intro_eq_cSVE} associated with the optimal control $\hat{u}$. The terms $\mu[M_b^\top\hat{p}_t]$ and $\mu[M_\sigma^\top\hat{q}_t]$ are regarded as the first order adjoint processes, which are defined by
\begin{equation*}
	\mu[M_b^\top\hat{p}_t]:=\int_{\bR_+}M_b(\theta)^\top\hat{p}_t(\theta)\dmu\ \ \text{and}\ \ \mu[M_\sigma^\top\hat{q}_t]:=\int_{\bR_+}M_\sigma(\theta)^\top\hat{q}_t(\theta)\dmu,
\end{equation*}
and the pair of random fields $\hat{p},\hat{q}:\Omega\times[0,T]\times\bR_+\to\bR^n$ is defined as the solution of the following first order adjoint equation:
\begin{equation}\label{intro_eq_adeq1}
	\begin{dcases}
	\diff\hat{p}_t(\theta)=\theta\hat{p}_t(\theta)\,\diff t-H_x\big(t,\hat{u}_t,\hat{X}_t,\mu[M_b^\top\hat{p}_t],\mu[M_\sigma^\top\hat{q}_t]\big)^\top\,\diff t+\hat{q}_t(\theta)\,\diff W_t,\ \ \theta\in\bR_+,\ t\in[0,T],\\
	\hat{p}_T(\theta)=-h_x(\hat{X}_T)^\top,\ \ \theta\in\bR_+.
	\end{dcases}
\end{equation}
The term $\mu^{\otimes2}[M_\sigma^\top\hat{P}_tM_\sigma]$ is regarded as the second order adjoint process or the risk adjustment term, which is defined by
\begin{equation*}
	\mu^{\otimes2}[M_\sigma^\top\hat{P}_tM_\sigma]:=\int_{\bR_+}\int_{\bR_+}M_\sigma(\theta_1)^\top\hat{P}_t(\theta_1,\theta_2)M_\sigma(\theta_2)\dmumu,
\end{equation*}
and the pair of a random field $\hat{P}:\Omega\times[0,T]\times\bR_+^2\to\bR^{n\times n}$ and an auxiliary random field $\hat{Q}:\Omega\times[0,T]\times\bR_+^2\to\bR^{n\times n}$ is defined as the solution of the following second order adjoint equation:
\begin{equation}\label{intro_eq_adeq2}
	\begin{dcases}
	\diff\hat{P}_t(\theta_1,\theta_2)=(\theta_1+\theta_2)\hat{P}_t(\theta_1,\theta_2)\,\diff t\\
	\hspace{2cm}-\Big\{b_x(t,\hat{u}_t,\hat{X}_t)^\top\mu[M_b^\top\hat{P}_t(\cdot,\theta_2)]+\mu[\hat{P}_t(\theta_1,\cdot)M_b]b_x(t,\hat{u}_t,\hat{X}_t)\\
	\hspace{3.5cm}+\sigma_x(t,\hat{u}_t,\hat{X}_t)^\top\mu^{\otimes2}[M_\sigma^\top\hat{P}_tM_\sigma]\sigma_x(t,\hat{u}_t,\hat{X}_t)\\
	\hspace{3.5cm}+\sigma_x(t,\hat{u}_t,\hat{X}_t)^\top\mu[M_\sigma^\top\hat{Q}_t(\cdot,\theta_2)]+\mu[\hat{Q}_t(\theta_1,\cdot)M_\sigma]\sigma_x(t,\hat{u}_t,\hat{X}_t)\\
	\hspace{3.5cm}+H_{xx}\big(t,\hat{u}_t,\hat{X}_t,\mu[M_b^\top\hat{p}_t],\mu[M_\sigma^\top\hat{q}_t]\big)\Big\}\,\diff t\\
	\hspace{2cm}+\hat{Q}_t(\theta_1,\theta_2)\,\diff W_t,\ \ (\theta_1,\theta_2)\in\bR_+^2,\ t\in[0,T],\\
	\hat{P}_T(\theta_1,\theta_2)=-h_{xx}(\hat{X}_T),\ \ (\theta_1,\theta_2)\in\bR_+^2.
	\end{dcases}
\end{equation}
The first and second order adjoint equations \eqref{intro_eq_adeq1} and \eqref{intro_eq_adeq2} above are systems of infinitely many finite dimensional BSDEs, which are parametrized by $\theta$ and $(\theta_1,\theta_2)$, respectively, and coupled via the integration with respect to the measure $\mu$. In another perspective, these are formulated as \emph{infinite dimensional BSEEs on weighted $L^2$ spaces}. The random fields $\hat{p}$ and $\hat{q}$ can be regarded as the \emph{lifted first order adjoint processes}; the first order adjoint processes $\mu[M_b^\top\hat{p}_t]$ and $\mu[M_\sigma^\top\hat{q}_t]$ are realised as their super positions reflecting the first order effects of the kernels $K_b$ and $K_\sigma$. Similarly, the random field $\hat{P}$ can be regarded as the \emph{lifted risk adjustment term}; the risk adjustment term $\mu^{\otimes2}[M_\sigma^\top\hat{P}_tM_\sigma]$ is realized as its superposition reflecting the second order effect of the kernel $K_\sigma$ on the diffusion. Notice the simple analogy of our global maximum principle to the classical one in the SDE setting. In the SDE setting, which corresponds to \eqref{intro_eq_cSVE} with $K_b(t)=K_\sigma(t)=I_{n\times n}$ and $\xi_t=\xi_0$, we can take $\mu=\delta_0$ and $M_b(\theta)=M_\sigma(\theta)=I_{n\times n}$. Then, the first and second order adjoint equations \eqref{intro_eq_adeq1} and \eqref{intro_eq_adeq2} with $\theta=0$ and $(\theta_1,\theta_2)=(0,0)$ reduce to the counterparts \eqref{intro_adeq1-SDE} and \eqref{intro_adeq2-SDE}, respectively, and we immediately recover the classical global maximum principle \eqref{intro_gMP-SDE} in the SDE setting. The adjoint equations \eqref{intro_eq_adeq1} and \eqref{intro_eq_adeq2} are much simpler than that of \cite{WaYo23} obtained in the (non-convolution type) regular kernel cases. The second order adjoint equation derived in \cite{WaYo23} is a quite complicated coupled system of BSVIEs, which consists of 4 equations involving non-standard Volterra structures. These BSVIE-type adjoint equations are actually not straightforward generalizations of the counterparts \eqref{intro_adeq1-SDE} and \eqref{intro_adeq2-SDE} in the SDE setting, and some technical transforms are necessary to recover the classical result; see \cite[Section 6.2]{WaYo23}. Compared to \cite{WaYo23}, our adjoint equations are tractable thanks to the semimartingale structure behind them and are natural generalizations of the counterparts in the SDE setting. As a price, in our study, the structure condition \eqref{intro_eq_kernel} is crucial. In \cref{app-BSVIE}, we investigate relationships between our BSEE-type adjoint equations and the BSVIE-type adjoint equations obtained in \cite{WaYo23}.

The idea of lifting SVEs to infinite dimensional Markovian systems originates from the work of Carmona and Coutin \cite{CaCo98}, where the authors study the Markovian lift of Riemann--Liouville fractional Brownian motion in terms of an infinite dimensional Ornstein--Uhlenbeck type process. Since then, general theory on infinite dimensional lifts of SVEs in the spirit of \cite{CaCo98} has been studied in \cite{AbiJaMiPh21,CuTe19,CuTe20,Ha24,Ha25,HaSt19} in different frameworks. Especially among them, the approach established in the author's previous works \cite{Ha24,Ha25} makes us possible to study the infinite dimensional lifts of SVEs in weighted $L^2$ spaces, which are tractable thanks to the Hilbert space structure. In the present paper, we adopt (and slightly generalize) the idea of \cite{Ha24,Ha25}. It is worth to remark that Abi Jaber, Miller and Pham \cite{AbiJaMiPh21} applied the technique of infinite dimensional lifts in the $L^1$ space to the study of linear-quadratic stochastic Volterra control problem without terminal cost. They characterized the optimal control by means of an infinite dimensional operator-valued Riccati equation (whose well-posedness was shown in their subsequent work \cite{AbiJaMiPh21+}), which has a similar structure to the second order adjoint equation \eqref{intro_eq_adeq2} derived in the present paper.

\underline{On the contribution (iii).}
As mentioned above, the adjoint equations \eqref{intro_eq_adeq1} and \eqref{intro_eq_adeq2} can be formulated as infinite dimensional BSEEs on weighted $L^2$ spaces. Actually, if the kernels are regular, it turns out that the corresponding adjoint equations reduce to ``standard'' BSEEs on Gelfand triplets of Hilbert spaces, and the well-established result in \cite{Pe92} is applicable; see also \cite{HuPe91,HuTa18,LiZh20,MeSh13,Zh92} for related works on infinite dimensional BSEEs on Gelfand triplets. However, a difficulty appears again in the singular kernel case, where the adjoint equations \eqref{intro_eq_adeq1} and \eqref{intro_eq_adeq2} are beyond the existing framework of BSEEs mentioned above; see \cref{BSEE_rem_Gelfand}. In \cref{MP-adeq} and \cref{BSEE}, we formulate a new class of BSEEs on weighted $L^2$ spaces, which is beyond the Gelfand triplet framework and includes the adjoint equations \eqref{intro_eq_adeq1} and \eqref{intro_eq_adeq2} in the singular kernel cases. We show in \cref{BSEE_theo_BSEE} its well-posedness. This result provides an important new finding in theory of BSDEs and BSEEs.

The rest of this paper is organized as follows. In \cref{pre}, we summarize notations we use throughout this paper, provide preliminary results, and formulate the stochastic Volterra control problem \eqref{intro_eq_cSVE}--\eqref{intro_eq_cost} giving precise assumptions. \cref{Taylor} corresponds to the contribution (i) mentioned above, where we show the Taylor type expansion of the controlled SVE \eqref{intro_eq_cSVE} in terms of the spike variation. \cref{MP} corresponds to the contribution (ii), which is the main part of this paper. We introduce the framework of infinite dimensional lifts of variational SVEs following \cite{Ha24} in \cref{MP-lift}, heuristically explain how to derive the adjoint equations \eqref{intro_eq_adeq1} and \eqref{intro_eq_adeq2} in \cref{MP-heuristic}, formulate the new class of BSEEs on weighted $L^2$-spaces including the adjoint equations in \cref{MP-adeq}, and rigorously prove the global maximum principle in \cref{MP-main}. \cref{BSEE} corresponds to the contribution (iii), where we prove well-posedness of general BSEEs on weighted $L^2$ spaces. In \cref{app-measurability}, we make some remarks on measurability issues arising in the study of BSEEs. In \cref{app-BSVIE}, we show some relationships between our BSEE-type adjoint equations and the BSVIE-type adjoint equations obtained in \cite{WaYo23}.

%%%%%%%%%%%%%%%%%%%%%%%%%%%%%%%%%%
%%%%%%%%%%%%%%%%%%%%%%%%%%%%%%%%%%
\section{Notations, preliminaries and problem formulation}\label{pre}
%%%%%%%%%%%%%%%%%%%%%%%%%%%%%%%%%%
%%%%%%%%%%%%%%%%%%%%%%%%%%%%%%%%%%

%%%%%%
\subsection{Notations}
%%%%%%

Throughout this paper, we fix a complete probability space $(\Omega,\cF,\bP)$ and a one-dimensional Brownian motion $W$ on it; the one-dimensionality of Brownian motion $W$ is just for simplicity of notations, and all results in this paper can be easily generalized to the multi-dimensional case. We denote by $\bF=(\cF_t)_{t\geq0}$ the augmentation of the filtration generated by $W$. We also fix $T>0$ and $n\in\bN$, that stand for the terminal time and the dimension of the state process, respectively.

The Euclidean norm and the standard inner product on a Euclidean space are denoted by $|\cdot|$ and $\langle\cdot,\cdot\rangle$, respectively. The expectation operator with respect to $\bP$ is denoted by $\bE[\cdot]$, and the conditional expectation operator associated with $\cF_t$ is denoted by $\bE_t[\cdot]$ for each $t\geq0$. For each set $A$, $\1_A$ denotes the indicator function. Given a measurable space $(A,\cA)$, we say that a map $\varphi$ from $\Omega\times[0,T]\times A$ to another measurable space is progressively measurable if for any $t\in[0,T]$, the restriction of $\varphi$ to $\Omega\times[0,t]\times A$ is $\cF_t\otimes\cB([0,t])\otimes\cA$-measurable.

For each sufficiently smooth function $\varphi:\bR^n\to\bR$, $\varphi_x:\bR^n\to\bR^{1\times n}$ denotes the derivative, and $\varphi_{xx}:\bR^n\to\bR^{n\times n}$ denotes the Hessian matrix. For each $\varphi=(\varphi^1,\dots,\varphi^m)^\top:\bR^n\to\bR^m$ with sufficiently smooth $\varphi^i:\bR^n\to\bR$, we denote
\begin{equation}\label{notation_eq_derivative}
	\varphi_x:=
	\begin{pmatrix}
	\varphi^1_x\\
	\vdots\\
	\varphi^m_x
	\end{pmatrix}:\bR^n\to\bR^{m\times n}\ \ \text{and}\ \ 
	\varphi_{xx}:=
	\begin{pmatrix}
	\varphi^1_{xx}\\
	\vdots\\
	\varphi^m_{xx}
	\end{pmatrix}:\bR^n\to\bR^{m\times(n\times n)}=\cL(\bR^{n\times n};\bR^m).
\end{equation}
For each $M=\begin{pmatrix}M^1\\\vdots\\M^\ell\end{pmatrix}\in\bR^{\ell\times(m\times n)}=\cL(\bR^{m\times n};\bR^\ell)$ with $M^k\in\bR^{m\times n}$, $k=1,\dots,\ell$, and each vector $x\in\bR^n$, $y\in\bR^m$ and $z\in\bR^\ell$, we write
\begin{equation}\label{notation_eq_matrix}
	\langle Mx,y\rangle:=
	\begin{pmatrix}
	\langle M^1x,y\rangle\\
	\vdots\\
	\langle M^\ell x,y\rangle
	\end{pmatrix}
	\in\bR^\ell\ \ \text{and}\ \ \langle z,M\rangle:=\sum^\ell_{k=1}z^kM^k\in\bR^{m\times n}.
\end{equation}
Notice that $\langle z,\langle Mx,y\rangle\rangle=\langle \langle z,M\rangle x,y\rangle$.

Let $E$ be a Euclidean space. For each $p\in[1,\infty]$ and two real numbers $s<t$, $(L^p(s,t;E),\|\cdot\|_{L^p(s,t)})$ denotes the standard Lebesgue space of $L^p$-maps from $(0,T)$ to $E$. Also, for each $t\geq0$, $(L^p_{\cF_t}(\Omega;E),\|\cdot\|_{L^p_{\cF_t}(\Omega)})$ denotes the $L^p$ space (with respect to $\bP$) of equivalence classes of $\cF_t$-measurable maps from $\Omega$ to $E$. We consider the following classes of stochastic processes:
\begin{itemize}
\item
$L^{p,\infty}_\bF(0,T;E)$: the set of (equivalence classes of) progressively measurable maps $\xi:\Omega\times[0,T]\to E$ such that
\begin{equation*}
	\|\xi\|_{L^{p,\infty}_\bF(0,T)}:=\esssup_{t\in[0,T]}\bE\big[|\xi_t|^p\big]^{1/p}<\infty.
\end{equation*}
\item
$C^p_\bF(0,T;E)$: the set of (equivalence classes of) progressively measurable maps $\xi:\Omega\times[0,T]\to E$ such that $\xi_t\in L^p_{\cF_t}(\Omega;E)$ for any $t\in[0,T]$ and that the map $t\mapsto \xi_t$ is continuous with respect to the norm in $L^p_{\cF_T}(\Omega;E)$. Denote
\begin{equation*}
	\|\xi\|_{C^p_\bF(0,T)}:=\sup_{t\in[0,T]}\bE\big[|\xi_t|^p\big]^{1/p}.
\end{equation*}
\end{itemize}

Lastly, we introduce the following unusual notation. Let $q\in[1,\infty]$ and $\ep\in(0,T]$. For each $K\in L^q(0,T;\bR^{n\times n})$, we set
\begin{equation*}
	\|K\|_{q,\ep}:=\sup_{\substack{0\leq s<t\leq T\\t-s\leq\ep}}\|K\|_{L^q(s,t)}.
\end{equation*}
Notice that $\|\cdot\|_{q,\ep}$ is a norm on $L^q(0,T;\bR^{n\times n})$ and equivalent to the original norm $\|\cdot\|_{L^q(0,T)}$. Also, notice that $\|\cdot\|_{q,T}=\|\cdot\|_{L^q(0,T)}$ and $\|\cdot\|_{\infty,\ep}=\|\cdot\|_{L^\infty(0,T)}$. If $K$ is in form of $K(t)=k(t)I_{n\times n}$ for some nonnegative and non-increasing scalar function $k:(0,T)\to\bR_+$ (such as the fractional kernel), then $\|K\|_{q,\ep}=\|k\|_{L^q(0,\ep)}$. The convergence rate of $\|K\|_{q,\ep}\to0$ as $\ep\downarrow0$ corresponds to the regularity and singularity of $K$. The norm $\|\cdot\|_{q,\ep}$ plays a crucial role in the variational argument for the controlled SVE \eqref{intro_eq_cSVE} with singular kernels; see \cref{Taylor_prop_SVE-expansion} below.

%% Remark

\begin{rem}\label{intro_rem_K-ep}
The following properties are trivial but important:
\begin{itemize}
\item
If $K\in L^q(0,T;\bR^{n\times n})$ with $q\in[1,\infty)$, then $\|K\|_{q,\ep}\to0$ as $\ep\downarrow0$. This follows from the uniform continuity of the function $t\mapsto\int^t_0|K(r)|^q\,\diff r$ on $[0,T]$.
\item
If $K\in L^{q'}(0,T;\bR^{n\times n})$ with $1\leq q\leq q'\leq \infty$, then $\|K\|_{q,\ep}\leq\|K\|_{q',\ep}\ep^{1/q-1/{q'}}$, where $1/\infty:=0$. This follows from H\"{o}lder's inequality.
\end{itemize}
\end{rem}

%%%%%%
\subsection{Preliminaries}
%%%%%%

The following lemma is related to the time-regularity of Volterra integrals, which plays a crucial role in \cref{Taylor_prop_SVE-expansion}.

%% Lemma

\begin{lemm}\label{pre_lemm_convolution}
\begin{itemize}
\item[(i)]
Let $K\in L^1(0,T;\bR^{n\times n})$ and $p\in[1,\infty)$. Then, the map $b\mapsto(\int^t_0K(t-s)b_s\,\diff s)_{t\in[0,T]}$ defines a bounded linear operator from $L^{p,\infty}_\bF(0,T;\bR^n)$ to $C^p_\bF(0,T;\bR^n)$. Furthermore, for any $b\in L^{p,\infty}_\bF(0,T;\bR^n)$, $\tau\in[0,T)$ and $\ep\in(0,T-\tau]$, it holds that
\begin{equation}\label{pre_eq_convolution-b}
	\sup_{t\in[0,T]}\bE\left[\left|\int^t_0K(t-s)b_s\1_{[\tau,\tau+\ep]}(s)\,\diff s\right|^p\right]^{1/p}\leq\|K\|_{1,\ep}\|b\|_{L^{p,\infty}_\bF(0,T)}.
\end{equation}
\item[(ii)]
Let $K\in L^2(0,T;\bR^{n\times n})$ and $p\in[2,\infty)$. Then, the map $\sigma\mapsto(\int^t_0K(t-s)\sigma_s\,\diff W_s)_{t\in[0,T]}$ defines a bounded linear operator from $L^{p,\infty}_\bF(0,T;\bR^n)$ to $C^p_\bF(0,T;\bR^n)$. Furthermore, for any $\sigma\in L^{p,\infty}_\bF(0,T;\bR^n)$, $\tau\in[0,T)$ and $\ep\in(0,T-\tau]$, it holds that
\begin{equation}\label{pre_eq_convolution-sigma}
	\sup_{t\in[0,T]}\bE\left[\left|\int^t_0K(t-s)\sigma_s\1_{[\tau,\tau+\ep]}(s)\,\diff W_s\right|^p\right]^{1/p}\leq C_p\|K\|_{2,\ep}\|\sigma\|_{L^{p,\infty}_\bF(0,T)},
\end{equation}
where $C_p$ is a constant depending only on $p$.
\end{itemize}
\end{lemm}

%% Proof

\begin{proof}
First, we show the assertion (i). Suppose that $K\in L^1(0,T;\bR^{n\times n})$ and $p\in[1,\infty)$. Let $b\in L^{p,\infty}_\bF(0,T;\bR^n)$ be fixed. Notice that, for each $0\leq t\leq t'\leq T$,
\begin{align*}
	&\bE\left[\left|\int^{t'}_0K(t'-s)b_s\,\diff s-\int^t_0K(t-s)b_s\,\diff s\right|^p\right]^{1/p}\\
	&\leq\bE\left[\left(\int^{t'}_t|K(t'-s)||b_s|\,\diff s\right)^p\right]^{1/p}+\bE\left[\left(\int^t_0|K(t'-s)-K(t-s)||b_s|\,\diff s\right)^p\right]^{1/p}.
\end{align*}
Using Minkowski's integral inequality, one has
\begin{equation*}
	\bE\left[\left(\int^{t'}_t|K(t'-s)||b_s|\,\diff s\right)^p\right]^{1/p}\leq\int^{t'}_t|K(t'-s)|\bE\big[|b_s|^p\big]^{1/p}\,\diff s\leq\int^{t'}_t|K(t'-s)|\,\diff s\,\|b\|_{L^{p,\infty}_\bF(0,T)}.
\end{equation*}
Similarly,
\begin{equation*}
	\bE\left[\left(\int^t_0|K(t'-s)-K(t-s)||b_s|\,\diff s\right)^p\right]^{1/p}\leq\int^t_0|K(t'-s)-K(t-s)|\,\diff s\,\|b\|_{L^{p,\infty}_\bF(0,T)}.
\end{equation*}
Therefore, we get
\begin{align*}
	&\bE\left[\left|\int^{t'}_0K(t'-s)b_s\,\diff s-\int^t_0K(t-s)b_s\,\diff s\right|^p\right]^{1/p}\\
	&\leq\left\{\int^{t'}_t|K(t'-s)|\,\diff s+\int^t_0|K(t'-s)-K(t-s)|\,\diff s\right\}\|b\|_{L^{p,\infty}_\bF(0,T)}.
\end{align*}
Since $K\in L^1(0,T;\bR^{n\times n})$, the last term above tends to zero as $t'-t\to0$. Therefore, the process $(\int^t_0K(t-s)b_s\,\diff s)_{t\in[0,T]}$ is in $C^p_\bF(0,T;\bR^n)$. The linearity of the map $b\mapsto(\int^t_0K(t-s)b_s\,\diff s)_{t\in[0,T]}$ is clear. Furthermore, for any $t\in[0,T]$, $\tau\in[0,T)$ and $\ep\in(0,T-\tau]$, noting the definition of the norm $\|\cdot\|_{1,\ep}$, we have
\begin{equation*}
	\bE\left[\left|\int^t_0K(t-s)b_s\1_{[\tau,\tau+\ep]}(s)\,\diff s\right|^p\right]^{1/p}\leq\int^t_0|K(t-s)|\1_{[\tau,\tau+\ep]}(s)\,\diff s\|b\|_{L^{p,\infty}_\bF(0,T)}\leq\|K\|_{1,\ep}\|b\|_{L^{p,\infty}_\bF(0,T)}.
\end{equation*}
Hence, the estimate \eqref{pre_eq_convolution-b} holds. In particular, considering $\tau=0$ and $\ep=T$, we see that the map $b\mapsto(\int^t_0K(t-s)b_s\,\diff s)_{t\in[0,T]}$ is a bounded linear operator from $L^{p,\infty}_\bF(0,T;\bR^n)$ to $C^p_\bF(0,T;\bR^n)$.

Next, we show the assertion (ii). Suppose that $K\in L^2(0,T;\bR^{n\times n})$ and $p\in[2,\infty)$. Let $\sigma\in L^{p,\infty}_\bF(0,T;\bR^n)$ be fixed. For each $0\leq t\leq t'\leq T$,
\begin{align*}
	&\bE\left[\left|\int^{t'}_0K(t'-s)\sigma_s\,\diff s-\int^t_0K(t-s)\sigma_s\,\diff W_s\right|^p\right]^{1/p}\\
	&\leq\bE\left[\left|\int^{t'}_tK(t'-s)\sigma_s\,\diff W_s\right|^p\right]^{1/p}+\bE\left[\left|\int^t_0(K(t'-s)-K(t-s))\sigma_s\,\diff W_s\right|^p\right]^{1/p}\\
	&\leq C_p\bE\left[\left(\int^{t'}_t|K(t'-s)|^2|\sigma_s|^2\,\diff s\right)^{p/2}\right]^{1/p}+C_p\bE\left[\left(\int^t_0|K(t'-s)-K(t-s)|^2|\sigma_s|^2\,\diff s\right)^{p/2}\right]^{1/p},
\end{align*}
where we used the Burkholder--Davis--Gundy inequality in the last line. Using Minkowski's integral inequality, one has
\begin{align*}
	\bE\left[\left(\int^{t'}_t|K(t'-s)|^2|\sigma_s|^2\,\diff s\right)^{p/2}\right]^{1/p}&\leq\left(\int^{t'}_t|K(t'-s)|^2\bE\big[|\sigma_s|^p\big]^{2/p}\,\diff s\right)^{1/2}\\
	&\leq\left(\int^{t'}_t|K(t'-s)|^2\,\diff s\right)^{1/2}\|\sigma\|_{L^{p,\infty}_\bF(0,T)}.
\end{align*}
Similarly,
\begin{equation*}
	\bE\left[\left(\int^t_0|K(t'-s)-K(t-s)|^2|\sigma_s|^2\,\diff s\right)^{p/2}\right]^{1/p}\leq\left(\int^t_0|K(t'-s)-K(t-s)|^2\,\diff s\right)^{1/2}\|\sigma\|_{L^{p,\infty}_\bF(0,T)}.
\end{equation*}
Hence,
\begin{align*}
	&\bE\left[\left|\int^{t'}_0K(t'-s)\sigma_s\,\diff W_s-\int^t_0K(t-s)\sigma_s\,\diff W_s\right|^p\right]^{1/p}\\
	&\leq C_p\left\{\left(\int^{t'}_t|K(t'-s)|^2\,\diff s\right)^{1/2}+\left(\int^t_0|K(t'-s)-K(t-s)|^2\,\diff s\right)^{1/2}\right\}\|\sigma\|_{L^{p,\infty}_\bF(0,T)}.
\end{align*}
Since $K\in L^2(0,T;\bR^{n\times n})$, the last term above tends to zero as $t'-t\to0$. Therefore, the process $(\int^t_0K(t-s)\sigma_s\,\diff W_s)_{t\in[0,T]}$ is in $C^p_\bF(0,T;\bR^n)$. The linearity of the map $\sigma\mapsto(\int^t_0K(t-s)\sigma_s\,\diff W_s)_{t\in[0,T]}$ is clear. Furthermore, for any $t\in[0,T]$, $\tau\in[0,T)$ and $\ep\in(0,T-\tau]$, noting the definition of the norm $\|\cdot\|_{2,\ep}$, we have
\begin{align*}
	\bE\left[\left|\int^t_0K(t-s)\sigma_s\1_{[\tau,\tau+\ep]}(s)\,\diff W_s\right|^p\right]^{1/p}&\leq C_p\bE\left[\left(\int^t_0|K(t-s)|^2|\sigma_s|^2\1_{[\tau,\tau+\ep]}(s)\,\diff s\right)^{p/2}\right]^{1/p}\\
	&\leq C_p\left(\int^t_0|K(t-s)|^2\1_{[\tau,\tau+\ep]}(s)\,\diff s\right)^{1/2}\|\sigma\|_{L^{p,\infty}_\bF(0,T)}\\
	&\leq C_p\|K\|_{2,\ep}\|\sigma\|_{L^{p,\infty}_\bF(0,T)}.
\end{align*}
Hence, the estimate \eqref{pre_eq_convolution-sigma} holds. In particular, considering $\tau=0$ and $\ep=T$, we see that the map $\sigma\mapsto(\int^t_0K(t-s)\sigma_s\,\diff W_s)_{t\in[0,T]}$ is a bounded linear operator from $L^{p,\infty}_\bF(0,T;\bR^n)$ to $C^p_\bF(0,T;\bR^n)$.
\end{proof}

The following is a known result on the well-posedness and moment estimate for SVEs with singular kernels; see for example \cite{Zh10}.

%% Lemma

\begin{lemm}\label{pre_lemm_SVE}
Let $p\in[2,\infty)$. Fix $\xi\in C^p_\bF(0,T;\bR^n)$, $K_b\in L^1(0,T;\bR^{n\times n})$ and $K_\sigma\in L^2(0,T;\bR^{n\times n})$. Let $b,\sigma:\Omega\times[0,T]\times\bR^n\to\bR^n$ be progressively measurable maps. Assume that the processes $(b(t,0))_{t\in[0,T]}$ and $(\sigma(t,0))_{t\in[0,T]}$ are in $L^{p,\infty}_\bF(0,T;\bR^n)$. Assume furthermore that there exists a constant $L>0$ such that
\begin{equation*}
	|b(t,x)-b(t,x')|+|\sigma(t,x)-\sigma(t,x')|\leq L|x-x'|
\end{equation*}
for any $x,x'\in\bR^n$ and $(\omega,t)\in\Omega\times[0,T]$. Then, the SVE
\begin{equation*}
	X_t=\xi_t+\int^t_0K_b(t-s)b(s,X_s)\,\diff s+\int^t_0K_\sigma(t-s)\sigma(s,X_s)\diff W_s,\ t\in[0,T],
\end{equation*}
admits a unique solution $X\in C^p_\bF(0,T;\bR^n)$. Furthermore, there exists a constant $C_p>0$, which depends only on $p,T,K_b,K_\sigma$ and $L$, such that
\begin{equation*}
	\|X\|_{C^p_\bF(0,T)}\leq C_p\Big\{\|\xi\|_{C^p_\bF(0,T)}+\|b(\cdot,0)\|_{L^{p,\infty}_\bF(0,T)}+\|\sigma(\cdot,0)\|_{L^{p,\infty}_\bF(0,T)}\Big\}.
\end{equation*}
\end{lemm}

%% Proof

\begin{proof}
The assertion follows from \cite[Theorem 3.1]{Zh10} applying to $g(t)=\xi_t+\int^t_0K_b(t-s)b(s,0)\,\diff s+\int^t_0K_\sigma(t-s)\sigma(s,0)\,\diff s$, $A(t,s,x)=K_b(t-s)(b(s,x)-b(s,0))$ and $B(t,s,x)=K_\sigma(t-s)(\sigma(s,x)-\sigma(s,0))$ in its notation, together with \cref{pre_lemm_convolution}.
\end{proof}

%%%%%%
\subsection{Problem formulation}
%%%%%%

We consider the optimal control problem described by the controlled SVE \eqref{intro_eq_cSVE} and the cost functional \eqref{intro_eq_cost}. We impose the following assumptions on the coefficients.

%% Assumption

\begin{assum}\label{control_assum_coefficient}
\begin{itemize}
\item
$\xi\in\bigcap_{p\in[2,\infty)}C^p_\bF(0,T;\bR^n)$, $K_b\in L^1(0,T;\bR^{n\times n})$, $K_\sigma\in L^2(0,T;\bR^{n\times n})$.
\item
The control domain $U$ is a separable topological space. We equip $U$ with the Borel $\sigma$-algebra $\cB(U)$.
\item
The maps $b,\sigma:\Omega\times[0,T]\times U\times\bR^n\to\bR^n$ and $f:\Omega\times[0,T]\times U\times\bR^n\to\bR$ are progressively measurable, and $h:\Omega\times\bR^n\to\bR$ is $\cF_T\otimes\cB(\bR^n)$-measurable. Furthermore, assume that the following hold:
\begin{itemize}
\item
For any $(\omega,t,x)\in\Omega\times[0,T]\times\bR^n$, the maps $u\mapsto b(t,u,x),\sigma(t,u,x),f(t,u,x)$ are continuous.
\item
For any $(\omega,t,u)\in\Omega\times[0,T]\times U$, the maps $x\mapsto b(t,u,x),\sigma(t,u,x),f(t,u,x),h(x)$ are twice continuously differentiable.
\item
There exist constants $\kappa,\gamma>0$ and a progressively measurable map $R:\Omega\times[0,T]\times U\to\bR_+$ such that, for any $(\omega,t,u)\in\Omega\times[0,T]\times U$ and any $x,x'\in\bR^n$,
\begin{align*}
	&|\varphi(t,u,0)|\leq R(t,u),\ \ |\varphi_x(t,u,x)|+|\varphi_{xx}(t,u,x)|\leq \kappa,\\
	&|\varphi_{xx}(t,u,x)-\varphi_{xx}(t,u,x')|\leq R(t,u)\big(1\vee|x|\vee|x'|\big)^\gamma|x-x'|,\ \ \text{for $\varphi\in\{b,\sigma\}$,}
\end{align*}
and
\begin{align*}
	&|f(t,u,x)|+|f_x(t,u,x)|+|f_{xx}(t,u,x)|\leq R(t,u)\big(1\vee|x|\big)^\gamma,\\
	&|f_{xx}(t,u,x)-f_{xx}(t,u,x')|\leq R(t,u)\big(1\vee|x|\vee|x'|\big)^\gamma|x-x'|,\\
	&|h(x)|+|h_x(x)|+|h_{xx}(x)|\leq\kappa\big(1\vee|x|\big)^\gamma,\\
	&|h_{xx}(x)-h_{xx}(x')|\leq\kappa\big(1\vee|x|\vee|x'|\big)^\gamma|x-x'|.
\end{align*}
Moreover, for each $u\in U$,
\begin{equation*}
	(R(t,u))_{t\in[0,T]}\in\bigcap_{p\in[2,\infty)}L^{p,\infty}_\bF(0,T;\bR_+).
\end{equation*}
\end{itemize}
\end{itemize}
\end{assum}

%% Remark

\begin{rem}
The separability of the control domain $U$ and the continuity of $b$, $\sigma$ and $f$ with respect to the control variable $u\in U$ are used only in the last step of the proof of the global maximum principle (\cref{MP_theo_MP}).
\end{rem}

Let \cref{control_assum_coefficient} hold. We denote by $\cU$ the set of all $U$-valued progressively measurable processes $u=(u_t)_{t\in[0,T]}$ such that
\begin{equation*}
	R^u\in\bigcap_{p\in[2,\infty)}L^{p,\infty}_\bF(0,T;\bR_+),\ \ \text{where $R^u_t:=R(t,u_t)$, $t\in[0,T]$}.
\end{equation*}
Each element $u\in\cU$ is referred to as a control process. By the assumption, every constant process $u_t=u_0\in U$, $t\in[0,T]$, belongs to $\cU$. Also, for any $u,v\in\cU$ and any progressively measurable set $A\subset\Omega\times[0,T]$, we have $u\1_A+v\1_{(\Omega\times[0,T])\setminus A}\in\cU$ and $R^{u\scriptsize{\1}_A+v\scriptsize{\1}_{(\Omega\times[0,T])\setminus A}}\leq R^u\vee R^v$. By \cref{pre_lemm_SVE}, for any $u\in\cU$, there exists a unique solution $X^u\in \bigcap_{p\in[2,\infty)}C^p_\bF(0,T;\bR^n)$ to the controlled SVE \eqref{intro_eq_cSVE}, which we call the state process corresponding to $u$. Furthermore, the following estimate holds:
\begin{equation}\label{control_eq_estimate}
	\|X^u\|_{C^p_\bF(0,T)}\leq C_p\Big\{\|\xi\|_{C^p_\bF(0,T)}+\|R^u\|_{L^{p,\infty}_\bF(0,T)}\Big\}<\infty\ \ \text{for any $p\in[2,\infty)$,}
\end{equation}
where $C_p>0$ is a constant depending only on $p,T,K_b,K_\sigma$ and $\kappa$. From this, we see that the cost functional $J(u)$ given by \eqref{intro_eq_cost} is well-defined and finite. We say that a control process $\hat{u}\in\cU$ is an (open-loop) \emph{optimal control} of the stochastic Volterra control problem \eqref{intro_eq_cSVE}--\eqref{intro_eq_cost} if it holds that
\begin{equation*}
	J(\hat{u})=\inf_{u\in\cU}J(u).
\end{equation*}
The purpose of this paper is to derive a \emph{global maximum principle} which provides a necessary condition for a given control process $\hat{u}\in\cU$ to be optimal.

%%%%%%%%%%%%%%%%%%%%%%%%%%%%%%%%%%
%%%%%%%%%%%%%%%%%%%%%%%%%%%%%%%%%%
\section{Taylor expansion of the controlled SVE with respect to the spike variation}\label{Taylor}
%%%%%%%%%%%%%%%%%%%%%%%%%%%%%%%%%%
%%%%%%%%%%%%%%%%%%%%%%%%%%%%%%%%%%

In order to derive the global maximum principle, we first investigate the Taylor type expansion for the controlled SVE \eqref{intro_eq_cSVE} with respect to the spike variation. Let \cref{control_assum_coefficient} hold. Fix a control process $\hat{u}\in\cU$, and denote the corresponding state process by $\hat{X}:=X^{\hat{u}}$. The control process $\hat{u}$ is a candidate of optimal controls, but we do not assume the optimality for a while. We denote $\hat{R}:=R^{\hat{u}}$, and
\begin{align*}
	&\hat{\varphi}_x(t):=\varphi_x(t,\hat{u}_t,\hat{X}_t)\ \ \text{and}\ \ \hat{\varphi}_{xx}(t):=\varphi_{xx}(t,\hat{u}_t,\hat{X}_t)\ \ \text{for $\varphi\in\{b,\sigma,f\}$;}\\
	&\hat{h}_x:=h_x(\hat{X}_T)\ \ \text{and}\ \ \hat{h}_{xx}:=h_{xx}(\hat{X}_T).
\end{align*}
Furthermore, given another control process $v\in\cU$, we denote
\begin{align*}
	&\delta\varphi^v(t):=\varphi(t,v_t,\hat{X}_t)-\varphi(t,\hat{u}_t,\hat{X}_t),\\
	&\delta\varphi^v_x(t):=\varphi_x(t,v_t,\hat{X}_t)-\varphi_x(t,\hat{u}_t,\hat{X}_t),\\
	&\delta\varphi^v_{xx}(t):=\varphi_{xx}(t,v_t,\hat{X}_t)-\varphi_{xx}(t,\hat{u}_t,\hat{X}_t),
\end{align*}
for $\varphi\in\{b,\sigma,f\}$. By \cref{control_assum_coefficient} and the estimate \eqref{control_eq_estimate}, we see that
\begin{equation}\label{Taylor_eq_var-coefficient-state}
\begin{split}
	&|\hat{\varphi}_x|+|\hat{\varphi}_{xx}|\leq\kappa,\ \ |\delta\varphi^v_x|+|\delta\varphi^v_{xx}|\leq2\kappa,\\
	&|\delta\varphi^v|\leq R^v+\hat{R}+2\kappa|\hat{X}|\in\bigcap_{p\in[2,\infty)}L^{p,\infty}_\bF(0,T;\bR_+),\ \ \text{for $\varphi\in\{b,\sigma\}$,}
\end{split}
\end{equation}
and
\begin{equation}\label{Taylor_eq_var-coefficient-cost}
\begin{split}
	&|\hat{f}_x|+|\hat{f}_{xx}|\leq\hat{R}\big(1\vee|\hat{X}|\big)^\gamma\in\bigcap_{p\in[2,\infty)}L^{p,\infty}_\bF(0,T;\bR_+),\\
	&|\delta f^v|+|\delta f^v_x|+|\delta f^v_{xx}|\leq\big(R^v+\hat{R}\big)\big(1\vee|\hat{X}|\big)^\gamma\in\bigcap_{p\in[2,\infty)}L^{p,\infty}_\bF(0,T;\bR_+),\\
	&|\hat{h}_x|+|\hat{h}_{xx}|\leq\kappa\big(1\vee|\hat{X}_T|\big)^\gamma\in\bigcap_{p\in[2,\infty)}L^p_{\cF_T}(\Omega;\bR_+).
\end{split}
\end{equation}

For each $v\in\cU$, $\tau\in[0,T)$ and $\ep\in(0,T-\tau]$, define the \emph{spike variation} $u^{v,\tau,\ep}$ of $\hat{u}$ with respect to $(v,\tau,\ep)$ by
\begin{equation}\label{Taylor_eq_spike-variation}
	u^{v,\tau,\ep}_t:=\hat{u}_t\1_{[0,T]\setminus[\tau,\tau+\ep]}(t)+v_t\1_{[\tau,\tau+\ep]}(t),\ \ t\in[0,T].
\end{equation}
Notice that $u^{v,\tau,\ep}\in\cU$. Let $X^{v,\tau,\ep}:=X^{u^{v,\tau,\ep}}$ be the state process corresponding to the control process $u^{v,\tau,\ep}$. We investigate a Taylor type expansion of $X^{v,\tau,\ep}$ around $\hat{X}$ in the spirit of \cite{Pe90}. For such a purpose, we introduce the following \emph{first order variational SVE}:
\begin{equation}\label{Taylor_eq_vareq-SVE1}
\begin{split}
	X^{1,v,\tau,\ep}_t&=\int^t_0K_b(t-s)\Big\{\hat{b}_x(s)X^{1,v,\tau,\ep}_s+\delta b^v(s)\1_{[\tau,\tau+\ep]}(s)\Big\}\,\diff s\\
	&\hspace{0.5cm}+\int^t_0K_\sigma(t-s)\Big\{\hat{\sigma}_x(s)X^{1,v,\tau,\ep}_s+\delta\sigma^v(s)\1_{[\tau,\tau+\ep]}(s)\Big\}\,\diff W_s,\ \ t\in[0,T],
\end{split}
\end{equation}
and the following \emph{second order variational SVE}:
\begin{equation}\label{Taylor_eq_vareq-SVE2}
\begin{split}
	X^{2,v,\tau,\ep}_t&=\int^t_0K_b(t-s)\Big\{\hat{b}_x(s)X^{2,v,\tau,\ep}_s+\frac{1}{2}\big\langle\hat{b}_{xx}(s)X^{1,v,\tau,\ep}_s,X^{1,v,\tau,\ep}_s\big\rangle+\delta b^v_x(s)X^{1,v,\tau,\ep}_s\1_{[\tau,\tau+\ep]}(s)\Big\}\,\diff s\\
	&\hspace{0.5cm}+\int^t_0K_\sigma(t-s)\Big\{\hat{\sigma}_x(s)X^{2,v,\tau,\ep}_s+\frac{1}{2}\big\langle\hat{\sigma}_{xx}(s)X^{1,v,\tau,\ep}_s,X^{1,v,\tau,\ep}_s\big\rangle+\delta\sigma^v_x(s)X^{1,v,\tau,\ep}_s\1_{[\tau,\tau+\ep]}(s)\Big\}\,\diff W_s,\\
	&\hspace{6cm}t\in[0,T].
\end{split}
\end{equation}
Here, recall the notations \eqref{notation_eq_derivative} and \eqref{notation_eq_matrix}. Noting \eqref{Taylor_eq_var-coefficient-state}, by \cref{pre_lemm_SVE}, there exist unique solutions $X^{1,v,\tau,\ep},X^{2,v,\tau,\ep}\in \bigcap_{p\in[2,\infty)}C^p_\bF(0,T;\bR^n)$ to the above SVEs. Also, notice that $X^{1,v,\tau,\ep}_t=X^{2,v,\tau,\ep}_t=0$ a.s.\ for any $t\in[0,\tau]$. We denote $X^{1,2,v,\tau,\ep}:=X^{1,v,\tau,\ep}+X^{2,v,\tau,\ep}$. Then $X^{1,2,v,\tau,\ep}$ solves the following SVE:
\begin{equation}\label{Taylor_eq_vareq-SVE12}
\begin{split}
	X^{1,2,v,\tau,\ep}_t&=\int^t_0K_b(t-s)\Big\{\hat{b}_x(s)X^{1,2,v,\tau,\ep}_s+\frac{1}{2}\big\langle\hat{b}_{xx}(s)X^{1,v,\tau,\ep}_s,X^{1,v,\tau,\ep}_s\big\rangle\\
	&\hspace{4.5cm}+\Big(\delta b^v(s)+\delta b^v_x(s)X^{1,v,\tau,\ep}_s\Big)\1_{[\tau,\tau+\ep]}(s)\Big\}\,\diff s\\
	&\hspace{0.5cm}+\int^t_0K_\sigma(t-s)\Big\{\hat{\sigma}_x(s)X^{1,2,v,\tau,\ep}_s+\frac{1}{2}\big\langle\hat{\sigma}_{xx}(s)X^{1,v,\tau,\ep}_s,X^{1,v,\tau,\ep}_s\big\rangle\\
	&\hspace{5cm}+\Big(\delta\sigma^v(s)+\delta\sigma^v_x(s)X^{1,v,\tau,\ep}_s\Big)\1_{[\tau,\tau+\ep]}(s)\Big\}\,\diff W_s,\ \ t\in[0,T].
\end{split}
\end{equation}
Define
\begin{align*}
	&\delta X^{v,\tau,\ep}:=X^{v,\tau,\ep}-\hat{X},\\
	&\delta X^{1,v,\tau,\ep}:=X^{v,\tau,\ep}-\hat{X}-X^{1,v,\tau,\ep}=\delta X^{v,\tau,\ep}-X^{1,v,\tau,\ep},\\
	&\delta X^{1,2,v,\tau,\ep}:=X^{v,\tau,\ep}-\hat{X}-X^{1,v,\tau,\ep}-X^{2,v,\tau,\ep}=\delta X^{v,\tau,\ep}-X^{1,2,v,\tau,\ep}=\delta X^{1,v,\tau,\ep}-X^{2,v,\tau,\ep}.
\end{align*}
Also, we set
\begin{equation}\label{Taylor_eq_vareq-cost}
\begin{split}
	J^{1,2,v,\tau,\ep}&:=\bE\left[\hat{h}_xX^{1,2,v,\tau,\ep}_T+\frac{1}{2}\big\langle\hat{h}_{xx}X^{1,v,\tau,\ep}_T,X^{1,v,\tau,\ep}_T\big\rangle\right]\\
	&\hspace{1cm}+\bE\left[\int^T_0\Big\{\hat{f}_x(s)X^{1,2,v,\tau,\ep}_s+\frac{1}{2}\big\langle\hat{f}_{xx}(s)X^{1,v,\tau,\ep}_s,X^{1,v,\tau,\ep}_s\big\rangle\Big\}\,\diff s\right]+\bE\left[\int^{\tau+\ep}_\tau\delta f^v(s)\,\diff s\right].
\end{split}
\end{equation}
Noting \eqref{Taylor_eq_var-coefficient-cost} and $X^1,X^{1,2}\in\bigcap_{p\in[2,\infty)}C^p_\bF(0,T;\bR^n)$, the right-hand side above is well-defined and finite. Finally, we set
\begin{equation*}
	\delta J^{1,2,v,\tau,\ep}:=J(u^{v,\tau,\ep})-J(\hat{u})-J^{1,2,v,\tau,\ep}.
\end{equation*}
Under the above setting, the following important proposition holds.

%% Proposition

\begin{prop}\label{Taylor_prop_SVE-expansion}
Let \cref{control_assum_coefficient} hold. Fix $\hat{u},v\in\cU$. Then, for any $p\in[2,\infty)$, there exists a constant $C_p>0$ such that, for any $\tau\in[0,T)$ and $\ep\in(0,T-\tau]$, the following hold:
\begin{align}
	\label{Taylor_eq_var-delta}
	&\|\delta X^{v,\tau,\ep}\|_{C^p_\bF(0,T)}\leq C_p\big\{\|K_b\|_{1,\ep}+\|K_\sigma\|_{2,\ep}\big\},\\
	\label{Taylor_eq_var-1}
	&\|X^{1,v,\tau,\ep}\|_{C^p_\bF(0,T)}\leq C_p\big\{\|K_b\|_{1,\ep}+\|K_\sigma\|_{2,\ep}\big\},\\
	\label{Taylor_eq_var-delta1}
	&\|\delta X^{1,v,\tau,\ep}\|_{C^p_\bF(0,T)}\leq C_p\big\{\|K_b\|_{1,\ep}+\|K_\sigma\|_{2,\ep}\big\}^2,\\
	\label{Taylor_eq_var-2}
	&\|X^{2,v,\tau,\ep}\|_{C^p_\bF(0,T)}\leq C_p\big\{\|K_b\|_{1,\ep}+\|K_\sigma\|_{2,\ep}\big\}^2,\\
	\label{Taylor_eq_var-delta12}
	&\|\delta X^{1,2,v,\tau,\ep}\|_{C^p_\bF(0,T)}\leq C_p\big\{\|K_b\|_{1,\ep}+\|K_\sigma\|_{2,\ep}\big\}^3.
\end{align}
Furthermore, there exists a constant $C>0$ such that, for any $\tau\in[0,T)$ and $\ep\in(0,T-\tau]$,
\begin{equation}\label{Taylor_eq_var-cost}
	|\delta J^{1,2,v,\tau,\ep}|\leq C\big\{\|K_b\|_{1,\ep}+\|K_\sigma\|_{2,\ep}\big\}^3+C\ep\big\{\|K_b\|_{1,\ep}+\|K_\sigma\|_{2,\ep}\big\}.
\end{equation}
In particular, if $K_b\in L^{3/2}(0,T;\bR^{n\times n})$ and $K_\sigma\in L^6(0,T;\bR^{n\times n})$, then
\begin{equation}\label{Taylor_eq_var-cost'}
	|\delta J^{1,2,v,\tau,\ep}|=o(\ep)\ \ \text{as $\ep\downarrow0$ for any $\tau\in[0,T)$}.
\end{equation}
\end{prop}

%% Proof

\begin{proof}
In this proof, we omit the super scripts $v,\tau,\ep$. For example, we denote $u=u^{v,\tau,\ep}$, $X=X^{v,\tau,\ep}$, $X^1=X^{1,v,\tau,\ep}$, $\delta b=\delta b^v$, $R^v=R$, and so on. For each $p\in[2,\infty)$, we denote by $C_p$ a positive constant which depends on $p$, $\hat{u}$, $v$ and other parameters given in \cref{control_assum_coefficient} but not on $\tau$ or $\ep$. The value of $C_p$ may vary from line to line. Notice that, by \eqref{control_eq_estimate} and $R^u\leq R\vee\hat{R}$,
\begin{equation}\label{Taylor_eq_estimate-X}
	\|\hat{X}\|_{C^p_\bF(0,T)}+\|X\|_{C^p_\bF(0,T)}\leq C_p\ \ \text{for any $p\in[2,\infty)$.}
\end{equation}
For $\varphi\in\{b,\sigma\}$, define
\begin{equation*}
	\tilde{\varphi}_x(s):=\int^1_0\varphi_x(s,u_s,r\hat{X}_s+(1-r)X_s)\,\diff r,\ \tilde{\varphi}_{xx}(s):=2\int^1_0r\varphi_{xx}(s,u_s,r\hat{X}_s+(1-r)X_s)\,\diff r.
\end{equation*}
Then, by Taylor's formula, we have
\begin{align}
	\nonumber&\varphi(s,u_s,X_s)-\varphi(s,\hat{u}_s,\hat{X}_s)\\
	\label{Taylor_eq_Taylor1}&=\tilde{\varphi}_x(s)\delta X_s+\delta\varphi(s)\1_{[\tau,\tau+\ep]}(s)\\
	\label{Taylor_eq_Taylor1'}&=\hat{\varphi}_x(s)\delta X_s+\Big(\tilde{\varphi}_x(s)-\varphi_x(s,u_s,\hat{X}_s)\Big)\delta X_s+\Big(\delta \varphi(s)+\delta\varphi_x(s)\delta X_s\Big)\1_{[\tau,\tau+\ep]}(s)\\
	\nonumber&=\hat{\varphi}_x(s)\delta X_s+\frac{1}{2}\big\langle\hat{\varphi}_{xx}(s)\delta X_s,\delta X_s\big\rangle+\frac{1}{2}\Big\langle\Big(\tilde{\varphi}_{xx}(s)-\varphi_{xx}(s,u_s,\hat{X}_s)\Big)\delta X_s,\delta X_s\Big\rangle\\
	\label{Taylor_eq_Taylor2}&\hspace{1cm}+\Big(\delta\varphi(s)+\delta\varphi_x(s)\delta X_s+\frac{1}{2}\big\langle\delta\varphi_{xx}(s)\delta X_s,\delta X_s\big\rangle\Big)\1_{[\tau,\tau+\ep]}(s).
\end{align}

\underline{The estimate \eqref{Taylor_eq_var-delta} for $\delta X=\delta X^{v,\tau,\ep}$.}
By the expansion \eqref{Taylor_eq_Taylor1}, we see that $\delta X$ solves the SVE
\begin{equation*}
	\delta X_t=\delta\xi_t+\int^t_0K_b(t-s)\tilde{b}_x(s)\delta X_s\,\diff s+\int^t_0K_\sigma(t-s)\tilde{\sigma}_x(s)\delta X_s\,\diff W_s,\ \ t\in[0,T],
\end{equation*}
where
\begin{equation*}
	\delta\xi_t:=\int^t_0K_b(t-s)\delta b(s)\1_{[\tau,\tau+\ep]}(s)\,\diff s+\int^t_0K_\sigma(t-s)\delta\sigma(s)\1_{[\tau,\tau+\ep]}(s)\,\diff W_s,\ \ t\in[0,T].
\end{equation*}
Noting \eqref{Taylor_eq_var-coefficient-state}, by \cref{pre_lemm_convolution}, we see that $\delta\xi\in C^p_\bF(0,T;\bR^n)$ for any $p\in[2,\infty)$ and
\begin{equation*}
	\|\delta\xi\|_{C^p_\bF(0,T)}\leq C_p\big\{\|K_b\|_{1,\ep}+\|K_\sigma\|_{2,\ep}\big\}.
\end{equation*}
Then, noting that $|\tilde{b}_x(s)|\leq\kappa$ and $|\tilde{\sigma}_x(s)|\leq\kappa$, by \cref{pre_lemm_SVE}, we get the estimate \eqref{Taylor_eq_var-delta}.

\underline{The estimate \eqref{Taylor_eq_var-1} for $X^1=X^{1,v,\tau,\ep}$.}
Notice that the first order variational SVE \eqref{Taylor_eq_vareq-SVE1} for $X^1$ can be written as
\begin{equation*}
	X^1_t=\xi^1_t+\int^t_0K_b(t-s)\hat{b}_x(s)X^1_s\,\diff s+\int^t_0K_\sigma(t-s)\hat{\sigma}_x(s)X^1_s\,\diff W_s,\ \ t\in[0,T],
\end{equation*}
where
\begin{equation*}
	\xi^1_t:=\int^t_0K_b(t-s)\delta b(s)\1_{[\tau,\tau+\ep]}(s)\,\diff s+\int^t_0K_\sigma(t-s)\delta\sigma(s)\1_{[\tau,\tau+\ep]}(s)\,\diff W_s,\ \ t\in[0,T].
\end{equation*}
Noting \eqref{Taylor_eq_var-coefficient-state}, by \cref{pre_lemm_convolution}, we see that $\xi^1\in C^p_\bF(0,T;\bR^n)$ for any $p\in[2,\infty)$ and
\begin{equation*}
	\|\xi^1\|_{C^p_\bF(0,T)}\leq C_p\big\{\|K_b\|_{1,\ep}+\|K_\sigma\|_{2,\ep}\big\}.
\end{equation*}
Then, noting that $|\hat{b}_x(s)|\leq\kappa$ and $|\hat{\sigma}_x(s)|\leq\kappa$, by \cref{pre_lemm_SVE}, we get the estimate \eqref{Taylor_eq_var-1}.

\underline{The estimate \eqref{Taylor_eq_var-delta1} for $\delta X^1=\delta X^{1,v,\tau,\ep}$.}
By the expansion \eqref{Taylor_eq_Taylor1'} and the first order variational SVE \eqref{Taylor_eq_vareq-SVE1} for $X^1$, we see that $\delta X^1$ solves the SVE
\begin{equation*}
	\delta X^1_t=\delta\xi^1_t+\int^t_0K_b(t-s)\hat{b}_x(s)\delta X^1_s\,\diff s+\int^t_0K_\sigma(t-s)\hat{\sigma}_x(s)\delta X^1_s\,\diff W_s,\ \ t\in[0,T],
\end{equation*}
where
\begin{align*}
	\delta\xi^1_t&:=\int^t_0K_b(t-s)\Big(\tilde{b}_x(s)-b_x(s,u_s,\hat{X}_s)\Big)\delta X_s\,\diff s+\int^t_0K_b(t-s)\delta b_x(s)\delta X_s\1_{[\tau,\tau+\ep]}(s)\,\diff s\\
	&\hspace{0.5cm}+\int^t_0K_\sigma(t-s)\Big(\tilde{\sigma}_x(s)-\sigma_x(s,u_s,\hat{X}_s)\Big)\delta X_s\,\diff W_s+\int^t_0K_\sigma(t-s)\delta\sigma_x(s)\delta X_s\1_{[\tau,\tau+\ep]}(s)\,\diff W_s,\ \ t\in[0,T].
\end{align*}
By \cref{control_assum_coefficient}, we have $|\tilde{\varphi}_x(s)-\varphi_x(s,u_s,\hat{X}_s)|\leq\frac{\kappa}{2}|\delta X_s|$ for $\varphi\in\{b,\sigma\}$. From this and \eqref{Taylor_eq_var-coefficient-state}, together with the estimate \eqref{Taylor_eq_var-delta} for $\delta X$, by using \cref{pre_lemm_convolution}, we see that $\delta\xi^1\in C^p_\bF(0,T;\bR^n)$ for any $p\in[2,\infty)$ and
\begin{align*}
	\|\delta\xi^1\|_{C^p_\bF(0,T)}&\leq C_p\big\{\|K_b\|_{1,\ep}+\|K_\sigma\|_{2,\ep}\big\}^2+C_p\|K_b\|_{1,\ep}\big\{\|K_b\|_{1,\ep}+\|K_\sigma\|_{2,\ep}\big\}\\
	&\hspace{0.5cm}+C_p\big\{\|K_b\|_{1,\ep}+\|K_\sigma\|_{2,\ep}\big\}^2+C_p\|K_\sigma\|_{2,\ep}\big\{\|K_b\|_{1,\ep}+\|K_\sigma\|_{2,\ep}\big\}\\
	&\leq C_p\big\{\|K_b\|_{1,\ep}+\|K_\sigma\|_{2,\ep}\big\}^2.
\end{align*}
Then, by \cref{pre_lemm_SVE}, we get the estimate \eqref{Taylor_eq_var-delta1}.

\underline{The estimate \eqref{Taylor_eq_var-2} for $X^2=X^{2,v,\tau,\ep}$.}
Notice that the second order variational SVE \eqref{Taylor_eq_vareq-SVE2} for $X^2$ can be written as
\begin{equation*}
	X^2_t=\xi^2_t+\int^t_0K_b(t-s)\hat{b}_x(s)X^2_s\,\diff s+\int^t_0K_\sigma(t-s)\hat{\sigma}_x(s)X^2_s\,\diff W_s,\ \ t\in[0,T],
\end{equation*}
where
\begin{align*}
	\xi^2_t&:=\frac{1}{2}\int^t_0K_b(t-s)\big\langle\hat{b}_{xx}(s)X^1_s,X^1_s\big\rangle\,\diff s+\int^t_0K_b(t-s)\delta b_x(s)X^1_s\1_{[\tau,\tau+\ep]}(s)\,\diff s\\
	&\hspace{0.5cm}+\frac{1}{2}\int^t_0K_\sigma(t-s)\big\langle\hat{\sigma}_{xx}(s)X^1_s,X^1_s\big\rangle\,\diff W_s+\int^t_0K_\sigma(t-s)\delta\sigma_x(s)X^1_s\1_{[\tau,\tau+\ep]}(s)\,\diff W_s,\ \ t\in[0,T].
\end{align*}
From \eqref{Taylor_eq_var-coefficient-state}, together with the estimate \eqref{Taylor_eq_var-1} for $X^1$, by using \cref{pre_lemm_convolution}, we see that $\xi^2\in C^p_\bF(0,T;\bR^n)$ for any $p\in[2,\infty)$ and
\begin{align*}
	\|\xi^2\|_{C^p_\bF(0,T)}&\leq C_p\big\{\|K_b\|_{1,\ep}+\|K_\sigma\|_{2,\ep}\big\}^2+C_p\|K_b\|_{1,\ep}\big\{\|K_b\|_{1,\ep}+\|K_\sigma\|_{2,\ep}\big\}\\
	&\hspace{0.5cm}+C_p\big\{\|K_b\|_{1,\ep}+\|K_\sigma\|_{2,\ep}\big\}^2+C_p\|K_\sigma\|_{2,\ep}\big\{\|K_b\|_{1,\ep}+\|K_\sigma\|_{2,\ep}\big\}\\
	&\leq C_p\big\{\|K_b\|_{1,\ep}+\|K_\sigma\|_{2,\ep}\big\}^2.
\end{align*}
Then, by \cref{pre_lemm_SVE}, we get the estimate \eqref{Taylor_eq_var-2}.

\underline{The estimate \eqref{Taylor_eq_var-delta12} for $\delta X^{1,2}=\delta X^{1,2,v,\tau,\ep}$.}
By the expansion \eqref{Taylor_eq_Taylor2} and the variational SVE \eqref{Taylor_eq_vareq-SVE12} for $X^{1,2}=X^1+X^2$, we see that $\delta X^{1,2}$ solves the SVE
\begin{equation*}
	\delta X^{1,2}_t=\delta\xi^{1,2}_t+\int^t_0K_b(t-s)\hat{b}_x(s)\delta X^{1,2}_s\,\diff s+\int^t_0K_\sigma(t-s)\hat{\sigma}_x(s)\delta X^{1,2}_s\,\diff W_s,\ \ t\in[0,T],
\end{equation*}
where
\begin{align*}
	\delta\xi^{1,2}_t&:=\frac{1}{2}\int^t_0K_b(t-s)\big\langle\hat{b}_{xx}(s)\delta X^1_s,\delta X_s+X^1_s\big\rangle\,\diff s\\
	&\hspace{0.5cm}+\frac{1}{2}\int^t_0K_b(t-s)\Big\langle\Big(\tilde{b}_{xx}(s)-b_{xx}(s,u_s,\hat{X}_s)\Big)\delta X_s,\delta X_s\Big\rangle\,\diff s\\
	&\hspace{0.5cm}+\int^t_0K_b(t-s)\delta b_x(s)\delta X^1_s\1_{[\tau,\tau+\ep]}(s)\,\diff s\\
	&\hspace{0.5cm}+\frac{1}{2}\int^t_0K_b(t-s)\big\langle\delta b_{xx}(s)\delta X_s,\delta X_s\big\rangle\1_{[\tau,\tau+\ep]}(s)\,\diff s\\
	&\hspace{0.5cm}+\frac{1}{2}\int^t_0K_\sigma(t-s)\big\langle\hat{\sigma}_{xx}(s)\delta X^1_s,\delta X_s+X^1_s\big\rangle\,\diff W_s\\
	&\hspace{0.5cm}+\frac{1}{2}\int^t_0K_\sigma(t-s)\Big\langle\Big(\tilde{\sigma}_{xx}(s)-\sigma_{xx}(s,u_s,\hat{X}_s)\Big)\delta X_s,\delta X_s\Big\rangle\,\diff W_s\\
	&\hspace{0.5cm}+\int^t_0K_\sigma(t-s)\delta\sigma_x(s)\delta X^1_s\1_{[\tau,\tau+\ep]}(s)\,\diff W_s\\
	&\hspace{0.5cm}+\frac{1}{2}\int^t_0K_\sigma(t-s)\big\langle\delta\sigma_{xx}(s)\delta X_s,\delta X_s\big\rangle\1_{[\tau,\tau+\ep]}(s)\,\diff W_s\\
	&=:I_1(t)+I_2(t)+I_3(t)+I_4(t)+I_5(t)+I_6(t)+I_7(t)+I_8(t),\ \ t\in[0,T].
\end{align*}
By \cref{control_assum_coefficient} and $R(s,u_s)=R^u_s\leq R_s\vee\hat{R}_s$ with $R=R^v$ and $\hat{R}=R^{\hat{u}}$, we have
\begin{equation*}
	|\tilde{\varphi}_{xx}(s)-\varphi_{xx}(s,u_s,\hat{X}_s)|\leq\frac{1}{3}\big(R_s\vee\hat{R}_s\big)\big(1\vee|X_s|\vee|\hat{X}_s|\big)^\gamma|\delta X_s|\ \ \text{for $\varphi\in\{b,\sigma\}$.}
\end{equation*}
Noting \eqref{Taylor_eq_var-coefficient-state} and \eqref{Taylor_eq_estimate-X}, by using \cref{pre_lemm_convolution} and H\"{o}lder's inequality, together with the estimates \eqref{Taylor_eq_var-delta}, \eqref{Taylor_eq_var-1} and \eqref{Taylor_eq_var-delta1} for $\delta X$, $X^1$ and $\delta X^1$, we see that $I_k\in C^p_\bF(0,T;\bR^n)$ for any $p\in[2,\infty)$ and $k\in\{1,\dots,8\}$, and
\begin{equation*}
	\|I_k\|_{C^p_\bF(0,T)}\leq
	\begin{dcases}
	C_p\big\{\|K_b\|_{1,\ep}+\|K_\sigma\|_{2,\ep}\big\}^3,\ \ k\in\{1,2,5,6\},\\
	C_p\|K_b\|_{1,\ep}\big\{\|K_b\|_{1,\ep}+\|K_\sigma\|_{2,\ep}\big\}^2,\ \ k\in\{3,4\},\\
	C_p\|K_\sigma\|_{2,\ep}\big\{\|K_b\|_{1,\ep}+\|K_\sigma\|_{2,\ep}\big\}^2,\ \ k\in\{7,8\}.
	\end{dcases}
\end{equation*}
Hence, we see that $\delta\xi^{1,2}\in C^p_\bF(0,T;\bR^n)$ and
\begin{equation*}
	\|\delta\xi^{1,2}\|_{C^p_\bF(0,T)}\leq C_p\big\{\|K_b\|_{1,\ep}+\|K_\sigma\|_{2,\ep}\big\}^3.
\end{equation*}
Then, by \cref{pre_lemm_SVE}, we get the estimate \eqref{Taylor_eq_var-delta12}.

\underline{The estimate \eqref{Taylor_eq_var-cost} for $\delta J^{1,2}=\delta J^{1,2,v,\tau,\ep}$.}
By using Taylor expansions for $h$ and $f$ similar to \eqref{Taylor_eq_Taylor2}, we have
\begin{align*}
	J(u)-J(\hat{u})&=\bE\left[\hat{h}_x\delta X_T+\frac{1}{2}\big\langle\hat{h}_{xx}\delta X_T,\delta X_T\big\rangle+\frac{1}{2}\Big\langle\Big(\tilde{h}_{xx}-\hat{h}_{xx}\Big)\delta X_T,\delta X_T\Big\rangle\right]\\
	&\hspace{0.5cm}+\bE\left[\int^T_0\Big\{\hat{f}_x(s)\delta X_s+\frac{1}{2}\big\langle\hat{f}_{xx}(s)\delta X_s,\delta X_s\big\rangle+\frac{1}{2}\Big\langle\Big(\tilde{f}_{xx}(s)-f_{xx}(s,u_s,\hat{X}_s)\Big)\delta X_s,\delta X_s\Big\rangle\Big\}\,\diff s\right]\\
	&\hspace{0.5cm}+\bE\left[\int^{\tau+\ep}_\tau\Big\{\delta f(s)+\delta f_x(s)\delta X_s+\frac{1}{2}\big\langle\delta f_{xx}(s)\delta X_s,\delta X_s\big\rangle\Big\}\,\diff s\right],\\
%	&=\bE\left[\hat{h}_xX^{1,2}_T+\frac{1}{2}\big\langle\hat{h}_{xx}X^1_T,X^1_T\big\rangle\right]+\bE\left[\int^T_0\Big\{\hat{f}_x(s)X^{1,2}_s+\frac{1}{2}\big\langle\hat{f}_{xx}(s)X^1_s,X^1_s\big\rangle\Big\}\,\diff s\right]+\bE\left[\int^{\tau+\ep}_\tau\delta f(s)\,\diff s\right]\\
%	&\hspace{0.5cm}+\bE\left[\hat{h}_x\delta X^{1,2}_T+\frac{1}{2}\big\langle\hat{h}_{xx}\delta X^1_T,\delta X_T+X^1_T\big\rangle+\frac{1}{2}\Big\langle\Big(\tilde{h}_{xx}-h_{xx}\Big)\delta X_T,\delta X_T\Big\rangle\right]\\
%	&\hspace{0.5cm}+\bE\left[\int^T_0\Big\{\hat{f}_x(s)\delta X^{1,2}_s+\frac{1}{2}\big\langle\hat{f}_{xx}(s)\delta X^1_s,\delta X_s+X^1_s\big\rangle+\frac{1}{2}\Big\langle\Big(\tilde{f}_{xx}(s)-f_{xx}(s,u_s,\hat{X}_s)\Big)\delta X_s,\delta X_s\Big\rangle\Big\}\,\diff s\right]\\
%	&\hspace{0.5cm}+\bE\left[\int^{\tau+\ep}_\tau\Big\{\delta f_x(s)\delta X_s+\frac{1}{2}\big\langle\delta f_{xx}(s)\delta X_s,\delta X_s\big\rangle\Big\}\,\diff s\right],
\end{align*}
where
\begin{equation*}
	\tilde{h}_{xx}:=2\int^1_0rh_{xx}(r\hat{X}_T+(1-r)X_T)\,\diff r,\ \ \tilde{f}_{xx}(s):=2\int^1_0rf_{xx}(s,u_s,r\hat{X}_s+(1-r)X_s)\,\diff r.
\end{equation*}
Hence, by the definition \eqref{Taylor_eq_vareq-cost} of $J^{1,2}=J^{1,2,v,\tau,\ep}$, the term $\delta J^{1,2}=J(u)-J(\hat{u})-J^{1,2}$ is expressed as
\begin{align*}
	\delta J^{1,2}&=\bE\left[\hat{h}_x\delta X^{1,2}_T+\frac{1}{2}\big\langle\hat{h}_{xx}\delta X^1_T,\delta X_T+X^1_T\big\rangle+\frac{1}{2}\Big\langle\Big(\tilde{h}_{xx}-\hat{h}_{xx}\Big)\delta X_T,\delta X_T\Big\rangle\right]\\
	&\hspace{0.5cm}+\bE\left[\int^T_0\Big\{\hat{f}_x(s)\delta X^{1,2}_s+\frac{1}{2}\big\langle\hat{f}_{xx}(s)\delta X^1_s,\delta X_s+X^1_s\big\rangle+\frac{1}{2}\Big\langle\Big(\tilde{f}_{xx}(s)-f_{xx}(s,u_s,\hat{X}_s)\Big)\delta X_s,\delta X_s\Big\rangle\Big\}\,\diff s\right]\\
	&\hspace{0.5cm}+\bE\left[\int^{\tau+\ep}_\tau\Big\{\delta f_x(s)\delta X_s+\frac{1}{2}\big\langle\delta f_{xx}(s)\delta X_s,\delta X_s\big\rangle\Big\}\,\diff s\right].
\end{align*}
By \cref{control_assum_coefficient} and $R(s,u_s)=R^u_s\leq R_s\vee\hat{R}_s$ with $R=R^v$ and $\hat{R}=R^{\hat{u}}$, we have
\begin{equation*}
	|\tilde{f}_{xx}(s)-f_{xx}(s,u_s,\hat{X}_s)|\leq\frac{1}{3}\big(R_s\vee\hat{R}_s\big)\big(1\vee|X_s|\vee|\hat{X}_s|\big)^\gamma|\delta X_s|\ \ \text{and}\ \ |\tilde{h}_{xx}-\hat{h}_{xx}|\leq\frac{\kappa}{3}(1\vee|X_T|\vee|\hat{X}_T|)^\gamma|\delta X_T|.
\end{equation*}
Noting \eqref{Taylor_eq_var-coefficient-cost} and \eqref{Taylor_eq_estimate-X}, by using H\"{o}lder's inequality and the estimates \eqref{Taylor_eq_var-delta}, \eqref{Taylor_eq_var-1}, \eqref{Taylor_eq_var-delta1} and \eqref{Taylor_eq_var-delta12} for $\delta X$, $X^1$, $\delta X^1$ and $\delta X^{1,2}$, we obtain \eqref{Taylor_eq_var-cost}. Finally, noting \cref{intro_rem_K-ep}, if $K_b\in L^{3/2}(0,T;\bR^{n\times n})$ and $K_\sigma\in L^6(0,T;\bR^{n\times n})$, we see that
\begin{equation*}
	\|K_b\|^3_{1,\ep}\leq\|K_b\|^3_{3/2,\ep}\ep,\ \ \|K_\sigma\|^3_{2,\ep}\leq\|K_\sigma\|^3_{6,\ep}\ep,
\end{equation*}
and
\begin{equation*}
	\lim_{\ep\downarrow0}\|K_b\|_{1,\ep}=\lim_{\ep\downarrow0}\|K_b\|_{3/2,\ep}=\lim_{\ep\downarrow0}\|K_\sigma\|_{2,\ep}=\lim_{\ep\downarrow0}\|K_\sigma\|_{6,\ep}=0.
\end{equation*}
From the above estimates and \eqref{Taylor_eq_var-cost}, we get \eqref{Taylor_eq_var-cost'}. This completes the proof.
\end{proof}

%% Remark

\begin{rem}
Consider for example the fractional kernels $K_b(t)=\frac{1}{\Gamma(\beta_b)}t^{\beta_b-1}I_{n\times n}$ and $K_\sigma(t)=\frac{1}{\Gamma(\beta_\sigma)}t^{\beta_\sigma-1}$ with $\beta_b\in(0,1)$ and $\beta_\sigma\in(\frac{1}{2},1)$. For these kernels, the conditions $K_b\in L^{3/2}(0,T;\bR^{n\times n})$ and $K_\sigma\in L^6(0,T;\bR^{n\times n})$ hold if and only if $\beta_b>\frac{1}{3}$ and $\beta_\sigma>\frac{5}{6}$.
\end{rem}

%%%%%%%%%%%%%%%%%%%%%%%%%%%%%%%%%%
%%%%%%%%%%%%%%%%%%%%%%%%%%%%%%%%%%
\section{Adjoint equations and global maximum principle}\label{MP}
%%%%%%%%%%%%%%%%%%%%%%%%%%%%%%%%%%
%%%%%%%%%%%%%%%%%%%%%%%%%%%%%%%%%%

By the variational argument (\cref{Taylor_prop_SVE-expansion}), under suitable integrability conditions on the kernels, we immediately obtain a first order necessary condition for a fixed control $\hat{u}\in\cU$ to be optimal, namely,
\begin{equation*}
	\liminf_{\ep\downarrow0}\frac{1}{\ep}J^{1,2,v,\tau,\ep}\geq0
\end{equation*}
for any $v\in\cU$ and $\tau\in[0,T)$, where $J^{1,2,v,\tau,\ep}$ is defined by \eqref{Taylor_eq_vareq-cost}. However, the above is still an implicit necessary condition that is not easily applicable, since the term $J^{1,2,v,\tau,\ep}$ involves the solutions $X^{1,v,\tau,\ep}$ and $X^{1,2,v,\tau,\ep}$ to the variational SVEs \eqref{Taylor_eq_vareq-SVE1} and \eqref{Taylor_eq_vareq-SVE12}, respectively, which depend on the choices of $v,\tau,\ep$. Our aim is to derive first and second order adjoint equations, which enable us to get rid of the terms $X^{1,v,\tau,\ep}$ and $X^{1,2,v,\tau,\ep}$ from $J^{1,2,v,\tau,\ep}$ through a kind of duality relation. To do so, we need to treat the quadratic functional of $(X^{1,v,\tau,\ep},X^{1,2,v,\tau,\ep})$ appearing in the expression \eqref{Taylor_eq_vareq-cost} of $J^{1,2,v,\tau,\ep}$. One of the main difficulty here is that the solutions $X^{1,v,\tau,\ep}$, $X^{1,2,v,\tau,\ep}$ of variational SVEs \eqref{Taylor_eq_vareq-SVE1} and \eqref{Taylor_eq_vareq-SVE12} with singular kernels are no longer semimartingales. The lack of the semimartingale property prevents us to use It\^{o}'s formula and the standard BSDE technique, which have been used in the study of stochastic control problems for SDEs.

In order to overcome the difficulty mentioned above, we utilize (and slightly generalize) the infinite dimensional framework proposed in \cite{Ha24}. We impose the following structure condition on the kernels.

%% Assumption

\begin{assum}\label{MP_assum_kernel}
The kernels $K_b,K_\sigma:(0,T)\to\bR^{n\times n}$ are of the following forms:
\begin{equation}\label{MP_eq_kernel}
	K_b(t)=\int_{\bR_+}e^{-\theta t}M_b(\theta)\dmu\ \ \text{and}\ \ K_\sigma(t)=\int_{\bR_+}e^{-\theta t}M_\sigma(\theta)\dmu,\ \ t\in(0,T),
\end{equation}
for some Borel measure $\mu$ on $\bR_+$ and matrix-valued Borel measurable maps $M_b,M_\sigma:\bR_+\to\bR^{n\times n}$ such that
\begin{equation}\label{MP_eq_mu-integrable}
	\int_{\bR_+}r(\theta)\dmu<\infty
\end{equation}
and
\begin{equation}\label{MP_eq_M-integrable}
	\int_{\bR_+}(1+\theta)^{-\alpha}r(\theta)|M_b(\theta)|^2\dmu<\infty,\ \ \int_{\bR_+}(1+\theta)^{1-\alpha}r(\theta)|M_\sigma(\theta)|^2\dmu<\infty
\end{equation}
for some $\alpha\in[0,1)$, where $r(\theta):=1\wedge(\theta^{-1/2})$.
\end{assum}

In the following, when we impose \cref{MP_assum_kernel} on the kernels $K_b$ and $K_\sigma$, we fix $\mu$, $M_b$, $M_\sigma$ and $\alpha$ appearing therein and use the notation $r(\theta):=1\wedge(\theta^{-1/2})$. The class of the kernels satisfying \cref{MP_assum_kernel} is characterized by the following lemma.

%% Lemma

\begin{lemm}\label{MP_lemm_kernel}
Let $\alpha\in[0,1)$. The kernels $K_b,K_\sigma:(0,T)\to\bR^{n\times n}$ satisfy \cref{MP_assum_kernel} if and only if they are in form of
\begin{equation}\label{MP_eq_CM}
	K_b(t)=\big(K^{i,j,+}_b(t)-K^{i,j,-}_b(t)\big)^n_{i,j=1}\ \ \text{and}\ \ K_\sigma(t)=\big(K^{i,j,+}_\sigma(t)-K^{i,j,-}_\sigma(t)\big)^n_{i,j=1},\ \ t\in(0,T),
\end{equation}
for some completely monotone functions\footnote{A (scalar) function $k:(0,\infty)\to\bR_+$ is said to be completely monotone if it is infinitely differentiable and such that $(-1)^m\frac{\diff^m}{\diff t^m}k(t)\geq0$ for any $t\in(0,\infty)$ and any nonnegative integer $m$.} $K^{i,j,\pm}_b,K^{i,j,\pm}_\sigma:(0,\infty)\to\bR_+$ such that
\begin{equation}\label{MP_eq_CM-integrability}
\begin{split}
	&\int^T_0t^{-\frac{1}{2}}K^{i,j,\pm}_b(t)\,\diff t<\infty\ \ \text{and}\ \ \lim_{t\downarrow0}K^{i,j,\pm}_\sigma(t)<\infty\ \ \text{for any $i,j\in\{1,\dots,n\}$ if $\alpha=0$, and}\\
	&\int^T_0t^{\frac{\alpha-1}{2}}K^{i,j,\pm}_b(t)\,\diff t<\infty\ \ \text{and}\ \ \int^T_0t^{\frac{\alpha}{2}-1}K^{i,j,\pm}_\sigma(t)\,\diff t<\infty\ \ \text{for any $i,j\in\{1,\dots,n\}$ if $\alpha\in(0,1)$.}
\end{split}
\end{equation}
Furthermore, under \cref{MP_assum_kernel}, $K_b\in L^{\frac{2}{1+\alpha}}(0,T;\bR^{n\times n})$ and $K_\sigma\in L^{\frac{2}{\alpha}}(0,T;\bR^{n\times n})$.
\end{lemm}

%% Remark

\begin{rem}
Notice that, by H\"{o}lder's inequality, \eqref{MP_eq_CM-integrability} holds if $K^{i,j,\pm}_b\in\bigcup_{\delta>0}L^{\frac{2}{1+\alpha}+\delta}(0,T;\bR)$ and $K^{i,j,\pm}_\sigma\in\bigcup_{\delta>0}L^{\frac{2}{\alpha}+\delta}(0,T;\bR)$ for any $i,j\in\{1,\dots,n\}$.
\end{rem}

%% Proof

\begin{proof}[Proof of \cref{MP_lemm_kernel}]
Suppose that $K_b$ and $K_\sigma$ satisfy \cref{MP_assum_kernel}. For each $i,j\in\{1,\dots,n\}$, define
\begin{equation}\label{MP_eq_CM'}
	K^{i,j,\pm}_b(t)=\int_{\bR_+}e^{-\theta t}\,\mu^{i,j,\pm}_b(\diff\theta),\ \ K^{i,j,\pm}_\sigma(t)=\int_{\bR_+}e^{-\theta t}\,\mu^{i,j,\pm}_\sigma(\diff\theta),\ \ t\in(0,\infty),
\end{equation}
where, for $\varphi\in\{b,\sigma\}$, $\mu^{i,j,+}_\varphi$ and $\mu^{i,j,-}_\varphi$ are Borel measures on $\bR_+$ defined by
\begin{equation*}
	\mu^{i,j,+}_\varphi(\diff\theta):=\max\{M_\varphi(\theta)^{i,j},0\}\,\mu(\diff\theta)\ \ \text{and}\ \ \mu^{i,j,-}_\varphi(\diff\theta):=\max\{-M_\varphi(\theta)^{i,j},0\}\,\mu(\diff\theta).
\end{equation*}
Clearly, the expressions in \eqref{MP_eq_CM} hold. From \eqref{MP_eq_mu-integrable} and \eqref{MP_eq_M-integrable}, using the Cauchy--Schwarz inequality, we can show that
\begin{equation}\label{MP_eq_mu-integrability'}
	\int_{\bR_+}\big(1\wedge(\theta^{-\frac{1+\alpha}{2}})\big)\,\mu^{i,j,\pm}_b(\diff\theta)<\infty\ \ \text{and}\ \ \int_{\bR_+}\big(1\wedge(\theta^{-\frac{\alpha}{2}})\big)\,\mu^{i,j,\pm}_\sigma(\diff\theta)<\infty\ \ \text{for any $i,j\in\{1,\dots,n\}$.}
\end{equation}
By Bernstein's theorem (cf.\ \cite[Theorem 1.4]{ScSoVo12}), we see that $K^{i,j,\pm}_b$ and $K^{i,j,\pm}_\sigma$ are completely monotone functions. Furthermore, \cite[Lemma 2.1]{Ha24} shows that \eqref{MP_eq_mu-integrability'} is equivalent to \eqref{MP_eq_CM-integrability} and implies that $K^{i,j,\pm}_b\in L^{\frac{2}{1+\alpha}}(0,T;\bR)$ and $K^{i,j,\pm}_\sigma\in L^{\frac{2}{\alpha}}(0,T;\bR)$ for each $i,j\in\{1,\dots,n\}$. The latter yields that $K_b\in L^{\frac{2}{1+\alpha}}(0,T;\bR^{n\times n})$ and $K_\sigma\in L^{\frac{2}{\alpha}}(0,T;\bR^{n\times n})$.

Conversely, assume that $K_b$ and $K_\sigma$ are in form of \eqref{MP_eq_CM} for some completely monotone functions $K^{i,j,\pm}_b,K^{i,j,\pm}_\sigma:(0,\infty)\to\bR_+$ satisfying \eqref{MP_eq_CM-integrability}. Again by Bernstein's theorem (cf.\ \cite[Theorem 1.4]{ScSoVo12}) and \cite[Lemma 2.1]{Ha24}, for each $i,j\in\{1,\dots,n\}$, there exist (unique) Borel measures $\mu^{i,j,\pm}_b$ and $\mu^{i,j,\pm}_\sigma$ on $\bR_+$ satisfying \eqref{MP_eq_CM'} and \eqref{MP_eq_mu-integrability'}. Define a Borel measure $\mu$ on $\bR_+$ by
\begin{equation*}
	\mu(\diff\theta):=\big(1+\theta\big)^{-\frac{\alpha}{2}}\sum^n_{i,j=1}\big\{\mu^{i,j,+}_b(\diff\theta)+\mu^{i,j,-}_b(\diff\theta)\big\}+(1+\theta)^{\frac{1-\alpha}{2}}\sum^n_{i,j=1}\big\{\mu^{i,j,+}_\sigma(\diff\theta)+\mu^{i,j,-}_\sigma(\diff\theta)\big\}.
\end{equation*}
Notice that $\mu^{i,j,\pm}_b$ and $\mu^{i,j,\pm}_\sigma$ are absolutely continuous with respect to $\mu$. Furthermore, by \eqref{MP_eq_mu-integrability'}, we see that $\mu$ satisfies the integrability condition \eqref{MP_eq_mu-integrable}. Define matrix-valued Borel measurable maps $M_b,M_\sigma:\bR_+\to\bR^{n\times n}$ by
\begin{equation*}
	M_b:=\left(\frac{\diff\mu^{i,j,+}_b}{\diff\mu}-\frac{\diff\mu^{i,j,-}_b}{\diff\mu}\right)^n_{i,j=1},\ \ M_\sigma:=\left(\frac{\diff\mu^{i,j,+}_\sigma}{\diff\mu}-\frac{\diff\mu^{i,j,-}_\sigma}{\diff\mu}\right)^n_{i,j=1}.
\end{equation*}
Noting that
\begin{equation*}
	0\leq\frac{\diff\mu^{i,j,\pm}_b}{\diff\mu}(\theta)\leq\big(1+\theta)^{\frac{\alpha}{2}}\ \ \text{and}\ \ 0\leq\frac{\diff\mu^{i,j,\pm}_\sigma}{\diff\mu}(\theta)\leq\big(1+\theta)^{-\frac{1-\alpha}{2}}\ \ \text{for $\mu$-a.e.\ $\theta\in\bR_+$ for any $i,j\in\{1,\dots,n\}$,}
\end{equation*}
we see that $M_b$ and $M_\sigma$ satisfy the integrability condition \eqref{MP_eq_M-integrable}. By \eqref{MP_eq_CM} and \eqref{MP_eq_CM'}, we see that the representations in \eqref{MP_eq_kernel} hold. This completes the proof.
\end{proof}

%% Remark

\begin{rem}
\begin{itemize}
\item
From \eqref{MP_eq_CM-integrability}, the number $\alpha\in[0,1)$ is regarded as the degree of singularity of the kernels $K_b$ and $K_\sigma$ at zero. Namely, the larger the number $\alpha$ is, the more singular the kernels $K_b$ and $K_\sigma$ can be. When $\alpha=0$, the kernel $K_\sigma$ appearing in the stochastic integral term becomes regular in the sense that $\lim_{t\downarrow0}|K_\sigma(t)|<\infty$, while the kernel $K_b$ appearing in the Lebesgue integral term is still allowed to be singular and belongs to $L^2(0,T;\bR^{n\times n})$. The number $\alpha$ can also be seen as the degree of the ``heavy tail'' (that is, the weakness of the integrability at $\infty$) of the vector measures $M_b(\theta)\dmu$ and $M_\sigma(\theta)\dmu$ in the representation \eqref{MP_eq_kernel}.
\item
From the last assertion of \cref{MP_lemm_kernel} and \cref{intro_rem_K-ep}, under \cref{MP_assum_kernel}, we have $\|K_b\|_{1,\ep}\leq\|K_b\|_{2/(1+\alpha),\ep}\ep^{\frac{1-\alpha}{2}}$ and $\|K_\sigma\|_{2,\ep}\leq\|K_\sigma\|_{2/\alpha,\ep}\ep^{\frac{1-\alpha}{2}}$. The conditions that $K_b\in L^{3/2}(0,T;\bR^{n\times n})$ and $K_\sigma\in L^6(0,T;\bR^{n\times n})$ assumed in the last assertion of \cref{Taylor_prop_SVE-expansion} are satisfied if \cref{MP_assum_kernel} holds with $\alpha=\frac{1}{3}$.
\item
Important examples of completely monotone functions can be found in \cite[Example 2.3]{Ha24}. For example, consider the fractional kernels in form of
\begin{equation*}
	K_b(t)=\frac{1}{\Gamma(\beta_b)}t^{\beta_b-1}I_{n\times n},\ \ K_\sigma(t)=\frac{1}{\Gamma(\beta_\sigma)}t^{\beta_\sigma-1}I_{n\times n},
\end{equation*}
for some constants $\beta_b\in(0,1)$ and $\beta_\sigma\in(\frac{1}{2},1)$. Here, $\Gamma(\beta):=\int_{\bR_+}\theta^{\beta-1}e^{-\theta}\,\diff \theta$, $\beta>0$, denotes the Gamma function. These kernels satisfy \cref{MP_assum_kernel} for any $(1-2\beta_b)\vee(2-2\beta_\sigma)<\alpha<1$. Indeed, we can take
\begin{equation*}
	\mu(\diff\theta)=\theta^{-\gamma}\1_{(0,\infty)}(\theta)\,\diff\theta,\ \ M_b(\theta)=\frac{\theta^{\gamma-\beta_b}}{\Gamma(\beta_b)\Gamma(1-\beta_b)}I_{n\times n},\ \ M_\sigma(\theta)=\frac{\theta^{\gamma-\beta_\sigma}}{\Gamma(\beta_\sigma)\Gamma(1-\beta_\sigma)}I_{n\times n},
\end{equation*}
where $\gamma$ is an arbitrary constant such that $\frac{1}{2}<\gamma<(\alpha+2\beta_b-\frac{1}{2})\wedge(\alpha+2\beta_\sigma-\frac{3}{2})\wedge1$.
\item
Suppose that $(K_b^1,K_\sigma^1)$ and $(K_b^2,K_\sigma^2)$ satisfy \cref{MP_assum_kernel} with $(\mu^i,M_b^i,M_\sigma^i,\alpha^i)$, $i=1,2$, as respective tuples. Then, for any constants $c_b^1,c_b^1,c_\sigma^1,c_\sigma^2\in\bR$, the pair $(K_b,K_\sigma)=(c_b^1K_b^1+c_b^2K_b^2,c_\sigma^1 K_\sigma^1+c_\sigma^2K_\sigma^2)$ satisfies \cref{MP_assum_kernel} with
\begin{equation*}
	\mu=\mu^1+\mu^2,\ \ M_b=c_b^1M_b^1\frac{\diff\mu^1}{\diff\mu}+c_b^2M_b^2\frac{\diff\mu^2}{\diff\mu},\ \ M_\sigma=c_\sigma^1M_\sigma^1\frac{\diff\mu^1}{\diff\mu}+c_\sigma^2M_\sigma^2\frac{\diff\mu^2}{\diff\mu},\ \ \alpha=\alpha^1\vee\alpha^2.
\end{equation*}
\end{itemize}
\end{rem}

%%%%%%
\subsection{Lifting the variational SVEs}\label{MP-lift}
%%%%%%

Suppose that \cref{control_assum_coefficient} and \cref{MP_assum_kernel} hold. Fix $\hat{u},v\in\cU$, $\tau\in[0,T)$ and $\ep\in(0,T-\tau]$. We use the same notations as in \cref{Taylor}.

Following \cite{Ha24}, we introduce ``lifts'' of the variational SVEs \eqref{Taylor_eq_vareq-SVE1} and \eqref{Taylor_eq_vareq-SVE12} for $X^{1,v,\tau,\ep}$ and $X^{1,2,v,\tau,\ep}$, respectively. For each $\theta\in\bR_+$, consider the following SDEs on $\bR^n$ parametrized by $\theta$:
\begin{equation}\label{MP_eq_vareq-SVE1-lift}
	\begin{dcases}
	\diff Y^{1,v,\tau,\ep}_t(\theta)=-\theta Y^{1,v,\tau,\ep}_t(\theta)\,\diff t+M_b(\theta)\big\{\hat{b}_x(t)X^{1,v,\tau,\ep}_t+\delta b^v(t)\1_{[\tau,\tau+\ep]}(t)\big\}\,\diff t\\
	\hspace{4cm}+M_\sigma(\theta)\big\{\hat{\sigma}_x(t)X^{1,v,\tau,\ep}_t+\delta\sigma^v(t)\1_{[\tau,\tau+\ep]}(t)\big\}\,\diff W_t,\ \ t\in[0,T],\\
	Y^{1,v,\tau,\ep}_0(\theta)=0,
	\end{dcases}
\end{equation}
and
\begin{equation}\label{MP_eq_vareq-SVE12-lift}
	\begin{dcases}
	\diff Y^{1,2,v,\tau,\ep}_t(\theta)=-\theta Y^{1,2,v,\tau,\ep}_t(\theta)\,\diff t+M_b(\theta)\Big\{\hat{b}_x(t)X^{1,2,v,\tau,\ep}_t+\frac{1}{2}\big\langle\hat{b}_{xx}(t)X^{1,v,\tau,\ep}_t,X^{1,v,\tau,\ep}_t\big\rangle\\
	\hspace{7cm}+\Big(\delta b^v(t)+\delta b^v_x(t)X^{1,v,\tau,\ep}_t\Big)\1_{[\tau,\tau+\ep]}(t)\Big\}\,\diff t\\
	\hspace{2.2cm}+M_\sigma(\theta)\Big\{\hat{\sigma}_x(t)X^{1,2,v,\tau,\ep}_t+\frac{1}{2}\big\langle\hat{\sigma}_{xx}(t)X^{1,v,\tau,\ep}_t,X^{1,v,\tau,\ep}_t\big\rangle\\
	\hspace{5cm}+\Big(\delta\sigma^v(t)+\delta\sigma^v_x(t)X^{1,v,\tau,\ep}_t\Big)\1_{[\tau,\tau+\ep]}(t)\Big\}\,\diff W_t,\ \ t\in[0,T],\\
	Y^{1,2,v,\tau,\ep}_0(\theta)=0.
	\end{dcases}
\end{equation}
Noting \eqref{Taylor_eq_var-coefficient-state} and $X^{1,v,\tau,\ep},X^{1,2,v,\tau,\ep}\in\bigcap_{p\in[2,\infty)}C^p_\bF(0,T;\bR^n)$, the above SDEs admit unique solutions $Y^{1,v,\tau,\ep}(\theta)$ and $Y^{1,2,v,\tau,\ep}(\theta)$ such that
\begin{equation*}
	\bE\left[\sup_{t\in[0,T]}|Y^{1,v,\tau,\ep}_t(\theta)|^p+\sup_{t\in[0,T]}|Y^{1,2,v,\tau,\ep}_t(\theta)|^p\right]<\infty\ \ \text{for any $p\in[2,\infty)$.}
\end{equation*}
The solutions have the following expressions:
\begin{equation}\label{MP_eq_vareq-SVE1-lift'}
\begin{split}
	Y^{1,v,\tau,\ep}_t(\theta)&=\int^t_0e^{-\theta(t-s)}M_b(\theta)\Big\{\hat{b}_x(s)X^{1,v,\tau,\ep}_s+\delta b^v(s)\1_{[\tau,\tau+\ep]}(s)\Big\}\,\diff s\\
	&\hspace{1cm}+\int^t_0e^{-\theta(t-s)}M_\sigma(\theta)\Big\{\hat{\sigma}_x(s)X^{1,v,\tau,\ep}_s+\delta\sigma^v(s)\1_{[\tau,\tau+\ep]}(s)\Big\}\,\diff W_s
\end{split}
\end{equation}
and
\begin{equation}\label{MP_eq_vareq-SVE12-lift'}
\begin{split}
	Y^{1,2,v,\tau,\ep}_t(\theta)&=\int^t_0e^{-\theta(t-s)}M_b(\theta)\Big\{\hat{b}_x(s)X^{1,2,v,\tau,\ep}_s+\frac{1}{2}\big\langle\hat{b}_{xx}(s)X^{1,v,\tau,\ep}_s,X^{1,v,\tau,\ep}_s\big\rangle\\
	&\hspace{4.5cm}+\Big(\delta b^v(s)+\delta b^v_x(s)X^{1,v,\tau,\ep}_s\Big)\1_{[\tau,\tau+\ep]}(s)\Big\}\,\diff s\\
	&\hspace{0.5cm}+\int^t_0e^{-\theta(t-s)}M_\sigma(\theta)\Big\{\hat{\sigma}_x(s)X^{1,2,v,\tau,\ep}_s+\frac{1}{2}\big\langle\hat{\sigma}_{xx}(s)X^{1,v,\tau,\ep}_s,X^{1,v,\tau,\ep}_s\big\rangle\\
	&\hspace{5cm}+\Big(\delta\sigma^v(s)+\delta\sigma^v_x(s)X^{1,v,\tau,\ep}_s\Big)\1_{[\tau,\tau+\ep]}(s)\Big\}\,\diff W_s.
\end{split}
\end{equation}
By the standard stability argument, one can show that there exist (jointly) progressively measurable versions of $Y^{1,v,\tau,\ep},Y^{1,2,v,\tau,\ep}:\Omega\times[0,T]\times\bR_+\to\bR^n$. In the following, we consider such versions.

As investigated in \cite{Ha24}, $Y^{1,v,\tau,\ep}$ and $Y^{1,2,v,\tau,\ep}$ can be seen as (infinite dimensional) \emph{lifts} of $X^{1,v,\tau,\ep}$ and $X^{1,2,v,\tau,\ep}$, respectively. Indeed, the following lemma shows that $X^{1,v,\tau,\ep}_t$ and $X^{1,2,v,\tau,\ep}_t$ are represented as the integrals of $Y^{1,v,\tau,\ep}_t(\theta)$ and $Y^{1,2,v,\tau,\ep}_t(\theta)$ with respect to $\mu(\diff\theta)$. In the following, for each $\mu$-integrable function $\psi$ on $\bR_+$ taking values in a Euclidean space, we denote
\begin{equation}\label{MP_eq_integral-map}
	\mu[\psi]:=\int_{\bR_+}\psi(\theta)\dmu.
\end{equation}

%% Lemma

\begin{lemm}\label{MP_lemm_vareq-lift}
Suppose that \cref{control_assum_coefficient} and \cref{MP_assum_kernel} hold. Fix $\hat{u},v\in\cU$, $\tau\in[0,T)$ and $\ep\in(0,T-\tau]$. Then, for any $p\in[2,\infty)$, there exists a constant $C_{p,\alpha}>0$, which does not depend on $\tau$ or $\ep$, such that
\begin{equation}\label{MP_eq_vareq-lift}
	\int_{\bR_+}(1+\theta)r(\theta)\Big\{\|Y^{1,v,\tau,\ep}(\theta)\|_{C^p_\bF(0,T)}^2+\|Y^{1,2,v,\tau,\ep}(\theta)\|_{C^p_\bF(0,T)}^2\Big\}\dmu\leq C_{p,\alpha}\ep^{1-\alpha}.
\end{equation}
In particular,
\begin{equation}\label{MP_eq_lift-L1}
	\sup_{t\in[0,T]}\bE\Big[\|Y^{1,v,\tau,\ep}_t\|^p_{L^1(\mu)}+\|Y^{1,2,v,\tau,\ep}_t\|^p_{L^1(\mu)}\Big]<\infty
\end{equation}
for any $p\in[2,\infty)$. Furthermore, it holds that
\begin{equation}\label{MP_eq_lift-Fubini}
	\mu[Y^{1,v,\tau,\ep}_t]=X^{1,v,\tau,\ep}_t\ \ \text{and}\ \ \mu[Y^{1,2,v,\tau,\ep}_t]=X^{1,2,v,\tau,\ep}_t\ \ \text{a.s.\ for any $t\in[0,T]$.}
\end{equation}
\end{lemm}

%% Proof

\begin{proof}
Let $p\in[2,\infty)$. We first show \eqref{MP_eq_vareq-lift}. Fix $\theta\in\bR_+$. Noting \eqref{MP_eq_vareq-SVE1-lift'} and \eqref{MP_eq_vareq-SVE12-lift'}, by using \cref{pre_lemm_convolution} (applying to the kernel $t\mapsto e^{-\theta t}$) and \cref{Taylor_prop_SVE-expansion}, we obtain
\begin{equation}\label{MP_eq_vareq-lift-estimate1}
\begin{split}
	&\|Y^{1,v,\tau,\ep}(\theta)\|_{C^p_\bF(0,T)}+\|Y^{1,2,v,\tau,\ep}(\theta)\|_{C^p_\bF(0,T)}\\
	&\leq C_p\big\{\|e^{-\theta\cdot}\|_{L^1(0,T)}|M_b(\theta)|+\|e^{-\theta\cdot}\|_{L^2(0,T)}|M_\sigma(\theta)|\big\}\big\{\|K_b\|_{1,\ep}+\|K_\sigma\|_{2,\ep}\big\}\\
	&\hspace{1cm}+C_p\|e^{-\theta\cdot}\|_{L^1(0,\ep)}|M_b(\theta)|+C_p\|e^{-\theta\cdot}\|_{L^2(0,\ep)}|M_\sigma(\theta)|,
\end{split}
\end{equation}
for some constant $C_p>0$, which does not depend on $\tau$ or $\ep$. Noting \cref{MP_lemm_kernel} and \cref{intro_rem_K-ep}, we see that
\begin{equation}\label{MP_eq_vareq-lift-estimate2}
	\|K_b\|_{1,\ep}^2\leq\|K_b\|_{\frac{2}{1+\alpha},\ep}^2\ep^{1-\alpha}\ \ \text{and}\ \ \|K_\sigma\|_{2,\ep}^2\leq\|K_\sigma\|_{\frac{2}{\alpha},\ep}^2\ep^{1-\alpha}.
\end{equation}
Furthermore, noting that $e^{-x}\leq1\wedge x^{-\beta}$ for any $x\geq0$ and $\beta\in(0,1)$, one has, for any $s\in[0,T]$,
\begin{align*}
	&\|e^{-\theta\cdot}\|_{L^1(0,s)}\leq s\wedge\left(\frac{2\theta^{-(1+\alpha)/2}s^{(1-\alpha)/2}}{1-\alpha}\right)\leq\left(s^{(1+\alpha)/2}\vee\frac{2}{1-\alpha}\right)(1\vee\theta)^{-(1+\alpha)/2}s^{(1-\alpha)/2},\\
	&\|e^{-\theta\cdot}\|_{L^2(0,s)}\leq\left(s\wedge\left(\frac{(2\theta)^{-\alpha}s^{1-\alpha}}{1-\alpha}\right)\right)^{1/2}\leq\left(s^\alpha\vee\frac{1}{2^\alpha(1-\alpha)}\right)^{1/2}(1\vee\theta)^{-\alpha/2}s^{(1-\alpha)/2}.
\end{align*}
Hence,
\begin{equation}\label{MP_eq_vareq-lift-estimate3}
\begin{split}
	&\int_{\bR_+}(1+\theta)r(\theta)\|e^{-\theta\cdot}\|_{L^1(0,s)}^2|M_b(\theta)|^2\dmu\leq C_\alpha\int_{\bR_+}(1+\theta)^{-\alpha}r(\theta)|M_b(\theta)|^2\dmu\,s^{1-\alpha},\\
	&\int_{\bR_+}(1+\theta)r(\theta)\|e^{-\theta\cdot}\|_{L^2(0,s)}^2|M_\sigma(\theta)|^2\dmu\leq C_\alpha\int_{\bR_+}(1+\theta)^{1-\alpha}r(\theta)|M_\sigma(\theta)|^2\dmu\,s^{1-\alpha},\\
	&\hspace{8cm}s\in[0,T],
\end{split}
\end{equation}
where $C_\alpha$ is a positive constant which depends only on $\alpha$ and $T$. Noting the integrability condition \eqref{MP_eq_M-integrable} in \cref{MP_assum_kernel}, by \eqref{MP_eq_vareq-lift-estimate1}, \eqref{MP_eq_vareq-lift-estimate2} and \eqref{MP_eq_vareq-lift-estimate3} (applying to $s=T$ and $s=\ep$), we obtain the desired estimate \eqref{MP_eq_vareq-lift}.
Furthermore, noting that $r(\theta)^{-1}\leq(1+\theta)r(\theta)$, the integrability condition \eqref{MP_eq_mu-integrable} in \cref{MP_assum_kernel} and the estimate \eqref{MP_eq_vareq-lift} yield that
\begin{align*}
	\sup_{t\in[0,T]}\bE\big[\|Y^{1,v,\tau,\ep}_t\|^p_{L^1(\mu)}\big]^{1/p}&\leq\int_{\bR_+}\|Y^{1,v,\tau,\ep}(\theta)\|_{C^p_\bF(0,T)}\dmu\\
	&\leq\left(\int_{\bR_+}r(\theta)\dmu\right)^{1/2}\left(\int_{\bR_+}r(\theta)^{-1}\|Y^{1,v,\tau,\ep}(\theta)\|_{C^p_\bF(0,T;\bR^n)}^2\dmu\right)^{1/2}<\infty.
\end{align*}
Similarly, we can show that $\sup_{t\in[0,T]}\bE[\|Y^{1,2,v,\tau,\ep}_t\|^p_{L^1(\mu)}]^{1/p}<\infty$. Hence, the estimate \eqref{MP_eq_lift-L1} holds.

Next, we show \eqref{MP_eq_lift-Fubini}. Fix $t\in[0,T]$. We integrate left- and right-hand sides of \eqref{MP_eq_vareq-SVE1-lift'} and \eqref{MP_eq_vareq-SVE12-lift'} with respect to $\mu(\diff\theta)$ and switch the orders of integrals by means of the (stochastic) Fubini theorem (cf.\ \cite[Theorem 2.2]{Ve11}). To do so, let us check the integrability conditions which ensures the applications of the stochastic Fubini theorem. Observe that, for any $b\in L^{1,\infty}_\bF(0,t;\bR^n)$,
\begin{align*}
	&\bE\left[\int_{\bR_+}\int^t_0e^{-\theta(t-s)}|M_b(\theta)||b_s|\,\diff s\dmu\right]\\
	&\leq\|b\|_{L^{1,\infty}_\bF(0,t)}\int_{\bR_+}\|e^{-\theta\cdot}\|_{L^1(0,t)}|M_b(\theta)|\dmu\\
	&\leq\|b\|_{L^{1,\infty}_\bF(0,t)}\left(\int_{\bR_+}r(\theta)\dmu\right)^{1/2}\left(\int_{\bR_+}r(\theta)^{-1}\|e^{-\theta\cdot}\|_{L^1(0,t)}^2|M_b(\theta)|^2\dmu\right)^{1/2},
\end{align*}
and the last line is finite thanks to the integrability condition \eqref{MP_eq_mu-integrable} in \cref{MP_assum_kernel} and \eqref{MP_eq_vareq-lift-estimate3}, together with $r(\theta)^{-1}\leq(1+\theta)r(\theta)$. Similarly, for any $\sigma\in L^{2,\infty}_\bF(0,t;\bR^n)$,
\begin{align*}
	&\bE\left[\int_{\bR_+}\left(\int^t_0e^{-2\theta(t-s)}|M_\sigma(\theta)|^2|\sigma_s|^2\,\diff s\right)^{1/2}\dmu\right]\\
	&\leq\|\sigma\|_{L^{2,\infty}_\bF(0,t)}\int_{\bR_+}\|e^{-\theta\cdot}\|_{L^2(0,t)}|M_\sigma(\theta)|\dmu\\
	&\leq\|\sigma\|_{L^{2,\infty}_\bF(0,t)}\left(\int_{\bR_+}r(\theta)\dmu\right)^{1/2}\left(\int_{\bR_+}r(\theta)^{-1}\|e^{-\theta\cdot}\|_{L^2(0,t)}^2|M_\sigma(\theta)|^2\dmu\right)^{1/2}<\infty.
\end{align*}
Thus, for such processes $b$ and $\sigma$, we can apply the (stochastic) Fubini theorem (cf.\ \cite{Ve11}), and the expressions \eqref{MP_eq_kernel} for $K_b$ and $K_\sigma$ yield that
\begin{equation*}
	\int_{\bR_+}\int^t_0e^{-\theta(t-s)}M_b(\theta)b_s\,\diff s\dmu=\int^t_0\int_{\bR_+}e^{-\theta(t-s)}M_b(\theta)\dmu\,b_s\,\diff s=\int^t_0K_b(t-s)b_s\,\diff s\ \ \text{a.s.}
\end{equation*}
and
\begin{equation*}
	\int_{\bR_+}\int^t_0e^{-\theta(t-s)}M_\sigma(\theta)\sigma_s\,\diff W_s\dmu=\int^t_0\int_{\bR_+}e^{-\theta(t-s)}M_\sigma(\theta)\dmu\,\sigma_s\,\diff W_s=\int^t_0K_\sigma(t-s)\sigma_s\,\diff W_s\ \ \text{a.s.}
\end{equation*}
Applying the above observations to \eqref{MP_eq_vareq-SVE1-lift'} and noting that $X^{1,v,\tau,\ep}$ solves the SVE \eqref{Taylor_eq_vareq-SVE1}, we get
\begin{align*}
	\mu[Y^{1,v,\tau,\ep}_t]&=\int_{\bR_+}Y^{1,v,\tau,\ep}_t(\theta)\dmu\\
	&=\int^t_0K_b(t-s)\Big\{\hat{b}_x(s)X^{1,v,\tau,\ep}_s+\delta b^v(s)\1_{[\tau,\tau+\ep]}(s)\Big\}\,\diff s\\
	&\hspace{1cm}+\int^t_0K_\sigma(t-s)\Big\{\hat{\sigma}_x(s)X^{1,v,\tau,\ep}_s+\delta\sigma^v(s)\1_{[\tau,\tau+\ep]}(s)\Big\}\,\diff W_s\\
	&=X^{1,v,\tau,\ep}_t\ \ \text{a.s.}
\end{align*}
Similarly, from the expression \eqref{MP_eq_vareq-SVE12-lift'} for $Y^{1,2,v,\tau,\ep}$ and the SVE \eqref{Taylor_eq_vareq-SVE12} for $X^{1,2,v,\tau,\ep}$, we obtain $\mu[Y^{1,2,v,\tau,\ep}_t]=X^{1,2,v,\tau,\ep}_t$ a.s. Thus, \eqref{MP_eq_lift-Fubini} holds. This completes the proof.
\end{proof}

%%%%%%
\subsection{Heuristic derivations of first and second order adjoint equations}\label{MP-heuristic}
%%%%%%

In this subsection, we heuristically demonstrate how to derive the adjoint equations. The formal calculations in this subsection will be rigorously proved in \cref{MP-main}. As before, suppose that \cref{control_assum_coefficient} and \cref{MP_assum_kernel} hold. Fix $\hat{u},v\in\cU$, $\tau\in[0,T)$ and $\ep\in(0,T-\tau]$. Let $Y^{1,v,\tau,\ep}$ and $Y^{1,2,v,\tau,\ep}$ be the lifts of $X^{1,v,\tau,\ep}$ and $X^{1,2,v,\tau,\ep}$, respectively, defined in the previous subsection. Our idea is to calculate the terms $\bE[\hat{h}_xX^{1,2,v,\tau,\ep}_T]$ and $\bE[\langle\hat{h}_{xx}X^{1,v,\tau,\ep}_T,X^{1,v,\tau,\ep}_T\rangle]$ appearing in $J^{1,2,v,\tau,\ep}$ by using semimartingale properties of $Y^{1,2,v,\tau,\ep}$ and $Y^{1,v,\tau,\ep}$, respectively.

First of all, noting \cref{MP_lemm_vareq-lift}, we have
\begin{equation}\label{MP_eq_terminal1}
	\bE\big[\hat{h}_xX^{1,2,v,\tau,\ep}_T]=\bE\big[\hat{h}_x\mu[Y^{1,2,v,\tau,\ep}_T]\big]=\int_{\bR_+}\bE\big[\hat{h}_xY^{1,2,v,\tau,\ep}_T(\theta)\big]\dmu
\end{equation}
and
\begin{equation}\label{MP_eq_terminal2}
\begin{split}
	\bE\Big[\big\langle\hat{h}_{xx}X^{1,v,\tau,\ep}_T,X^{1,v,\tau,\ep}_T\big\rangle\Big]&=\bE\Big[\big\langle\hat{h}_{xx}\mu[Y^{1,v,\tau,\ep}_T],\mu[Y^{1,v,\tau,\ep}_T]\big\rangle\Big]\\
	&=\int_{\bR_+}\int_{\bR_+}\bE\Big[\big\langle\hat{h}_{xx}Y^{1,v,\tau,\ep}_T(\theta_2),Y^{1,v,\tau,\ep}_T(\theta_1)\big\rangle\Big]\,\dmumu.
\end{split}
\end{equation}
Here, we used Fubini's theorem, which is justified by the integrability conditions \eqref{Taylor_eq_var-coefficient-cost} and \eqref{MP_eq_lift-L1}. We focus on the expectations $\bE\big[\hat{h}_xY^{1,2,v,\tau,\ep}_T(\theta)\big]$ and $\bE[\langle\hat{h}_{xx}Y^{1,v,\tau,\ep}_T(\theta_2),Y^{1,v,\tau,\ep}_T(\theta_1)\rangle]$ for fixed $\theta\in\bR_+$ and $(\theta_1,\theta_2)\in\bR_+^2$. Let $\bR^n$-valued process $(\hat{g}_t(\theta))_{t\in[0,T]}$ and $\bR^{n\times n}$-valued process $(\hat{G}_t(\theta_1,\theta_2))_{t\in[0,T]}$ be given, and assume that they are progressively measurable and satisfy
\begin{equation}\label{MP_eq_integrability-gG}
	\bE\left[\left(\int^T_0|\hat{g}_t(\theta)|\,\diff t\right)^2+\left(\int^T_0|\hat{G}_t(\theta_1,\theta_2)|\,\diff t\right)^2\right]<\infty.
\end{equation}
Consider the following backward stochastic differential equations (BSDEs):
\begin{equation}\label{MP_eq_BSDE1}
	\begin{dcases}
	\diff\hat{p}_t(\theta)=\theta\hat{p}_t(\theta)\,\diff t-\hat{g}_t(\theta)\,\diff t+\hat{q}_t(\theta)\,\diff W_t,\ \ t\in[0,T],\\
	\hat{p}_T(\theta)=-\hat{h}_x^\top,
	\end{dcases}
\end{equation}
and
\begin{equation}\label{MP_eq_BSDE2}
	\begin{dcases}
	\diff\hat{P}_t(\theta_1,\theta_2)=(\theta_1+\theta_2)\hat{P}_t(\theta_1,\theta_2)\,\diff t-\hat{G}_t(\theta_1,\theta_2)\,\diff t+\hat{Q}_t(\theta_1,\theta_2)\,\diff W_t,\ \ t\in[0,T],\\
	\hat{P}_T(\theta_1,\theta_2)=-\hat{h}_{xx}.
	\end{dcases}
\end{equation}
Noting that $\hat{h}_x^\top\in L^2_{\cF_T}(\Omega;\bR^n)$ and $\hat{h}_{xx}\in L^2_{\cF_T}(\Omega;\bR^{n\times n})$ (see \eqref{Taylor_eq_var-coefficient-cost}), by the well-known result on BSDEs (cf.\ \cite[Proposition 4.3.1]{Zh17}), we see that the above BSDEs admit unique adapted solutions $(\hat{p}(\theta),\hat{q}(\theta))$ and $(\hat{P}(\theta_1,\theta_2),\hat{Q}(\theta_1,\theta_2))$, which are $\bR^n$- and $\bR^{n\times n}$-valued respectively and satisfy
\begin{equation*}
	\bE\left[\sup_{t\in[0,T]}|\hat{p}_t(\theta)|^2+\int^T_0|\hat{q}_t(\theta)|^2\,\diff t+\sup_{t\in[0,T]}|\hat{P}_t(\theta_1,\theta_2)|^2+\int^T_0|\hat{Q}_t(\theta_1,\theta_2)|^2\,\diff t\right]<\infty.
\end{equation*}
For a while, as in the proof of \cref{Taylor_prop_SVE-expansion}, we sometimes omit the super scripts $v,\tau,\ep$. Noting that $Y^{1,2}(\theta)=Y^{1,2,v,\tau,\ep}(\theta)$ is an $\bR^n$-valued It\^{o} process which satisfies the SDE \eqref{MP_eq_vareq-SVE12-lift} (for fixed $\theta$), by using It\^{o}'s formula and recalling the notation \eqref{notation_eq_matrix}, we have
\begin{align}
	\nonumber
	&-\bE\big[\hat{h}_xY^{1,2}_T(\theta)\big]\\
	\nonumber
	&=\bE\Big[\big\langle\hat{p}_T(\theta),Y^{1,2}_T(\theta)\big\rangle-\big\langle\hat{p}_0(\theta),Y^{1,2}_0(\theta)\big\rangle\Big]\\
%	\nonumber
%	&=\bE\left[\int^T_0\Big\{\langle p_t(\theta),\hat{b}_x(t)X^{1,2}_t\rangle+\langle q_t(\theta),\hat{\sigma}_x(t)X^{1,2}_t\rangle-\langle g_t(\theta),Y^{1,2}_t(\theta)\rangle\right.\\
%	\nonumber
%	&\left.\hspace{1.5cm}+\langle p_t(\theta)\delta b(t)\rangle\1_{[\tau,\tau+\ep]}(t)+\frac{1}{2}\Big\langle p_t(\theta),\big\langle\hat{b}_{xx}(t)X^1_t,X^1_t\big\rangle\Big\rangle\vphantom{\int^T_0}\right.\\
%	\nonumber
%	&\left.\hspace{1.5cm}+\langle q_t(\theta),\delta\sigma(t)\rangle\1_{[\tau,\tau+\ep]}(t)+\langle q_t(\theta),\delta\sigma_x(t)X^1_t\rangle\1_{[\tau,\tau+\ep]}(t)+\frac{1}{2}\Big\langle q_t(\theta),\big\langle\hat{\sigma}_{xx}(t)X^1_t,X^1_t\big\rangle\Big\rangle\Big\}\diff t\vphantom{\int^T_0}\right]\\
	\nonumber
	&=\bE\left[\int^T_0\Big\{\big\langle\hat{b}_x(t)^\top M_b(\theta)^\top\hat{p}_t(\theta)+\hat{\sigma}_x(t)^\top M_\sigma(\theta)^\top\hat{q}_t(\theta),X^{1,2}_t\big\rangle-\big\langle\hat{g}_t(\theta),Y^{1,2}_t(\theta)\big\rangle\right.\\
	\nonumber
	&\left.\hspace{2cm}+\frac{1}{2}\Big\langle\Big(\big\langle M_b(\theta)^\top\hat{p}_t(\theta),\hat{b}_{xx}(t)\big\rangle+\big\langle M_\sigma(\theta)^\top\hat{q}_t(\theta),\hat{\sigma}_{xx}(t)\big\rangle\Big)X^1_t,X^1_t\Big\rangle\Big\}\,\diff t\vphantom{\int^T_0}\right]\\
	\label{MP_eq_duality1-theta}
	&\hspace{0.5cm}+\bE\left[\int^{\tau+\ep}_\tau\Big\{\big\langle M_b(\theta)^\top\hat{p}_t(\theta),\delta b(t)+\delta b_x(t)X^1_t\big\rangle+\big\langle M_\sigma(\theta)^\top\hat{q}_t(\theta),\delta\sigma(t)+\delta\sigma_x(t)X^1_t\big\rangle\Big\}\,\diff t\right].
\end{align}
Similarly, noting the SDE \eqref{MP_eq_vareq-SVE1-lift} for $Y^1(\theta_1)$ and $Y^1(\theta_2)$ (for fixed $\theta_1$ and $\theta_2$), by using It\^{o}'s formula, we have
\begin{align}
	\nonumber
	&-\bE\Big[\big\langle\hat{h}_{xx}Y^1_T(\theta_2),Y^1_T(\theta_1)\big\rangle\Big]\\
	\nonumber
	&=\bE\Big[\big\langle\hat{P}_T(\theta_1,\theta_2)Y^1_T(\theta_2),Y^1_T(\theta_1)\big\rangle-\big\langle\hat{P}_0(\theta_1,\theta_2)Y^1_0(\theta_2),Y^1_0(\theta_1)\big\rangle\Big]\\
%	\nonumber
%	&=\bE\Big[\int^T_0\Big\{\langle P_t(\theta_1,\theta_2)Y^1_t(\theta_1),\hat{b}_x(t)X^1_t\rangle+\langle P_t(\theta_1,\theta_2)\hat{b}_x(t)X^1_t,Y^1_t(\theta_2)\rangle\\
%	\nonumber
%	&\hspace{1.5cm}+\big\langle P_t(\theta_1,\theta_2)\big\{\hat{\sigma}_x(t)X^1_t+\delta\sigma(t)\1_{[\tau,\tau+\ep]}(t)\big\},\hat{\sigma}_x(t)X^1_t+\delta\sigma(t)\1_{[\tau,\tau+\ep]}(t)\big\rangle\\
%	\nonumber
%	&\hspace{1.5cm}+\big\langle Q_t(\theta_1,\theta_2)Y^1_t(\theta_1),\hat{\sigma}_x(t)X^1_t+\delta\sigma(t)\1_{[\tau,\tau+\ep]}(t)\big\rangle\\
%	\nonumber
%	&\hspace{1.5cm}+\big\langle Q_t(\theta_1,\theta_2)\big\{\hat{\sigma}_x(t)X^1_t+\delta\sigma(t)\1_{[\tau,\tau+\ep]}(t)\big\rangle\big\},Y^1_t(\theta_2)\big\rangle\\
%	\nonumber
%	&\hspace{1.5cm}-\langle G_t(\theta_1,\theta_2)Y^1_t(\theta_1),Y^1_t(\theta_2)\rangle\Big\}\,\diff t\Big]\\
	\nonumber
	&=\bE\left[\int^T_0\Big\{\big\langle\hat{b}_x(t)^\top M_b(\theta_1)^\top\hat{P}_t(\theta_1,\theta_2)Y^1_t(\theta_2),X^1_t\big\rangle+\big\langle\hat{P}_t(\theta_1,\theta_2)M_b(\theta_2)\hat{b}_x(t)X^1_t,Y^1_t(\theta_1)\big\rangle\right.\\
	\nonumber
	&\left.\hspace{2cm}+\big\langle\hat{\sigma}_x(t)^\top M_\sigma(\theta_1)^\top\hat{P}_t(\theta_1,\theta_2)M_\sigma(\theta_2)\hat{\sigma}_x(t)X^1_t,X^1_t\big\rangle\right.\\
	\nonumber
	&\left.\hspace{2cm}+\big\langle\hat{\sigma}_x(t)^\top M_\sigma(\theta_1)^\top\hat{Q}_t(\theta_1,\theta_2)Y^1_t(\theta_2),X^1_t\big\rangle+\big\langle \hat{Q}_t(\theta_1,\theta_2)M_\sigma(\theta_2)\hat{\sigma}_x(t)X^1_t,Y^1_t(\theta_1)\big\rangle\right.\\
	\nonumber
	&\left.\hspace{2cm}-\big\langle\hat{G}_t(\theta_1,\theta_2)Y^1_t(\theta_2),Y^1_t(\theta_1)\big\rangle\Big\}\,\diff t\vphantom{\int^T_0}\right]\\
	\nonumber
	&\hspace{0.5cm}+\bE\left[\int^{\tau+\ep}_\tau\Big\{\big\langle M_b(\theta_1)^\top\hat{P}_t(\theta_1,\theta_2)Y^1_t(\theta_2),\delta b(t)\big\rangle+\big\langle\hat{P}_t(\theta_1,\theta_2)M_b(\theta_2)\delta b(t),Y^1_t(\theta_1)\big\rangle\right.\\
	\nonumber
	&\left.\hspace{3cm}+\big\langle M_\sigma(\theta_1)^\top\hat{P}_t(\theta_1,\theta_2)M_\sigma(\theta_2)\delta\sigma(t),\delta\sigma(t)\big\rangle\right.\\
	\nonumber
	&\left.\hspace{3cm}+\big\langle M_\sigma(\theta_1)^\top\hat{P}_t(\theta_1,\theta_2)M_\sigma(\theta_2)\hat{\sigma}_x(t)X^1_t,\delta\sigma(t)\big\rangle\right.\\
	\nonumber
	&\left.\hspace{3cm}+\big\langle M_\sigma(\theta_1)^\top\hat{P}_t(\theta_1,\theta_2)M_\sigma(\theta_2)\delta\sigma(t),\hat{\sigma}_x(t)X^1_t\big\rangle\right.\\
	\label{MP_eq_duality2-theta}
	&\left.\hspace{3cm}+\big\langle M_\sigma(\theta_1)^\top\hat{Q}_t(\theta_1,\theta_2)Y^1_t(\theta_2),\delta\sigma(t)\big\rangle+\big\langle\hat{Q}_t(\theta_1,\theta_2)M_\sigma(\theta_2)\delta\sigma(t),Y^1_t(\theta_1)\big\rangle\Big\}\,\diff t\vphantom{\int^{\tau+\ep}_\tau}\right].
\end{align}
Next, we integrate both sides of \eqref{MP_eq_duality1-theta} and \eqref{MP_eq_duality2-theta} with respect to $\dmu$ and $\dmumu$, respectively. Heuristically, we assume that the maps $(\omega,t,\theta)\mapsto\hat{g}_t(\theta),\hat{p}_t(\theta),\hat{q}_t(\theta)$ and $(\omega,t,\theta_1,\theta_2)\mapsto\hat{G}_t(\theta_1,\theta_2),\hat{P}_t(\theta_1,\theta_2),\hat{Q}_t(\theta_1,\theta_2)$ are jointly measurable and satisfy suitable integrability conditions that enable us to switch the orders of integrations for $\dmu\otimes\diff\bP\otimes\diff t$ and $\mu(\diff\theta_1)\otimes\mu(\diff\theta_2)\otimes\diff\bP\otimes\diff t$ by means of Fubini's theorem. Recall the notation \eqref{MP_eq_integral-map} of the integral map $\mu[\cdot]$. From \eqref{MP_eq_duality1-theta}, we get (at least formally)
\begin{align}
	\nonumber
	&-\int_{\bR_+}\bE\big[\hat{h}_xY^{1,2}_T(\theta)\big]\dmu\\
	\nonumber
	&=\bE\left[\int^T_0\left\{\big\langle\hat{b}_x(t)^\top\mu[M_b^\top\hat{p}_t]+\hat{\sigma}_x(t)^\top\mu[M_\sigma^\top\hat{q}_t],X^{1,2}_t\big\rangle-\int_{\bR_+}\big\langle\hat{g}_t(\theta),Y^{1,2}_t(\theta)\big\rangle\dmu\right.\right.\\
	\nonumber
	&\left.\left.\hspace{2cm}+\frac{1}{2}\Big\langle\Big(\big\langle\mu[M_b^\top\hat{p}_t],\hat{b}_{xx}(t)\big\rangle+\big\langle\mu[M_\sigma^\top\hat{q}_t],\hat{\sigma}_{xx}(t)\big\rangle\Big)X^1_t,X^1_t\Big\rangle\vphantom{\int_{\bR_+}}\right\}\,\diff t\vphantom{\int^T_0}\right]\\
	\label{MP_eq_duality1-mu}
	&\hspace{0.5cm}+\bE\left[\int^{\tau+\ep}_\tau\Big\{\big\langle\mu[M_b^\top\hat{p}_t],\delta b(t)+\delta b_x(t)X^1_t\big\rangle+\big\langle\mu[M_\sigma^\top \hat{q}_t],\delta\sigma(t)+\delta\sigma_x(t)X^1_t\big\rangle\Big\}\,\diff t\right].
\end{align}
Also, from \eqref{MP_eq_duality2-theta}, we get (at least formally)
\begin{align}
	\nonumber
	&-\int_{\bR_+}\int_{\bR_+}\bE\Big[\big\langle\hat{h}_{xx}Y^1_T(\theta_2),Y^1_T(\theta_1)\big\rangle\Big]\dmumu\\
	\nonumber
	&=\bE\left[\int^T_0\left\{\int_{\bR_+}\big\langle\hat{b}_x(t)^\top\mu[M_b^\top\hat{P}_t(\cdot,\theta_2)]Y^1_t(\theta_2),X^1_t\big\rangle\,\mu(\diff\theta_2)+\int_{\bR_+}\big\langle\mu[\hat{P}_t(\theta_1,\cdot)M_b]\hat{b}_x(t)X^1_t,Y^1_t(\theta_1)\big\rangle\,\mu(\diff\theta_1)\right.\right.\\
	\nonumber
	&\left.\left.\hspace{1.5cm}+\big\langle\hat{\sigma}_x(t)^\top\mu^{\otimes 2}[M_\sigma^\top\hat{P}_tM_\sigma]\hat{\sigma}_x(t)X^1_t,X^1_t\big\rangle\right.\right.\\
	\nonumber
	&\left.\left.\hspace{1.5cm}+\int_{\bR_+}\big\langle\hat{\sigma}_x(t)^\top\mu[M_\sigma^\top\hat{Q}_t(\cdot,\theta_2)]Y^1_t(\theta_2),X^1_t\big\rangle\,\mu(\diff\theta_2)+\int_{\bR_+}\big\langle\mu[\hat{Q}_t(\theta_1,\cdot)M_\sigma]\hat{\sigma}_x(t)X^1_t,Y^1_t(\theta_1)\big\rangle\,\mu(\diff\theta_1)\right.\right.\\
	\nonumber
	&\left.\left.\hspace{1.5cm}-\int_{\bR_+}\int_{\bR_+}\big\langle\hat{G}_t(\theta_1,\theta_2)Y^1_t(\theta_2),Y^1_t(\theta_1)\big\rangle\,\dmumu\right\}\,\diff t\right]\\
	\nonumber
	&\hspace{0.5cm}+\bE\left[\int^{\tau+\ep}_\tau\left\{\int_{\bR_+}\big\langle\mu[M_b^\top\hat{P}_t(\cdot,\theta_2)]Y^1_t(\theta_2),\delta b(t)\big\rangle\,\mu(\diff\theta_2)+\int_{\bR_+}\big\langle\mu[\hat{P}_t(\theta_1,\cdot)M_b]\delta b(t),Y^1_t(\theta_1)\big\rangle\,\mu(\diff\theta_1)\right.\right.\\
	\nonumber
	&\left.\left.\hspace{1.5cm}+\big\langle\mu^{\otimes 2}[M_\sigma^\top\hat{P}_tM_\sigma]\delta\sigma(t),\delta\sigma(t)\big\rangle\right.\right.\\
	\nonumber
	&\left.\left.\hspace{1.5cm}+\big\langle\mu^{\otimes 2}[M_\sigma^\top\hat{P}_tM_\sigma]\hat{\sigma}_x(t)X^1_t,\delta\sigma(t)\big\rangle+\big\langle\mu^{\otimes 2}[M_\sigma^\top\hat{P}_tM_\sigma]\delta\sigma(t),\hat{\sigma}_x(t)X^1_t\big\rangle\right.\right.\\
	\label{MP_eq_duality2-mu}
	&\left.\left.\hspace{1.5cm}+\int_{\bR_+}\big\langle\mu[M_\sigma^\top\hat{Q}_t(\cdot,\theta_2)]Y^1_t(\theta_2),\delta\sigma(t)\big\rangle\,\mu(\diff\theta_2)+\int_{\bR_+}\big\langle\mu[\hat{Q}_t(\theta_1,\cdot)M_\sigma]\delta\sigma(t),Y^1_t(\theta_1)\big\rangle\,\mu(\diff\theta_1)\right\}\,\diff t\vphantom{\int^T_0}\right],
\end{align}
where
\begin{align*}
	&\mu[M_b^\top\hat{P}_t(\cdot,\theta_2)]:=\int_{\bR_+}M_b^\top(\theta_1)\hat{P}_t(\theta_1,\theta_2)\,\mu(\diff\theta_1),\ \mu[\hat{P}_t(\theta_1,\cdot)M_b]:=\int_{\bR_+}\hat{P}_t(\theta_1,\theta_2)M_b(\theta_2)\,\mu(\diff\theta_2),\\
	&\mu^{\otimes2}[M_\sigma^\top\hat{P}_tM_\sigma]:=\int_{\bR_+}\int_{\bR_+}M_\sigma(\theta_1)^\top\hat{P}_t(\theta_1,\theta_2)M_\sigma(\theta_2)\dmumu,
\end{align*}
and the terms $\mu[M_\sigma^\top\hat{Q}_t(\cdot,\theta_2)]$ and $\mu[\hat{Q}_t(\theta_1,\cdot)M_\sigma]$ are defined similarly. The above calculations will be proved rigorously after deriving adjoint equations and showing their well-posedness; see \cref{MP_prop_J12-small} below.

By \eqref{MP_eq_terminal1}, \eqref{MP_eq_terminal2}, \eqref{MP_eq_duality1-mu} and \eqref{MP_eq_duality2-mu}, the term $J^{1,2}=J^{1,2,v,\tau,\ep}$ defined by \eqref{Taylor_eq_vareq-cost} is expressed as follows:
\begin{align}
	\nonumber
	&J^{1,2}\\
	\nonumber
	&=-\bE\left[\int^T_0\left\{\big\langle\hat{b}_x(t)^\top\mu[M_b^\top\hat{p}_t]+\hat{\sigma}_x(t)^\top\mu[M_\sigma^\top\hat{q}_t]-\hat{f}_x(t)^\top,X^{1,2}_t\big\rangle-\int_{\bR_+}\big\langle\hat{g}_t(\theta),Y^{1,2}_t(\theta)\big\rangle\dmu\right\}\,\diff t\right]\\
	\nonumber
	&\hspace{0.3cm}-\frac{1}{2}\bE\left[\int^T_0\left\{\int_{\bR_+}\big\langle\hat{b}_x(t)^\top\mu[M_b^\top\hat{P}_t(\cdot,\theta_2)]Y^1_t(\theta_2),X^1_t\big\rangle\,\mu(\diff\theta_2)+\int_{\bR_+}\big\langle\mu[\hat{P}_t(\theta_1,\cdot)M_b]\hat{b}_x(t)X^1_t,Y^1_t(\theta_1)\big\rangle\,\mu(\diff\theta_1)\right.\right.\\
	\nonumber
	&\left.\left.\hspace{1.5cm}+\big\langle\hat{\sigma}_x(t)^\top\mu^{\otimes 2}[M_\sigma^\top\hat{P}_tM_\sigma]\hat{\sigma}_x(t)X^1_t,X^1_t\big\rangle\right.\right.\\
	\nonumber
	&\left.\left.\hspace{1.5cm}+\int_{\bR_+}\big\langle\hat{\sigma}_x(t)^\top\mu[M_\sigma^\top\hat{Q}_t(\cdot,\theta_2)]Y^1_t(\theta_2),X^1_t\big\rangle\,\mu(\diff\theta_2)+\int_{\bR_+}\big\langle\mu[\hat{Q}_t(\theta_1,\cdot)M_\sigma]\hat{\sigma}_x(t)X^1_t,Y^1_t(\theta_1)\big\rangle\,\mu(\diff\theta_1)\right.\right.\\
	\nonumber
	&\left.\left.\hspace{1.5cm}+\Big\langle\Big(\big\langle\mu[M_b^\top\hat{p}_t],\hat{b}_{xx}(t)\big\rangle+\big\langle\mu[M_\sigma^\top\hat{q}_t],\hat{\sigma}_{xx}(t)\big\rangle-\hat{f}_{xx}(t)\Big)X^1_t,X^1_t\Big\rangle\right.\right.\\
	\nonumber
	&\left.\left.\hspace{1.5cm}-\int_{\bR_+}\int_{\bR_+}\big\langle\hat{G}_t(\theta_1,\theta_2)Y^1_t(\theta_2),Y^1_t(\theta_1)\big\rangle\,\dmumu\right\}\,\diff t\right]\\
	\nonumber
	&\hspace{0.3cm}-\bE\left[\int^{\tau+\ep}_\tau\left\{\big\langle\mu[M_b^\top\hat{p}_t],\delta b(t)\big\rangle+\big\langle\mu[M_\sigma^\top\hat{q}_t],\delta\sigma(t)\big\rangle-\delta f(t)+\frac{1}{2}\big\langle\mu^{\otimes2}[M_\sigma^\top\hat{P}_tM_\sigma]\delta\sigma(t),\delta\sigma(t)\big\rangle\right\}\,\diff t\right]\\
	\nonumber
	&\hspace{0.3cm}-\bE\left[\int^{\tau+\ep}_\tau\left\{\big\langle\mu[M_b^\top\hat{p}_t],\delta b_x(t)X^1_t\big\rangle+\big\langle\mu[M_\sigma^\top\hat{q}_t],\delta\sigma_x(t)X^1_t\big\rangle\right.\right.\\
	\nonumber
	&\left.\left.\hspace{1.5cm}+\frac{1}{2}\int_{\bR_+}\big\langle\mu[M_b^\top\hat{P}_t(\cdot,\theta_2)]Y^1_t(\theta_2),\delta b(t)\big\rangle\,\mu(\diff\theta_2)+\frac{1}{2}\int_{\bR_+}\big\langle\mu[\hat{P}_t(\theta_1,\cdot)M_b]\delta b(t),Y^1_t(\theta_1)\big\rangle\,\mu(\diff\theta_1)\right.\right.\\
	\nonumber
	&\left.\left.\hspace{1.5cm}+\frac{1}{2}\big\langle\mu^{\otimes 2}[M_\sigma^\top\hat{P}_tM_\sigma]\hat{\sigma}_x(t)X^1_t,\delta\sigma(t)\big\rangle+\frac{1}{2}\big\langle\mu^{\otimes2}[M_\sigma^\top\hat{P}_tM_\sigma]\delta\sigma(t),\hat{\sigma}_x(t)X^1_t\big\rangle\right.\right.\\
	\label{MP_eq_vareq-cost'}
	&\left.\left.\hspace{1.5cm}+\frac{1}{2}\int_{\bR_+}\big\langle\mu[M_\sigma^\top\hat{Q}_t(\cdot,\theta_2)]Y^1_t(\theta_2),\delta\sigma(t)\big\rangle\,\mu(\diff\theta_2)+\frac{1}{2}\int_{\bR_+}\big\langle\mu[\hat{Q}_t(\theta_1,\cdot)M_\sigma]\delta\sigma(t),Y^1_t(\theta_1)\big\rangle\,\mu(\diff\theta_1)\right\}\,\diff t\right].
\end{align}
Notice that the third expectation in the right-hand side of \eqref{MP_eq_vareq-cost'} does not involve $X^{1,v,\tau,\ep}$ or $X^{1,2,v,\tau,\ep}$. The last expectation is expected to be of order $o(\ep)$ (which will be shown after introducing precise forms of adjoint equations and proving their well-posedness; see \cref{MP_prop_J12-small} below). Moreover, noting that $\mu[Y^1_t]=X^1_t$ and $\mu[Y^{1,2}_t]=X^{1,2}_t$ a.s.\ for any $t\in[0,T]$ by \cref{MP_lemm_vareq-lift}, the first and second expectations would vanish if $\hat{g}$ and $\hat{G}$ were of the following form:
\begin{equation}\label{MP_eq_generator-g}
	\hat{g}_t(\theta)=\hat{b}_x(t)^\top\mu[M_b^\top\hat{p}_t]+\hat{\sigma}_x(t)^\top\mu[M_\sigma^\top\hat{q}_t]-\hat{f}_x(t)^\top
\end{equation}
and
\begin{equation}\label{MP_eq_generator-G}
\begin{split}
	\hat{G}_t(\theta_1,\theta_2)&=\hat{b}_x(t)^\top\mu[M_b^\top\hat{P}_t(\cdot,\theta_2)]+\mu[\hat{P}_t(\theta_1,\cdot)M_b]\hat{b}_x(t)+\hat{\sigma}_x(t)^\top\mu^{\otimes2}[M_\sigma^\top\hat{P}_tM_\sigma]\hat{\sigma}_x(t)\\
	&\hspace{0.5cm}+\hat{\sigma}_x(t)^\top\mu[M_\sigma^\top\hat{Q}_t(\cdot,\theta_2)]+\mu[\hat{Q}_t(\theta_1,\cdot)M_\sigma]\hat{\sigma}_x(t)\\
	&\hspace{0.5cm}+\langle\mu[M_b^\top\hat{p}_t],\hat{b}_{xx}(t)\rangle+\langle\mu[M_\sigma^\top\hat{q}_t],\hat{\sigma}_{xx}(t)\rangle-\hat{f}_{xx}(t).
\end{split}
\end{equation}
Inserting the above expressions to the BSDEs \eqref{MP_eq_BSDE1} and \eqref{MP_eq_BSDE2}, we get the \emph{first order adjoint equation}
\begin{equation}\label{MP_eq_adeq1}
	\begin{dcases}
	\diff\hat{p}_t(\theta)=\theta\hat{p}_t(\theta)\,\diff t-\Big\{\hat{b}_x(t)^\top\mu[M_b^\top\hat{p}_t]+\hat{\sigma}_x(t)^\top\mu[M_\sigma^\top\hat{q}_t]-\hat{f}_x(t)^\top\Big\}\,\diff t+\hat{q}_t(\theta)\,\diff W_t,\ \ \theta\in\bR_+,\ t\in[0,T],\\
	\hat{p}_T(\theta)=-\hat{h}_x^\top,\ \ \theta\in\bR_+,
	\end{dcases}
\end{equation}
and the \emph{second order adjoint equation}
\begin{equation}\label{MP_eq_adeq2}
	\begin{dcases}
	\diff\hat{P}_t(\theta_1,\theta_2)=(\theta_1+\theta_2)\hat{P}_t(\theta_1,\theta_2)\,\diff t\\
	\hspace{2cm}-\Big\{\hat{b}_x(t)^\top\mu[M_b^\top\hat{P}_t(\cdot,\theta_2)]+\mu[\hat{P}_t(\theta_1,\cdot)M_b]\hat{b}_x(t)+\hat{\sigma}_x(t)^\top\mu^{\otimes2}[M_\sigma^\top\hat{P}_tM_\sigma]\hat{\sigma}_x(t)\\
	\hspace{3.5cm}+\hat{\sigma}_x(t)^\top\mu[M_\sigma^\top\hat{Q}_t(\cdot,\theta_2)]+\mu[\hat{Q}_t(\theta_1,\cdot)M_\sigma]\hat{\sigma}_x(t)\\
	\hspace{3.5cm}+\big\langle\mu[M_b^\top\hat{p}_t],\hat{b}_{xx}(t)\big\rangle+\big\langle\mu[M_\sigma^\top\hat{q}_t],\hat{\sigma}_{xx}(t)\big\rangle-\hat{f}_{xx}(t)\Big\}\,\diff t\\
	\hspace{2cm}+\hat{Q}_t(\theta_1,\theta_2)\,\diff W_t,\ \ (\theta_1,\theta_2)\in\bR_+^2,\ t\in[0,T],\\
	\hat{P}_T(\theta_1,\theta_2)=-\hat{h}_{xx},\ \ (\theta_1,\theta_2)\in\bR_+^2.
	\end{dcases}
\end{equation}
Notice that the above equations depend only on the fixed control process $\hat{u}\in\cU$ and is independent of the choices of $v\in\cU$, $\tau\in[0,T)$ and $\ep\in(0,T-\tau]$.

Summarizing the above heuristic arguments, by using the solutions $(\hat{p},\hat{q})$ and $(\hat{P},\hat{Q})$ of the first and second order adjoint equations \eqref{MP_eq_adeq1} and \eqref{MP_eq_adeq2}, the term $J^{1,2}=J^{1,2,v,\tau,\ep}$ is written as
\begin{equation}\label{MP_eq_vareq-cost''}
	J^{1,2}=-\bE\left[\int^{\tau+\ep}_\tau\left\{\big\langle\mu[M_b^\top\hat{p}_t],\delta b(t)\big\rangle+\big\langle\mu[M_\sigma^\top\hat{q}_t],\delta\sigma(t)\big\rangle-\delta f(t)+\frac{1}{2}\big\langle\mu^{\otimes2}[M_\sigma^\top\hat{P}_tM_\sigma]\delta\sigma(t),\delta\sigma(t)\big\rangle\right\}\,\diff t\right]-\cE,
\end{equation}
where $\cE=\cE^{v,\tau,\ep}$ is the last expectation in the right-hand side of \eqref{MP_eq_vareq-cost'} (defined through the solutions $(\hat{p},\hat{q})$ and $(\hat{P},\hat{Q})$ of the adjoint equations), which is expected to be of order $o(\ep)$. From this, together with \cref{Taylor_prop_SVE-expansion}, one can obtain the global maximum principle. In order to rigorously prove the heuristic arguments in this subsection, we need to investigate the following three points:
\begin{itemize}
\item
Define appropriate notions of the ``solutions'' of the first and second order adjoint equations \eqref{MP_eq_adeq1} and \eqref{MP_eq_adeq2}, and prove their well-posedness.
\item
Justify the applications of Fubini's theorem deriving \eqref{MP_eq_duality1-mu} and \eqref{MP_eq_duality2-mu} from \eqref{MP_eq_duality1-theta} and \eqref{MP_eq_duality2-theta}, respectively.
\item
Show that the remainder term $\cE=\cE^{v,\tau,\ep}$, the last expectation in the right-hand side of \eqref{MP_eq_vareq-cost'}, is of order $o(\ep)$ as $\ep\downarrow0$.
\end{itemize}

The first and second order adjoint equations \eqref{MP_eq_adeq1} and \eqref{MP_eq_adeq2} are systems of infinitely many finite dimensional BSDEs for $(\hat{p}(\theta),\hat{q}(\theta))$ and $(\hat{P}(\theta_1,\theta_2),\hat{Q}(\theta_1,\theta_2))$ with parameters $\theta\in\bR_+$ and $(\theta_1,\theta_2)\in\bR_+^2$, respectively, which are coupled via the integral terms with respect to $\mu$. In another perspective, they can be regarded as infinite dimensional \emph{backward stochastic evolution equations} (BSEEs) defined on certain spaces of (univariate and two variate, respectively) functions. In \cref{MP-adeq}, we formulate a general framework of BSEEs and show the well-posedness of the first and second order adjoint equations \eqref{MP_eq_adeq1} and \eqref{MP_eq_adeq2} in a unified manner; see also \cref{BSEE} for the proof of the well-posedness of the general BSEE. After that, in \cref{MP-main}, we rigorously justify the heuristic arguments developed in this subsection and show the global maximum principle.

%%%%%%
\subsection{Well-posedness of the adjoint equations}\label{MP-adeq}
%%%%%%

In this subsection, we show the well-posedness of the first and second order adjoint equations \eqref{MP_eq_adeq1} and \eqref{MP_eq_adeq2}. If the measure $\mu$ is finite and if the maps $M_b$ and $M_\sigma$ are bounded (that correspond to the case where the kernels $K_b$ and $K_\sigma$ are regular), then \eqref{MP_eq_adeq1} and \eqref{MP_eq_adeq2} can be regarded as BSEEs on $L^2(\mu;\bR^n)$ and $L^2(\mu^{\otimes2};\bR^{n\times n})$, respectively. Hence, the general results on BSEEs on a (single) Hilbert space (cf.\ \cite{HuPe91}) can be applied to them. However, this is not the case for the singular kernels in \cref{MP_assum_kernel}. Since the measure $\mu$ is typically an infinite measure, the terminal condition $-\hat{h}_x^\top$ or $-\hat{h}_{xx}$ does no longer belong to $L^p(\mu;\bR^n)$ or $L^p(\mu^{\otimes2};\bR^{n\times n})$ for each $p\in[1,\infty)$. Also, there is a difficulty to deal with the well-posedness and continuity (in the operator sense) of the integral maps appearing in the generators of \eqref{MP_eq_adeq1} and \eqref{MP_eq_adeq2}. Our idea to overcome these difficulties in the singular kernel case is to formulate the BSEEs on (a family of) weighted $L^2$ spaces in the spirit of the author's previous work \cite{Ha24}, where the infinite dimensional Markovian lifts of SVEs are investigated.

Now let us introduce a suitable framework of BSEEs on weighted $L^2$ spaces, that contains the first and second order adjoint equations \eqref{MP_eq_adeq1} and \eqref{MP_eq_adeq2} as special cases. Let $(\Theta,\Sigma,\nu)$ be a countably generated $\sigma$-finite measure space, that is, $\Sigma$ is a countably generated $\sigma$-algebra on a (non-empty) set $\Theta$, and $\nu$ is a $\sigma$-finite measure on $(\Theta,\Sigma)$. Let $\varpi:\Theta\to\bR_+$ be a nonnegative $\Sigma$-measurable function, and let $E$ be a Euclidean space. For each $\beta\in\bR$, we denote by $\bH_\beta=\bH_\beta(\Theta,\Sigma,\nu,\varpi;E)$ the set of the $\nu$-equivalence classes of $\Sigma$-measurable functions $\psi:\Theta\to E$ such that
\begin{equation}\label{BSEE_eq_norm}
	\|\psi\|_{\bH_\beta}:=\left(\int_\Theta(1+\varpi(\vth))^\beta|\psi(\vth)|^2\,\nu(\diff\vth)\right)^{1/2}<\infty.
\end{equation}
For each $\psi_1,\psi_2\in\bH_\beta$, define
\begin{equation*}
	\langle\psi_1,\psi_2\rangle_{\bH_\beta}:=\int_\Theta(1+\varpi(\vth))^\beta\langle\psi_1(\vth),\psi_2(\vth)\rangle\,\nu(\diff\vth).
\end{equation*}
Then, $(\bH_\beta,\|\cdot\|_{\bH_\beta},\langle\cdot,\cdot\rangle_{\bH_\beta})$ is a separable Hilbert space\footnote{Notice that $\psi\mapsto(1+\varpi)^{\beta/2}\psi$ is an isometric isomorphism from $\bH_\beta$ to $L^2(\Theta,\Sigma,\nu;E)$. Since $\Sigma$ is countably generated and $\nu$ is $\sigma$-finite, the Hilbert space $L^2(\Theta,\Sigma,\nu;E)$ is separable, and hence $\bH_\beta$ is separable as well.}. We equip $\bH_\beta$ with the Borel $\sigma$-algebra $\cB(\bH_\beta)$. Clearly, we have the natural continuous embedding $\bH_\beta\hookrightarrow\bH_{\beta'}$ for $\beta\geq\beta'$. Notice also that $\bH_0=L^2(\Theta,\Sigma,\nu;E)$.

Let an $\cF_T$-measurable random variable $\Phi:\Omega\to\bH_{\beta_\Phi}$ and a progressively measurable map $\cG:\Omega\times[0,T]\times\bH_{\beta_\cP}\times\bH_{\beta_\cQ}\to\bH_{\beta_\cG}$ be given, where $\beta_\Phi,\beta_\cG,\beta_\cP,\beta_\cQ\in\bR$ are given parameters. We consider the following BSEE on the weighted $L^2$-space:
\begin{equation}\label{BSEE_eq_BSEE}
	\begin{dcases}
	\diff\cP_t(\vth)=\varpi(\vth)\cP_t(\vth)\,\diff t-\cG(t,\cP_t,\cQ_t)(\vth)\,\diff t+\cQ_t(\vth)\,\diff W_t,\ \ \vth\in\Theta,\ t\in[0,T],\\
	\cP_T(\vth)=\Phi(\vth),\ \ \vth\in\Theta.
	\end{dcases}
\end{equation}
The above is a formal expression. The precise notion of the solution is defined as follows.

%% Definition

\begin{defi}\label{BSEE_defi_solution}
Under the above setting, we say that a pair $(\cP,\cQ)$ of progressively measurable processes $\cP:\Omega\times[0,T]\to\bH_{\beta_\cP}$ and $\cQ:\Omega\times[0,T]\to\bH_{\beta_\cQ}$ is a \emph{solution} to the BSEE \eqref{BSEE_eq_BSEE} if there exist (jointly) progressively measurable maps $\bar{p},\bar{q},\bar{g}:\Omega\times[0,T]\times\Theta\to E$ such that the following hold:
\begin{itemize}
\item
$[\bar{p}_t(\cdot)]=\cP_t$, $[\bar{q}_t(\cdot)]=\cQ_t$, and $[\bar{g}_t(\cdot)]=\cG(t,\cP_t,\cQ_t)$ a.s. for a.e.\ $t\in[0,T]$, as well as $[\bar{p}_T(\cdot)]=\Phi$ a.s., where $[\psi(\cdot)]$ denotes the $\nu$-equivalence class of a $\Sigma$-measurable map $\psi:\Theta\to E$;
\item
for each fixed $\vth\in\Theta$, it holds that
\begin{equation}\label{BSEE_eq_L^2-solution}
	\bE\left[\sup_{t\in[0,T]}|\bar{p}_t(\vth)|^2+\int^T_0|\bar{q}_t(\vth)|^2\,\diff t+\left(\int^T_0|\bar{g}_t(\vth)|\,\diff t\right)^2\right]<\infty,
\end{equation}
and the process $\bar{p}(\vth)=(\bar{p}_t(\vth))_{t\in[0,T]}$ is an $E$-valued continuous semimartingale satisfying
\begin{equation*}
	\diff \bar{p}_t(\vth)=\varpi(\vth)\bar{p}_t(\vth)\,\diff t-\bar{g}_t(\vth)\,\diff t+\bar{q}_t(\vth)\,\diff W_t,\ \ t\in[0,T],
\end{equation*}
in the usual It\^{o}'s sense.
\end{itemize}
We call a tuple $(\bar{p},\bar{q},\bar{g})$ satisfying the above properties a \emph{good representative} for the solution $(\cP,\cQ)$.
\end{defi}

%% Remark

\begin{rem}
Notice that the process $\cP:\Omega\times[0,T]\to\bH_{\beta_\cP}$ itself is not necessarily an $\bH_{\beta_\cP}$-valued semimartingale.
\end{rem}

We summarize the above setup and impose the following assumption on the coefficients $\Phi$ and $\cG$.

%% Assumption

\begin{assum}\label{BSEE_assum_coefficient}
Let $(\Theta,\Sigma,\nu)$ be a countably generated $\sigma$-finite measure space. Let $\varpi:\Theta\to\bR_+$ be a nonnegative $\Sigma$-measurable function, and let $E$ be a Euclidean space. Fix $\alpha\in[0,1)$. Let $\Phi:\Omega\to\bH_0$ be an $\cF_T$-measurable random variable and $\cG:\Omega\times[0,T]\times\bH_{1+\alpha}\times\bH_\alpha\to\bH_0$ be a progressively measurable map. Assume that
\begin{equation}\label{BSEE_eq_BSEE-generator}
	\bE\left[\|\Phi\|^2_{\bH_0}+\int^T_0(T-t)^\alpha\|\cG(t,0,0)\|^2_{\bH_0}\,\diff t\right]<\infty.
\end{equation}
Furthermore, assume that there exists a constant $L>0$ such that
\begin{equation}\label{BSEE_eq_BSEE-Lip}
	\|\cG(t,\cP,\cQ)-\cG(t,\cP',\cQ')\|_{\bH_0}\leq L\big\{\|\cP-\cP'\|_{\bH_{1+\alpha}}+\|\cQ-\cQ'\|_{\bH_\alpha}\big\}
\end{equation}
for any $(\cP,\cQ),(\cP',\cQ')\in\bH_{1+\alpha}\times\bH_\alpha$ and $(\omega,t)\in\Omega\times[0,T]$.
\end{assum}

%% Remark

\begin{rem}
As we will see below, the parameter $\alpha\in[0,1)$ in \cref{BSEE_assum_coefficient} corresponds to the one in \cref{MP_assum_kernel}. We are in particular interested in the ``singular case'' $\alpha\in(0,1)$.
\end{rem}

The following is the main result in this subsection. The proof is postponed to \cref{BSEE}.

%% Theorem

\begin{theo}\label{BSEE_theo_BSEE}
Let \cref{BSEE_assum_coefficient} hold. Then, the BSEE \eqref{BSEE_eq_BSEE} has a solution $(\cP,\cQ):\Omega\times[0,T]\to\bH_{1+\alpha}\times\bH_\alpha$ such that
\begin{equation}\label{BSEE_eq_BSEE-estimate}
\begin{split}
	&\bE\left[\esssup_{t\in[0,T]}\|\cP_t\|^2_{\bH_0}+\int^T_0\|\cP_t\|^2_{\bH_1}\,\diff t+\int^T_0\|\cQ_t\|^2_{\bH_0}\,\diff t\right.\\
	&\left.\hspace{0.5cm}+\esssup_{t\in[0,T]}(T-t)^\alpha\|\cP_t\|^2_{\bH_\alpha}+\int^T_0(T-t)^\alpha\|\cP_t\|^2_{\bH_{1+\alpha}}\,\diff t+\int^T_0(T-t)^\alpha\|\cQ_t\|^2_{\bH_\alpha}\,\diff t\right]\\
	&\leq C_{T,L,\alpha}\,\bE\left[\|\Phi\|^2_{\bH_0}+\int^T_0(T-t)^\alpha\|\cG(t,0,0)\|_{\bH_0}^2\,\diff t\right],
\end{split}
\end{equation}
where $C_{T,L,\alpha}>0$ is a constant depending only on $T$, $L$ and $\alpha$. In particular,
\begin{equation}\label{BSEE_eq_solution-integrability}
	\bE\left[\int^T_0(T-t)^\alpha\|\cG(t,\cP_t,\cQ_t)\|_{\bH_0}^2\,\diff t\right]<\infty.
\end{equation}
Furthermore, the solution is unique in the following sense: if $(\cP',\cQ'):\Omega\times[0,T]\to\bH_{1+\alpha}\times\bH_\alpha$ is another solution of the BSEE \eqref{BSEE_eq_BSEE} satisfying \eqref{BSEE_eq_solution-integrability} with $(\cP,\cQ)$ replaced by $(\cP',\cQ')$, then $\cP'_t=\cP_t$ in $\bH_{1+\alpha}$ and $\cQ'_t=\cQ_t$ in $\bH_\alpha$ a.s.\ for a.e.\ $t\in[0,T]$.
\end{theo}

%% Remark

\begin{rem}\label{BSEE_rem_Gelfand}
Notice the relations
\begin{equation*}
	\bH_{1+\alpha}\hookrightarrow\bH_1\hookrightarrow\bH_\alpha\hookrightarrow\bH_0=L^2(\Theta,\Sigma,\nu;E)\hookrightarrow\bH_{-1}.
\end{equation*}
The terminal random variable $\Phi$ and generator $\cG$ take values in the space $\bH_0$. If \cref{BSEE_assum_coefficient} holds with $\alpha=0$, then \eqref{BSEE_eq_BSEE} reduces to a BSEE defined on the Gelfand triplet $\bH_1\hookrightarrow\bH_0\hookrightarrow\bH_{-1}$, and the well-established result in \cite{Pe92} is applicable; see also \cite{HuPe91,HuTa18,LiZh20,MeSh13,Zh92} for further works on infinite dimensional BSEEs on Gelfand triplets. However, this is not the case when $\alpha\in(0,1)$. The difficulty in the ``singular case'' is that $(\cP,\cQ)\mapsto\cG(t,\cP,\cQ)$ is not necessarily continuous (or even not defined) as a map from $\bH_1\times\bH_0$ to $\bH_0$. In this case, the BSEE \eqref{BSEE_eq_BSEE} is beyond the existing frameworks of BSEEs on Gelfand triplets, and \cref{BSEE_theo_BSEE} is a quite non-trivial result.
\end{rem}

We apply \cref{BSEE_theo_BSEE} to the first and second order adjoint equations \eqref{MP_eq_adeq1} and \eqref{MP_eq_adeq2}. Let \cref{control_assum_coefficient} and \cref{MP_assum_kernel} hold. Fix $\hat{u}\in\cU$. The first order adjoint equation \eqref{MP_eq_adeq1} can be regarded as a BSEE \eqref{BSEE_eq_BSEE} with
\begin{equation*}
	(\Theta,\Sigma,\nu)=(\bR_+,\cB(\bR_+),r(\theta)\mu(\diff\theta)),\ \ \varpi(\theta)=\theta,\ \ E=\bR^n,
\end{equation*}
and
\begin{equation}\label{BSEE_eq_adeq1-coefficient}
	\Phi(\theta)=-\hat{h}_x^\top,\ \ \cG(t,p,q)(\theta)=\hat{b}_x(t)^\top\mu[M_b^\top p]+\hat{\sigma}_x(t)^\top\mu[M_\sigma^\top q]-\hat{f}_x(t)^\top.
\end{equation}
For each $\beta\in\bR$ and a Euclidean space $E$, we denote $\bH^1_\beta(E):=\bH_\beta(\bR_+,\cB(\bR_+),r(\theta)\mu(\diff\theta),\theta;E)$. We sometimes omit the dependency on $E$ when there is no ambiguity. By \eqref{MP_eq_mu-integrable}, the Euclidean space $E$ is continuously embedded into $\bH^1_0(E)=L^2(\bR_+,\cB(\bR_+),r(\theta)\mu(\diff\theta);E)$ by the natural injection $p\mapsto(\theta\mapsto p)$. Notice that the condition \eqref{MP_eq_M-integrable} means that $M_b\in\bH^1_{-\alpha}(\bR^{n\times n})$ and $M_\sigma\in\bH^1_{1-\alpha}(\bR^{n\times n})$. Observe that, for any $p\in\bH^1_{1+\alpha}=\bH^1_{1+\alpha}(\bR^n)$ and $q\in\bH^1_\alpha=\bH^1_\alpha(\bR^n)$,
\begin{align}
	\nonumber
	\int_{\bR_+}|M_b(\theta)||p(\theta)|\dmu&\leq\left(\int_{\bR_+}(1+\theta)^{-1-\alpha}r(\theta)^{-1}|M_b(\theta)|^2\dmu\right)^{1/2}\left(\int_{\bR_+}(1+\theta)^{1+\alpha}r(\theta)|p(\theta)|^2\dmu\right)^{1/2}\\
	\label{MP_eq_Mb-p}
	&\leq\|M_b\|_{\bH^1_{-\alpha}}\|p\|_{\bH^1_{1+\alpha}}
\end{align}
and
\begin{align}
	\nonumber
	\int_{\bR_+}|M_\sigma(\theta)||q(\theta)|\dmu&\leq\left(\int_{\bR_+}(1+\theta)^{-\alpha}r(\theta)^{-1}|M_\sigma(\theta)|^2\dmu\right)^{1/2}\left(\int_{\bR_+}(1+\theta)^\alpha r(\theta)|q(\theta)|^2\dmu\right)^{1/2}\\
	\label{MP_eq_Msigma-q}
	&\leq\|M_\sigma\|_{\bH^1_{1-\alpha}}\|q\|_{\bH^1_\alpha},
\end{align}
where we used the fact that $r(\theta)^{-1}\leq(1+\theta)r(\theta)$. In particular, the map $p\mapsto\mu[M_b^\top p]$ is a bounded linear operator from $\bH^1_{1+\alpha}$ to $\bH^1_0$, and the map $q\mapsto\mu[M_\sigma^\top q]$ is a bounded linear operator from $\bH^1_\alpha$ to $\bH^1_0$. From the above observations, together with \eqref{Taylor_eq_var-coefficient-state} and \eqref{Taylor_eq_var-coefficient-cost}, we see that the coefficient \eqref{BSEE_eq_adeq1-coefficient} satisfies \cref{BSEE_assum_coefficient} (with the same parameter $\alpha\in[0,1)$ as in \cref{MP_assum_kernel}). Hence, by \cref{BSEE_theo_BSEE}, the first order adjoint equation \eqref{MP_eq_adeq1} has a unique solution $(\hat{p},\hat{q})$ such that
\begin{equation}\label{MP_eq_adeq1-estimate}
	\bE\left[\int^T_0(T-t)^\alpha\|\hat{p}_t\|_{\bH^1_{1+\alpha}}^2\,\diff t+\int^T_0(T-t)^\alpha\|\hat{q}_t\|_{\bH^1_\alpha}^2\,\diff t+\int^T_0(T-t)^\alpha\|\hat{g}_t\|_{\bH^1_0}^2\,\diff t\right]<\infty,
\end{equation}
where $\hat{g}:\Omega\times[0,T]\to\bH^1_0$ is given by \eqref{MP_eq_generator-g}.

%% Remark

\begin{rem}
In the terminology of \cite{Ha24}, $\bH^1_0=\cH_\mu$, $\bH^1_1=\cV_\mu$, $\bH^1_{-1}=\cV^*_\mu$ and $\bH^1_2=\cD(A)$.
\end{rem}

The second order adjoint equation \eqref{MP_eq_adeq1} can be regarded as a BSEE \eqref{BSEE_eq_BSEE} with
\begin{equation*}
	(\Theta,\Sigma,\nu)=(\bR_+^2,\cB(\bR_+^2),r(\theta_1)r(\theta_2)\mu(\diff\theta_1)\mu(\diff\theta_2)),\ \ \varpi(\theta_1,\theta_2)=\theta_1+\theta_2,\ \ E=\bR^{n\times n},
\end{equation*}
and
\begin{equation}\label{BSEE_eq_adeq2-coefficient}
\begin{split}
	&\Phi(\theta_1,\theta_2)=-\hat{h}_{xx},\\
	&\cG(t,P,Q)(\theta_1,\theta_2)=\hat{b}_x(t)^\top\mu[M_b^\top P(\cdot,\theta_2)]+\mu[P(\theta_1,\cdot)M_b]\hat{b}_x(t)+\hat{\sigma}_x(t)^\top\mu^{\otimes2}[M_\sigma^\top PM_\sigma]\hat{\sigma}_x(t)\\
	&\hspace{3cm}+\hat{\sigma}_x(t)^\top\mu[M_\sigma^\top Q(\cdot,\theta_2)]+\mu[Q(\theta_1,\cdot)M_\sigma]\hat{\sigma}_x(t)\\
	&\hspace{3cm}+\langle\mu[M_b^\top\hat{p}_t],\hat{b}_{xx}(t)\rangle+\langle\mu[M_\sigma^\top\hat{q}_t],\hat{\sigma}_{xx}(t)\rangle-\hat{f}_{xx}(t).
\end{split}
\end{equation}
For each $\beta\in\bR$ and a Euclidean space $E$, we denote $\bH^2_\beta(E):=\bH_\beta(\bR_+^2,\cB(\bR_+^2),r(\theta_1)r(\theta_2)\mu(\diff\theta_1)\mu(\diff\theta_2),\theta_1+\theta_2;E)$. We sometimes omit the dependency on $E$ when there is no ambiguity. By \eqref{MP_eq_mu-integrable}, the space $\bH^1_0(E)=L^2(\bR_+,\cB(\bR_+),r(\theta)\mu(\diff\theta);E)$ is continuously embedded into $\bH^2_0(E)=L^2(\bR_+^2,\cB(\bR_+^2),r(\theta_1)r(\theta_2)\mu(\diff\theta_1)\mu(\diff\theta_2);E)$ by the natural injections
\begin{equation*}
	\big(\theta\mapsto p(\theta)\big)\mapsto\big((\theta_1,\theta_2)\mapsto p(\theta_1)\big)\ \ \text{and}\ \ \big(\theta\mapsto p(\theta)\big)\mapsto\big((\theta_1,\theta_2)\mapsto p(\theta_2)\big).
\end{equation*}
Namely, we have the continuous embeddings $E\hookrightarrow \bH^1_0(E)\hookrightarrow\bH^2_0(E)$. By similar calculations as above, for any $P\in\bH^2_{1+\alpha}=\bH^2_{1+\alpha}(\bR^{n\times n})$ and $Q\in\bH^2_\alpha=\bH^2_\alpha(\bR^{n\times n})$, we see that
\begin{align}
	\nonumber
	&\left(\int_{\bR_+}\left(\int_{\bR_+}|M_b(\theta_1)||P(\theta_1,\theta_2)|\,\mu(\diff\theta_1)\right)^2r(\theta_2)\,\mu(\diff\theta_2)\right)^{1/2}\\
	\nonumber
	&\leq\|M_b\|_{\bH^1_{-\alpha}}\left(\int_{\bR_+}\int_{\bR_+}(1+\theta_1)^{1+\alpha}|P(\theta_1,\theta_2)|^2r(\theta_1)r(\theta_2)\dmumu\right)^{1/2}\\
	\label{MP_eq_Mb-P}
	&\leq\|M_b\|_{\bH^1_{-\alpha}}\|P\|_{\bH^2_{1+\alpha}}
\end{align}
and
\begin{align}
	\nonumber
	&\left(\int_{\bR_+}\left(\int_{\bR_+}|M_\sigma(\theta_1)||Q(\theta_1,\theta_2)|\,\mu(\diff\theta_1)\right)^2r(\theta_2)\,\mu(\diff\theta_2)\right)^{1/2}\\
	\nonumber
	&\leq\|M_\sigma\|_{\bH^1_{1-\alpha}}\left(\int_{\bR_+}\int_{\bR_+}(1+\theta_1)^\alpha|Q(\theta_1,\theta_2)|^2r(\theta_1)r(\theta_2)\dmumu\right)^{1/2}\\
	\label{MP_eq_Msigma-Q}
	&\leq\|M_\sigma\|_{\bH^1_{1-\alpha}}\|Q\|_{\bH^2_\alpha}.
\end{align}
Similarly, we have
\begin{equation*}
	\left(\int_{\bR_+}\left(\int_{\bR_+}|P(\theta_1,\theta_2)||M_b(\theta_2)|\,\mu(\diff\theta_2)\right)^2r(\theta_1)\,\mu(\diff\theta_1)\right)^{1/2}\leq\|M_b\|_{\bH^1_{-\alpha}}\|P\|_{\bH^2_{1+\alpha}}
\end{equation*}
and
\begin{equation*}
	\left(\int_{\bR_+}\left(\int_{\bR_+}|Q(\theta_1,\theta_2)||M_\sigma(\theta_2)|\,\mu(\diff\theta_2)\right)^2r(\theta_1)\,\mu(\diff\theta_1)\right)^{1/2}\leq\|M_\sigma\|_{\bH^1_{1-\alpha}}\|Q\|_{\bH^2_\alpha}.
\end{equation*}
Furthermore, noting that $\|\cdot\|_{\bH^2_{2\alpha}}\leq\|\cdot\|_{\bH^2_{1+\alpha}}$ for $\alpha\in[0,1)$,
\begin{align}
	\nonumber
	&\int_{\bR_+^2}|M_\sigma(\theta_1)||P(\theta_1,\theta_2)||M_\sigma(\theta_2)|\dmumu\\
	\nonumber
	&\leq\int_{\bR_+}(1+\theta)^{-\alpha}|M_\sigma(\theta)|^2r(\theta)^{-1}\dmu\,\left(\int_{\bR_+}\int_{\bR_+}(1+\theta_1)^\alpha(1+\theta_2)^\alpha|P(\theta_1,\theta_2)|^2r(\theta_1)r(\theta_2)\dmumu\right)^{1/2}\\
	\nonumber
	&\leq\int_{\bR_+}(1+\theta)^{1-\alpha}|M_\sigma(\theta)|^2r(\theta)\dmu\left(\int_{\bR_+}\int_{\bR_+}(1+\theta_1+\theta_2)^{2\alpha}|P(\theta_1,\theta_2)|^2r(\theta_1)r(\theta_2)\dmumu\right)^{1/2}\\
	\nonumber
	&=\|M_\sigma\|_{\bH^1_{1-\alpha}}^2\|P\|_{\bH^2_{2\alpha}}\\
	\label{MP_eq_Msigma-P}
	&\leq\|M_\sigma\|_{\bH^1_{1-\alpha}}^2\|P\|_{\bH^2_{1+\alpha}}.
\end{align}
In particular, the maps $P\mapsto((\theta_1,\theta_2)\mapsto\mu[M_b^\top P(\cdot,\theta_2)])$, $P\mapsto((\theta_1,\theta_2)\mapsto\mu[P(\theta_1,\cdot)M_b])$ and $P\mapsto\mu[M_\sigma^\top PM_\sigma]$ are bounded linear operators from $\bH^2_{1+\alpha}$ to $\bH^2_0$, and the maps $Q\mapsto((\theta_1,\theta_2)\mapsto\mu[M_\sigma^\top Q(\cdot,\theta_2)])$ and $Q\mapsto((\theta_1,\theta_2)\mapsto\mu[Q(\theta_1,\cdot)M_\sigma])$ are bounded linear operators from $\bH^2_\alpha$ to $\bH^2_0$. From the above observations, together with \eqref{Taylor_eq_var-coefficient-state}, \eqref{Taylor_eq_var-coefficient-cost} and the estimate \eqref{MP_eq_adeq1-estimate} for the solution $(\hat{p},\hat{q})$ of the first order adjoint equation \eqref{MP_eq_adeq1}, we see that the coefficient \eqref{BSEE_eq_adeq2-coefficient} satisfies \cref{BSEE_assum_coefficient} (with the same parameter $\alpha\in[0,1)$ as in \cref{MP_assum_kernel}). Hence, by \cref{BSEE_theo_BSEE}, the second order adjoint equation \eqref{MP_eq_adeq2} has a unique solution $(\hat{P},\hat{Q})$ such that
\begin{equation}\label{MP_eq_adeq2-estimate} 
	\bE\left[\int^T_0(T-t)^\alpha\|\hat{P}_t\|_{\bH^2_{1+\alpha}}^2\,\diff t+\int^T_0(T-t)^\alpha\|\hat{Q}_t\|_{\bH^2_\alpha}^2\,\diff t+\int^T_0(T-t)^\alpha\|\hat{G}_t\|_{\bH^2_0}^2\,\diff t\right]<\infty,
\end{equation}
where $\hat{G}:\Omega\times[0,T]\to\bH^2_0$ is given by \eqref{MP_eq_generator-G}.

Summarizing the above arguments, we obtain the following result.

%% Proposition

\begin{prop}\label{MP_prop_adeq-wellposed}
Let \cref{control_assum_coefficient} and \cref{MP_assum_kernel} hold. Fix $\hat{u}\in\cU$. Then, there exist unique solutions $(\hat{p},\hat{q}):\Omega\times[0,T]\to\bH^1_{1+\alpha}(\bR^n)\times\bH^1_\alpha(\bR^n)$ and $(\hat{P},\hat{Q}):\Omega\times[0,T]\to\bH^2_{1+\alpha}(\bR^{n\times n})\times\bH^2_\alpha(\bR^{n\times n})$ to the first and second order adjoint equations \eqref{MP_eq_adeq1} and \eqref{MP_eq_adeq2}, respectively. Furthermore, the estimates \eqref{MP_eq_adeq1-estimate} and \eqref{MP_eq_adeq2-estimate} hold.
\end{prop}

%% Remark

\begin{rem}\label{MP_rem_adeq2-symmetry}
Observe that $(\check{P},\check{Q}):\Omega\times[0,T]\to\bH^2_{1+\alpha}(\bR^{n\times n})\times\bH^2_\alpha(\bR^{n\times n})$ defined by $\check{
P}_t(\theta_1,\theta_2):=\hat{P}_t(\theta_2,\theta_1)^\top$ and $\check{Q}_t(\theta_1,\theta_2):=\hat{Q}_t(\theta_2,\theta_1)^\top$ satisfies the same BSEE \eqref{MP_eq_adeq2} as $(\hat{P},\hat{Q})$. Thus, by the uniqueness of the solution, we see that $(\check{P},\check{Q})=(\hat{P},\hat{Q})$. In particular, the term $\mu^{\otimes2}[M_\sigma^\top\hat{P}_tM_\sigma]$ is a symmetric matrix.
\end{rem}

%%%%%%
\subsection{The global maximum principle}\label{MP-main}
%%%%%%

Now we return to the analysis on the global maximum principle. Using \cref{MP_prop_adeq-wellposed} on the well-posedness of the first and second order adjoint equations \eqref{MP_eq_adeq1} and \eqref{MP_eq_adeq2}, we rigorously justify the formal calculations in \cref{MP-heuristic}.

%% Proposition

\begin{prop}\label{MP_prop_J12-small}
Suppose that \cref{control_assum_coefficient} and \cref{MP_assum_kernel} hold. Fix $\hat{u}\in\cU$. Then, there exists a Lebesgue-null set $\hat{\cN}\subset[0,T)$ such that, for any $v\in\cU$ and any $\tau\in[0,T)\setminus\hat{\cN}$, it holds that
\begin{equation}\label{MP_eq_J12-small}
\begin{split}
	J^{1,2,v,\tau,\ep}&=-\bE\left[\int^{\tau+\ep}_\tau\left\{\big\langle\mu[M_b^\top\hat{p}_t],\delta b^v(t)\big\rangle+\big\langle\mu[M_\sigma^\top\hat{q}_t],\delta\sigma^v(t)\big\rangle-\delta f^v(t)\vphantom{\frac{1}{2}}\right.\right.\\
	&\hspace{4cm}\left.\left.+\frac{1}{2}\big\langle\mu^{\otimes2}[M_\sigma^\top\hat{P}_tM_\sigma]\delta\sigma^v(t),\delta\sigma^v(t)\big\rangle\right\}\,\diff t\right]+o(\ep)\ \ \text{as $\ep\downarrow0$,}
\end{split}
\end{equation}
where $J^{1,2,v,\tau,\ep}$ is defined by \eqref{Taylor_eq_vareq-cost}, and $(\hat{p},\hat{q}):\Omega\times[0,T]\to\bH^1_{1+\alpha}\times\bH^1_\alpha$ and $(\hat{P},\hat{Q}):\Omega\times[0,T]\to\bH^2_{1+\alpha}\times\bH^2_\alpha$ are the solutions of the first and second order adjoint equations \eqref{MP_eq_adeq1} and \eqref{MP_eq_adeq2}, respectively.
\end{prop}

%% Proof

\begin{proof}
Let $v\in\cU$, $\tau\in[0,T)$ and $\ep\in(0,T-\tau]$ be given. As before, in this proof, we omit the super scripts $v,\tau,\ep$. We denote by $C$ a positive constant which is independent of $\tau$ and $\ep$ (but may depend on $v$). The value of $C$ may vary from line to line. We show \eqref{MP_eq_J12-small} by (A) deriving \eqref{MP_eq_vareq-cost'} rigorously and (B) showing that the last expectation in \eqref{MP_eq_vareq-cost'} is of order $o(\ep)$ as $\ep\downarrow0$ for any $\tau\in[0,T)\setminus\hat{\cN}$, where $\hat{\cN}\subset[0,T)$ is a Lebesgue-null set independent of the choice of $v$. We identify the solution $(\hat{p},\hat{q}):\Omega\times[0,T]\to\bH^1_{1+\alpha}\times\bH^1_\alpha$ of the first order adjoint equation \eqref{MP_eq_adeq1} as well as the process $\hat{g}:\Omega\times[0,T]\to\bH^1_0$ defined by \eqref{MP_eq_generator-g} with a good representative $(\bar{\hat{p}},\bar{\hat{q}},\bar{\hat{g}}):\Omega\times[0,T]\times\bR_+\to\bR^n\times\bR^n\times\bR^n$ (see \cref{BSEE_defi_solution}). Also, we identify the solution $(\hat{P},\hat{Q}):\Omega\times[0,T]\to\bH^2_{1+\alpha}\times\bH^2_\alpha$ of the second order adjoint equation \eqref{MP_eq_adeq2} as well as the process $\hat{G}:\Omega\times[0,T]\to\bH^2_0$ defined by \eqref{MP_eq_generator-G} with a good representative $(\bar{\hat{P}},\bar{\hat{Q}},\bar{\hat{G}}):\Omega\times[0,T]\times\bR_+^2\to\bR^{n\times n}\times\bR^{n\times n}\times\bR^{n\times n}$. These identifications do not matter in this proof since the objective \eqref{MP_eq_vareq-cost'} is independent of the choice of the good representatives. By the same reason, we may assume that $\hat{p}_T(\theta)=-\hat{h}_x^\top$ a.s.\ for any $\theta\in\bR_+$ and $\hat{P}_T(\theta_1,\theta_2)=-\hat{h}_{xx}$ a.s.\ for any $(\theta_1,\theta_2)\in\bR_+^2$.

\underline{(A) Rigorous derivation of \eqref{MP_eq_vareq-cost'}.} By \cref{BSEE_defi_solution}, we see that the generators $\hat{g}(\theta)$ and $\hat{G}(\theta_1,\theta_2)$ defined by \eqref{MP_eq_generator-g} and \eqref{MP_eq_generator-G} (through the solutions $(\hat{p},\hat{q})$ and $(\hat{P},\hat{Q})$) satisfy the integrability condition \eqref{MP_eq_integrability-gG}. Thus, by the arguments in \cref{MP-heuristic}, the equalities \eqref{MP_eq_duality1-theta} and \eqref{MP_eq_duality2-theta} hold for any $\theta\in\bR_+$ and any $(\theta_1,\theta_2)\in\bR_+^2$, respectively. In order to show \eqref{MP_eq_vareq-cost'} rigorously, it remains to justify the applications of Fubini's theorem deriving \eqref{MP_eq_duality1-mu} and \eqref{MP_eq_duality2-mu} from \eqref{MP_eq_duality1-theta} and \eqref{MP_eq_duality2-theta}, respectively.

\underline{(A-1) Justification of Fubini's theorem deriving \eqref{MP_eq_duality1-mu} from \eqref{MP_eq_duality1-theta}.} In order to justify the derivation of \eqref{MP_eq_duality1-mu} from \eqref{MP_eq_duality1-theta}, it suffices to show that
\begin{align}
	\label{MP_eq_Fubini1-1}
	&\bE\left[\int^T_0\int_{\bR_+}|\hat{g}_t(\theta)||Y^{1,2}_t(\theta)|\dmu\,\diff t\right]<\infty,\\
	\label{MP_eq_Fubini1-2}
	&\bE\left[\int^T_0\int_{\bR_+}|M_b(\theta)||\hat{p}_t(\theta)||\xi_t|\dmu\,\diff t\right]<\infty\ \ \text{for any $\xi\in L^{2,\infty}_\bF(0,T;\bR^n)$, and}\\
	\label{MP_eq_Fubini1-3}
	&\bE\left[\int^T_0\int_{\bR_+}|M_\sigma(\theta)||\hat{q}_t(\theta)||\xi_t|\dmu\,\diff t\right]<\infty\ \ \text{for any $\xi\in L^{2,\infty}_\bF(0,T;\bR^n)$.}
\end{align}
First, we show \eqref{MP_eq_Fubini1-1}. Noting $r(\theta)^{-1}\leq(1+\theta)r(\theta)$, $\alpha\in[0,1)$ and the estimates \eqref{MP_eq_adeq1-estimate} for $\hat{g}$ and \eqref{MP_eq_vareq-lift} for $Y^{1,2}=Y^{1,2,v,\tau,\ep}$, by the Cauchy--Schwarz inequality and Tonelli's theorem, we have
\begin{align*}
	&\bE\left[\int^T_0\int_{\bR_+}|\hat{g}_t(\theta)||Y^{1,2}_t(\theta)|\dmu\,\diff t\right]\\
	&\leq\left(\int_{\bR_+}r(\theta)^{-1}\|Y^{1,2}(\theta)\|_{C^2_\bF(0,T)}^2\dmu\right)^{1/2}\left(\int_{\bR_+}r(\theta)\left(\int^T_0\bE\big[|\hat{g}_t(\theta)|^2\big]^{1/2}\,\diff t\right)^2\dmu\right)^{1/2}\\
	&\leq C\ep^{(1-\alpha)/2}\left(\int^T_0(T-t)^{-\alpha}\,\diff t\right)^{1/2}\bE\left[\int^T_0(T-t)^\alpha\|\hat{g}_t\|_{\bH^1_0}^2\,\diff t\right]^{1/2}\\
	&<\infty.
\end{align*}
Hence, \eqref{MP_eq_Fubini1-1} holds. Next, we show \eqref{MP_eq_Fubini1-2} and \eqref{MP_eq_Fubini1-3}. Let $\xi\in L^{2,\infty}_\bF(0,T;\bR^n)$. Noting \eqref{MP_eq_Mb-p} and \eqref{MP_eq_adeq1-estimate}, we have
\begin{align*}
	&\bE\left[\int^T_0\int_{\bR_+}|M_b(\theta)||\hat{p}_t(\theta)||\xi_t|\dmu\,\diff t\right]\\
	&\leq\|\xi\|_{L^{2,\infty}_\bF(0,T)}\int^T_0\bE\left[\left(\int_{\bR_+}|M_b(\theta)||\hat{p}_t(\theta)|\dmu\right)^2\right]^{1/2}\,\diff t\\
	&\leq\|\xi\|_{L^{2,\infty}_\bF(0,T)}\left(\int^T_0(T-t)^{-\alpha}\,\diff t\right)^{1/2}\|M_b\|_{\bH^1_{-\alpha}}\bE\left[\int^T_0(T-t)^\alpha\|\hat{p}_t\|_{\bH^1_{1+\alpha}}^2\,\diff t\right]^{1/2}\\
	&<\infty.
\end{align*}
Hence, \eqref{MP_eq_Fubini1-2} holds. Similarly, noting \eqref{MP_eq_Msigma-q} and \eqref{MP_eq_adeq1-estimate}, we have
\begin{align*}
	&\bE\left[\int^T_0\int_{\bR_+}|M_\sigma(\theta)||\hat{q}_t(\theta)||\xi_t|\dmu\,\diff t\right]\\
	&\leq\|\xi\|_{L^{2,\infty}_\bF(0,T)}\int^T_0\bE\left[\left(\int_{\bR_+}|M_\sigma(\theta)||\hat{q}_t(\theta)|\dmu\right)^2\right]^{1/2}\,\diff t\\
	&\leq\|\xi\|_{L^{2,\infty}_\bF(0,T)}\left(\int^T_0(T-t)^{-\alpha}\,\diff t\right)^{1/2}\|M_\sigma\|_{\bH^1_{1-\alpha}}\bE\left[\int^T_0(T-t)^\alpha\|\hat{q}_t\|_{\bH^1_\alpha}^2\,\diff t\right]^{1/2}\\
	&<\infty.
\end{align*}
Hence, \eqref{MP_eq_Fubini1-3} holds. The above observations justify the derivation of \eqref{MP_eq_duality1-mu} from \eqref{MP_eq_duality1-theta} by means of Fubini's theorem, and the proof of (A-1) is completed.

\underline{(A-2) Justification of Fubini's theorem deriving \eqref{MP_eq_duality2-mu} from \eqref{MP_eq_duality2-theta}.} In order to justify the derivation of \eqref{MP_eq_duality2-mu} from \eqref{MP_eq_duality2-theta}, noting the symmetry of $(\hat{P},\hat{Q})$ (see \cref{MP_rem_adeq2-symmetry}), it suffices to show that
\begin{align}
	\label{MP_eq_Fubini2-1}
	&\bE\left[\int^T_0\int_{\bR_+}\int_{\bR_+}|\hat{G}_t(\theta_1,\theta_2)||Y^1_t(\theta_1)||Y^1_t(\theta_2)|\dmumu\,\diff t\right]<\infty,\\
	\label{MP_eq_Fubini2-2}
	&\bE\left[\int^T_0\int_{\bR_+}\int_{\bR_+}|M_b(\theta_1)||\hat{P}_t(\theta_1,\theta_2)||Y^1_t(\theta_2)||\xi_t|\dmumu\,\diff t\right]<\infty\ \ \text{for any $\xi\in L^{4,\infty}_\bF(0,T;\bR^n)$,}\\
	\label{MP_eq_Fubini2-3}
	&\bE\left[\int^T_0\int_{\bR_+}\int_{\bR_+}|M_\sigma(\theta_1)||\hat{Q}_t(\theta_1,\theta_2)||Y^1_t(\theta_2)||\xi_t|\dmumu\,\diff t\right]<\infty\ \ \text{for any $\xi\in L^{4,\infty}_\bF(0,T;\bR^n)$, and}\\
	\label{MP_eq_Fubini2-4}
	&\bE\left[\int^T_0\int_{\bR_+}\int_{\bR_+}|M_\sigma(\theta_1)||\hat{P}_t(\theta_1,\theta_2)||M_\sigma(\theta_2)||\xi^1_t||\xi^2_t|\dmumu\,\diff t\right]<\infty\ \ \text{for any $\xi^1,\xi^2\in L^{4,\infty}_\bF(0,T;\bR^n)$.}
\end{align}
First, we show \eqref{MP_eq_Fubini2-1}. Noting the estimates \eqref{MP_eq_adeq2-estimate} for $\hat{G}$ and \eqref{MP_eq_vareq-lift} for $Y^1=Y^{1,v,\tau,\ep}$, by H\"{o}lder's inequality and Tonelli's theorem, we have
\begin{align*}
	&\bE\left[\int^T_0\int_{\bR_+}\int_{\bR_+}|\hat{G}_t(\theta_1,\theta_2)||Y^1_t(\theta_1)||Y^1_t(\theta_2)|\dmumu\,\diff t\right]\\
	&\leq\int_{\bR_+}r(\theta)^{-1}\|Y^1(\theta)\|_{C^4_\bF(0,T)}^2\dmu\,\left(\int_{\bR_+}\int_{\bR_+}r(\theta_1)r(\theta_2)\left(\int^T_0\bE\big[|\hat{G}_t(\theta_1,\theta_2)|^2\big]^{1/2}\,\diff t\right)^2\dmumu\right)^{1/2}\\
	&\leq C\ep^{1-\alpha}\left(\int^T_0(T-t)^{-\alpha}\,\diff t\right)^{1/2}\bE\left[\int^T_0(T-t)^\alpha\|\hat{G}_t\|_{\bH^2_0}^2\,\diff t\right]^{1/2}\\
	&<\infty.
\end{align*}
Hence, \eqref{MP_eq_Fubini2-1} holds. Next, we show \eqref{MP_eq_Fubini2-2} and \eqref{MP_eq_Fubini2-3}. Let $\xi\in L^{4,\infty}_\bF(0,T;\bR^n)$. Noting \eqref{MP_eq_vareq-lift}, \eqref{MP_eq_Mb-P} and \eqref{MP_eq_adeq2-estimate}, we have
\begin{align*}
	&\bE\left[\int^T_0\int_{\bR_+}\int_{\bR_+}|M_b(\theta_1)||\hat{P}_t(\theta_1,\theta_2)||Y^1_t(\theta_2)||\xi_t|\dmumu\,\diff t\right]\\
	&\leq\|\xi\|_{L^{4,\infty}_\bF(0,T)}\left(\int_{\bR_+}r(\theta)^{-1}\|Y^1(\theta)\|_{C^4_\bF(0,T)}^2\dmu\right)^{1/2}\\
	&\hspace{1cm}\times\left(\int_{\bR_+}r(\theta_2)\left(\int^T_0\bE\left[\left(\int_{\bR_+}|M_b(\theta_1)||\hat{P}_t(\theta_1,\theta_2)|\,\mu(\diff\theta_1)\right)^2\right]^{1/2}\,\diff t\right)^2\,\mu(\diff\theta_2)\right)^{1/2}\\
	&\leq C\ep^{(1-\alpha)/2}\|\xi\|_{L^{4,\infty}_\bF(0,T)}\left(\int^T_0(T-t)^{-\alpha}\,\diff t\right)^{1/2}\|M_b\|_{\bH^1_{-\alpha}}\bE\left[\int^T_0(T-t)^\alpha\|\hat{P}_t\|_{\bH^2_{1+\alpha}}^2\,\diff t\right]^{1/2}\\
	&<\infty.
\end{align*}
Hence, \eqref{MP_eq_Fubini2-2} holds. Also, noting \eqref{MP_eq_vareq-lift}, \eqref{MP_eq_Msigma-Q} and \eqref{MP_eq_adeq2-estimate}, we have
\begin{align*}
	&\bE\left[\int^T_0\int_{\bR_+}\int_{\bR_+}|M_\sigma(\theta_1)||\hat{Q}_t(\theta_1,\theta_2)||Y^1_t(\theta_2)||\xi_t|\dmumu\,\diff t\right]\\
	&\leq\|\xi\|_{L^{4,\infty}_\bF(0,T)}\left(\int_{\bR_+}r(\theta)^{-1}\|Y^1(\theta)\|_{C^4_\bF(0,T)}^2\dmu\right)^{1/2}\\
	&\hspace{1cm}\times\left(\int_{\bR_+}r(\theta_2)\left(\int^T_0\bE\left[\left(\int_{\bR_+}|M_\sigma(\theta_1)||\hat{Q}_t(\theta_1,\theta_2)|\,\mu(\diff\theta_1)\right)^2\right]^{1/2}\,\diff t\right)^2\,\mu(\diff\theta_2)\right)^{1/2}\\
	&\leq C\ep^{(1-\alpha)/2}\|\xi\|_{L^{4,\infty}_\bF(0,T)}\left(\int^T_0(T-t)^{-\alpha}\,\diff t\right)^{1/2}\|M_\sigma\|_{\bH^1_{1-\alpha}}\bE\left[\int^T_0(T-t)^\alpha\|\hat{Q}_t\|_{\bH^2_\alpha}^2\,\diff t\right]^{1/2}\\
	&<\infty.
\end{align*}
Hence, \eqref{MP_eq_Fubini2-3} holds. Lastly, we show \eqref{MP_eq_Fubini2-4}. Let $\xi^1,\xi^2\in L^{4,\infty}_\bF(0,T;\bR^n)$. Noting \eqref{MP_eq_vareq-lift}, \eqref{MP_eq_Msigma-P} and \eqref{MP_eq_adeq2-estimate}, we have
\begin{align*}
	&\bE\left[\int^T_0\int_{\bR_+}\int_{\bR_+}|M_\sigma(\theta_1)||\hat{P}_t(\theta_1,\theta_2)||M_\sigma(\theta_2)||\xi^1_t||\xi^2_t|\dmumu\,\diff t\right]\\
	&\leq\|\xi^1\|_{L^{4,\infty}_\bF(0,T)}\|\xi^2\|_{L^{4,\infty}_\bF(0,T)}\int^T_0\bE\left[\left(\int_{\bR_+}\int_{\bR_+}|M_\sigma(\theta_1)||\hat{P}_t(\theta_1,\theta_2)||M_\sigma(\theta_2)|\dmumu\right)^2\right]^{1/2}\,\diff t\\
	&\leq\|\xi^1\|_{L^{4,\infty}_\bF(0,T)}\|\xi^2\|_{L^{4,\infty}_\bF(0,T)}\left(\int^T_0(T-t)^{-\alpha}\,\diff t\right)^{1/2}\|M_\sigma\|_{\bH^1_{1-\alpha}}^2\bE\left[\int^T_0(T-t)^\alpha\|\hat{P}_t\|_{\bH^2_{1+\alpha}}^2\,\diff t\right]^{1/2}\\
	&<\infty.
\end{align*}
Hence, \eqref{MP_eq_Fubini2-4} holds. The above observations justify the derivation of \eqref{MP_eq_duality2-mu} from \eqref{MP_eq_duality2-theta} by means of Fubini's theorem, and the proof of (A-2) is completed. 

The observations in (A) above rigorously prove the equality \eqref{MP_eq_vareq-cost'}. By the relations \eqref{MP_eq_generator-g} and \eqref{MP_eq_generator-G} for $\hat{g}(\theta)$ and $\hat{G}(\theta_1,\theta_2)$, we see that the first and second expectations in the right-hand side of \eqref{MP_eq_vareq-cost'} are zeros. Consequently, noting the symmetry of $(\hat{P},\hat{Q})$ (see \cref{MP_rem_adeq2-symmetry}), we obtain the expression \eqref{MP_eq_vareq-cost''}, where $\cE=\cE^{v,\tau,\ep}\in\bR$ is given by
\begin{align*}
	\cE&:=\bE\left[\int^{\tau+\ep}_\tau\Big\{\big\langle\mu[M_b^\top\hat{p}_t],\delta b_x(t)X^1_t\rangle+\langle\mu[M_\sigma^\top\hat{q}_t],\delta\sigma_x(t)X^1_t\big\rangle\vphantom{\int_{\bR_+}}\right.\\
	&\left.\hspace{2cm}+\int_{\bR_+}\big\langle\mu[M_b^\top\hat{P}_t(\cdot,\theta_2)]Y^1_t(\theta_2),\delta b(t)\big\rangle\,\mu(\diff\theta_2)+\int_{\bR_+}\big\langle\mu[M_\sigma^\top\hat{Q}_t(\cdot,\theta_2)]Y^1_t(\theta_2),\delta\sigma(t)\big\rangle\,\mu(\diff\theta_2)\vphantom{\int_{\bR_+}}\right.\\
	&\left.\hspace{2cm}+\big\langle\mu^{\otimes 2}[M_\sigma^\top\hat{P}_tM_\sigma]\hat{\sigma}_x(t)X^1_t,\delta\sigma(t)\big\rangle\vphantom{\int_{\bR_+}}\Big\}\,\diff t\right],
\end{align*}
which is nothing but the last expectation in the right-hand side of \eqref{MP_eq_vareq-cost'}.

\underline{(B) Convergence with order $o(\ep)$ of the last expectation in the right-hand side of \eqref{MP_eq_vareq-cost'}.} In order to show \eqref{MP_eq_J12-small}, it remains to prove that there exists a Lebesgue-null set $\hat{\cN}\subset[0,T)$, which is independent of the choice of $v\in\cU$, such that $\cE^{v,\tau,\ep}=o(\ep)$ as $\ep\downarrow0$ for any $\tau\in[0,T)\setminus\hat{\cN}$. For this purpose, it suffices to show that the following hold:
\begin{align}
	\label{MP_eq_remainder1}&\lim_{\ep\downarrow0}\frac{1}{\ep}\bE\left[\int^{\tau+\ep}_\tau\int_{\bR_+}|M_b(\theta)||\hat{p}_t(\theta)||\delta b_x(t)X^1_t|\dmu\,\diff t\right]=0,\\
	\label{MP_eq_remainder2}&\lim_{\ep\downarrow0}\frac{1}{\ep}\bE\left[\int^{\tau+\ep}_\tau\int_{\bR_+}|M_\sigma(\theta)||\hat{q}_t(\theta)||\delta\sigma_x(t)X^1_t|\dmu\,\diff t\right]=0,\\
	\label{MP_eq_remainder3}&\lim_{\ep\downarrow0}\frac{1}{\ep}\bE\left[\int^{\tau+\ep}_\tau\int_{\bR_+}\int_{\bR_+}|M_b(\theta_1)||\hat{P}_t(\theta_1,\theta_2)||Y^1_t(\theta_2)||\delta b(t)|\dmumu\,\diff t\right]=0,\\
	\label{MP_eq_remainder4}&\lim_{\ep\downarrow0}\frac{1}{\ep}\bE\left[\int^{\tau+\ep}_\tau\int_{\bR_+}\int_{\bR_+}|M_\sigma(\theta_1)||\hat{Q}_t(\theta_1,\theta_2)||Y^1_t(\theta_2)||\delta\sigma(t)|\dmumu\,\diff t\right]=0,\ \ \text{and}\\
	\label{MP_eq_remainder5}&\lim_{\ep\downarrow0}\frac{1}{\ep}\bE\left[\int^{\tau+\ep}_\tau\int_{\bR_+}\int_{\bR_+}|M_\sigma(\theta_1)||\hat{P}_t(\theta_1,\theta_2)||M_\sigma(\theta_2)||\hat{\sigma}_x(t)X^1_t||\delta\sigma(t)|\dmumu\,\diff t\right]=0,
\end{align}
for any $\tau\in[0,T)\setminus\hat{\cN}$ for some Lebesgue-null set $\hat{\cN}\subset[0,T)$ which is independent of the choice of $v$. The above five claims are shown by applying the arguments we have proved so far. Let us prove \eqref{MP_eq_remainder1}. Applying the argument showing \eqref{MP_eq_Fubini1-2} above to $\xi_t=\delta b_x(t)X^1_t\1_{[\tau,\tau+\ep]}(t)$ and noting $|\delta b_x(t)|\leq2\kappa$ by \cref{control_assum_coefficient}, we see that, for any $v\in\cU$, $\tau\in[0,T)$ and $\ep\in(0,T-\tau]$,
\begin{align*}
	&\bE\left[\int^{\tau+\ep}_\tau\int_{\bR_+}|M_b(\theta)||\hat{p}_t(\theta)||\delta b_x(t)X^1_t|\dmu\,\diff t\right]\\
	&\leq2\kappa\|X^1\|_{C^2_\bF(0,T)}\left(\int^{\tau+\ep}_\tau(T-t)^{-\alpha}\,\diff t\right)^{1/2}\|M_b\|_{\bH^1_{-\alpha}}\bE\left[\int^{\tau+\ep}_\tau(T-t)^\alpha\|\hat{p}_t\|_{\bH^1_{1+\alpha}}^2\,\diff t\right]^{1/2}.
\end{align*}
By \eqref{Taylor_eq_var-1} in \cref{Taylor_prop_SVE-expansion}, we have $\lim_{\ep\downarrow0}\|X^{1,v,\tau,\ep}\|_{C^2_\bF(0,T)}=0$ for any $v\in\cU$ and any $\tau\in[0,T)$. Also, we have $\lim_{\ep\downarrow0}\frac{1}{\ep}\int^{\tau+\ep}_\tau(T-t)^{-\alpha}\,\diff t=(T-\tau)^{-\alpha}<\infty$ for any $\tau\in[0,T)$. Lastly, noting the estimate \eqref{MP_eq_adeq1-estimate} for $\hat{p}$, by Lebesgue's differentiation theorem, there exists a Lebesgue-null set $\hat{\cN}\subset[0,T)$, which does not depend on $v$, such that
\begin{equation*}
	\lim_{\ep\downarrow0}\frac{1}{\ep}\bE\left[\int^{\tau+\ep}_\tau(T-t)^\alpha\|\hat{p}_t\|_{\bH^1_{1+\alpha}}^2\,\diff t\right]=(T-\tau)^\alpha\bE\Big[\|\hat{p}_\tau\|_{\bH^1_{1+\alpha}}^2\Big]<\infty
\end{equation*}
for any $\tau\in[0,T)\setminus\hat{\cN}$. Hence, we see that \eqref{MP_eq_remainder1} holds for any $\tau\in[0,T)\setminus\hat{\cN}$. Similarly, we can show \eqref{MP_eq_remainder2}, \eqref{MP_eq_remainder3}, \eqref{MP_eq_remainder4} and \eqref{MP_eq_remainder5} by using the arguments showing \eqref{MP_eq_Fubini1-3}, \eqref{MP_eq_Fubini2-2}, \eqref{MP_eq_Fubini2-3} and \eqref{MP_eq_Fubini2-4} above, respectively.

Consequently, we have $\cE^{v,\tau,\ep}=o(\ep)$ as $\ep\downarrow0$ for any $\tau\in[0,T)\setminus\hat{\cN}$ for some Lebesgue-null set $\hat{\cN}\subset[0,T)$, which is independent of the choice of $v$. This completes the proof of \cref{MP_prop_J12-small}.
\end{proof}

%% Remark

\begin{rem}
Unlike the last assertion of \cref{Taylor_prop_SVE-expansion}, \cref{MP_prop_J12-small} holds true without assuming that $K_b\in L^{3/2}(0,T;\bR^{n\times n})$ or $K_\sigma\in L^6(0,T;\bR^{n\times n})$.
\end{rem}

Recall the definition \eqref{intro_eq_Hamiltonian} of the \emph{Hamiltonian} $H:\Omega\times[0,T]\times U\times\bR^n\times\bR^n\times\bR^n\to\bR$:
\begin{equation*}
	H(t,u,x,p,q):=\langle p,b(t,u,x)\rangle+\langle q,\sigma(t,u,x)\rangle-f(t,u,x),
\end{equation*}
for $(\omega,t,u,x,p,q)\in\Omega\times[0,T]\times U\times\bR^n\times\bR^n\times\bR^n$, where we suppressed the dependency on $\omega$. This is the ``standard'' Hamiltonian appearing in control problems of SDEs; see \cite{Pe90} and \cite[Chapter 3]{YoZh99}. Using the Hamiltonian, the first and second order adjoint equations \eqref{MP_eq_adeq1} and \eqref{MP_eq_adeq2} can be written as \eqref{intro_eq_adeq1} and \eqref{intro_eq_adeq2}, respectively.

Combining \cref{Taylor_prop_SVE-expansion} and \cref{MP_prop_J12-small}, we obtain the following \emph{global maximum principle}, which is the main result of this paper.

%% Theorem

\begin{theo}\label{MP_theo_MP}
Suppose that \cref{control_assum_coefficient} and \cref{MP_assum_kernel} with $\alpha=\frac{1}{3}$ hold. Let $\hat{u}\in\cU$ be given, and denote $\hat{X}:=X^{\hat{u}}$. Also, let $(\hat{p},\hat{q})$ and $(\hat{P},\hat{Q})$ be the solutions of the associated first and second order adjoint equations \eqref{MP_eq_adeq1} and \eqref{MP_eq_adeq2}, respectively. If $\hat{u}$ is an optimal control of the stochastic Volterra control problem \eqref{intro_eq_cSVE}--\eqref{intro_eq_cost}, then the following variational inequality holds:
\begin{equation}\label{MP_eq_variational-inequality}
\begin{split}
	&H\big(t,\hat{u}_t,\hat{X}_t,\mu[M_b^\top\hat{p}_t],\mu[M_\sigma^\top\hat{q}_t]\big)-H\big(t,v,\hat{X}_t,\mu[M_b^\top\hat{p}_t],\mu[M_\sigma^\top\hat{q}_t]\big)\\
	&-\frac{1}{2}\Big\langle\mu^{\otimes2}[M_\sigma^\top\hat{P}_tM_\sigma]\big\{\sigma(t,\hat{u}_t,\hat{X}_t)-\sigma(t,v,\hat{X}_t)\big\},\sigma(t,\hat{u}_t,\hat{X}_t)-\sigma(t,v,\hat{X}_t)\Big\rangle\geq0\\
	&\hspace{4cm}\text{for all $v\in U$ a.s.\ for a.e.\ $t\in[0,T]$}.
\end{split}
\end{equation}
\end{theo}

%% Proof

\begin{proof}
By \cref{MP_lemm_kernel}, \cref{MP_assum_kernel} with $\alpha=\frac{1}{3}$ implies that $K_b\in L^{3/2}(0,T;\bR^{n\times n})$ and $K_\sigma\in L^6(0,T;\bR^{n\times n})$. Hence, the last assertion in \cref{Taylor_prop_SVE-expansion} holds.

Assume that $\hat{u}\in\cU$ is an optimal control. Define the so-called $\cH$-function $\hat{\cH}:\Omega\times[0,T]\times U\to\bR$ by
\begin{align*}
	\hat{\cH}_t(v)&:=H\big(t,v,\hat{X}_t,\mu[M_b^\top\hat{p}_t],\mu[M_\sigma^\top\hat{q}_t]\big)\\
	&\hspace{1cm}+\frac{1}{2}\Big\langle\mu^{\otimes2}[M_\sigma^\top \hat{P}_tM_\sigma]\big\{\sigma(t,\hat{u}_t,\hat{X}_t)-\sigma(t,v,\hat{X}_t)\big\},\sigma(t,\hat{u}_t,\hat{X}_t)-\sigma(t,v,\hat{X}_t)\Big\rangle,
\end{align*}
for $(\omega,t,v)\in\Omega\times[0,T]\times U$, where we suppressed the dependency on $\omega$. It can be easily shown that the $\cH$-function is progressively measurable and satisfies $\bE[\int^T_0|\hat{\cH}_t(v_t)|\,\diff t]<\infty$ for any $v\in\cU$. Furthermore, by the continuity conditions (with respect to $u\in U$) in \cref{control_assum_coefficient}, we see that $v\mapsto\hat{\cH}_t(v)$ is continuous for any $(\omega,t)\in\Omega\times[0,T]$. Notice that the variational inequality \eqref{MP_eq_variational-inequality} is equivalent to the following maximum condition:
\begin{equation}\label{MP_eq_maximum-condition}
	\hat{\cH}_t(\hat{u}_t)=\max_{v\in U}\hat{\cH}_t(v)\ \ \text{a.s.\ for a.e.\ $t\in[0,T]$}.
\end{equation}
In the following, we will show \eqref{MP_eq_maximum-condition}.

Noting that the control domain $U$ is separable (by \cref{control_assum_coefficient}), we can take a countable dense subset $\{v^\ell\}_{\ell\in\bN}$ of $U$. Each $v^\ell\in U$ can be seen as a constant control process $v^\ell_t:=v^\ell$, $t\in[0,T]$, which belongs to $\cU$ thanks to \cref{control_assum_coefficient}. Notice that (the $\bP$-equivalence classes of) the functions $t\mapsto\hat{\cH}_t(\hat{u}_t)$ and $t\mapsto\hat{\cH}_t(v^\ell)$ are Bochner-integrable as maps from $[0,T]$ to the separable Banach space $L^1(\Omega,\cF_T,\bP)$. Hence, by Lebesgue's differentiation theorem for Bochner-integrable functions (cf.\ \cite[Theorem 3.8.5]{HiPh57}), there exists a Lebesgue-null set $\check{\cN}\subset[0,T)$ such that
\begin{equation}\label{MP_eq_Lebesgue}
	\lim_{\ep\downarrow0}\frac{1}{\ep}\int^{\tau+\ep}_\tau\bE\big[|\hat{\cH}_t(\hat{u}_t)-\hat{\cH}_\tau(\hat{u}_\tau)|\big]\,\diff t=0\ \ \text{and}\ \ \lim_{\ep\downarrow0}\frac{1}{\ep}\int^{\tau+\ep}_\tau\bE\big[|\hat{\cH}_t(v^\ell)-\hat{\cH}_\tau(v^\ell)|\big]\,\diff t=0
\end{equation}
for any $\ell\in\bN$ and any $\tau\in[0,T)\setminus\check{\cN}$. Let $\cN:=\hat{\cN}\cup\check{\cN}$, where $\hat{\cN}\subset[0,T)$ is the Lebesgue-null set specified in \cref{MP_prop_J12-small}. Here, we emphasize that the null set $\cN$ depends only on the fixed (optimal) control $\hat{u}$ and the sequence $\{v^\ell\}_{\ell\in\bN}$.

Let $\ell\in\bN$, $\tau\in[0,T)\setminus\cN$ and $A_\tau\in\cF_\tau$ be given. Define
\begin{equation*}
	v^{\ell,\tau}_t(\omega):=v^\ell\1_{A_\tau\times[\tau,T]}(\omega,t)+\hat{u}_t(\omega)\1_{(\Omega\times[0,T])\setminus(A_\tau\times[\tau,T])}(\omega,t).
\end{equation*}
Notice that $v^{\ell,\tau}\in\cU$. Since $K_b\in L^{3/2}(0,T;\bR^{n\times n})$ and $K_\sigma\in L^6(0,T;\bR^{n\times n})$, by the last assertion in \cref{Taylor_prop_SVE-expansion}, we have
\begin{equation*}
	J(u^{v^{\ell,\tau},\tau,\ep})-J(\hat{u})=J^{1,2,v^{\ell,\tau},\tau,\ep}+o(\ep)\ \ \text{as $\ep\downarrow0$},
\end{equation*}
where $u^{v^{\ell,\tau},\tau,\ep}\in\cU$ and $J^{1,2,v^{\ell,\tau},\tau,\ep}$ are defined by \eqref{Taylor_eq_spike-variation} and \eqref{Taylor_eq_vareq-cost}, respectively, with the choice $v=v^{\ell,\tau}$. Furthermore, noting that $\tau\in[0,T)\setminus\cN$, by \cref{MP_prop_J12-small} and the definition of the $\cH$-function, we see that 
\begin{equation*}
	J^{1,2,v^{\ell,\tau},\tau,\ep}=\bE\left[\int^{\tau+\ep}_\tau\Big\{\hat{\cH}_t(\hat{u}_t)-\hat{\cH}_t(v^{\ell,\tau}_t)\Big\}\,\diff t\right]+o(\ep)\ \ \text{as $\ep\downarrow0$}.
\end{equation*}
However, by the definition of $v^{\ell,\tau}$ and \eqref{MP_eq_Lebesgue}, we have
\begin{align*}
	\bE\left[\int^{\tau+\ep}_\tau\Big\{\hat{\cH}_t(\hat{u}_t)-\hat{\cH}_t(v^{\ell,\tau}_t)\Big\}\,\diff t\right]&=\bE\left[\int^{\tau+\ep}_\tau\Big\{\hat{\cH}_t(\hat{u}_t)-\hat{\cH}_t(v^\ell)\Big\}\,\diff t\,\1_{A_\tau}\right]\\
	&=\bE\Big[\big\{\hat{\cH}_\tau(\hat{u}_\tau)-\hat{\cH}_\tau(v^\ell)\big\}\1_{A_\tau}\Big]\,\ep+o(\ep)\ \ \text{as $\ep\downarrow0$}.
\end{align*}
Consequently, we have
\begin{equation*}
	J(u^{v^{\ell,\tau},\tau,\ep})-J(\hat{u})=\bE\Big[\big\{\hat{\cH}_\tau(\hat{u}_\tau)-\hat{\cH}_\tau(v^\ell)\big\}\1_{A_\tau}\Big]\,\ep+o(\ep)\ \ \text{as $\ep\downarrow0$}.
\end{equation*}
By the optimality of $\hat{u}$, we have $J(u^{v^{\ell,\tau},\tau,\ep})-J(\hat{u})\geq0$ for any $\ep\in(0,T-\tau]$, and hence it holds that
\begin{equation*}
	\bE\Big[\big\{\hat{\cH}_\tau(\hat{u}_\tau)-\hat{\cH}_\tau(v^\ell)\big\}\1_{A_\tau}\Big]\geq0.
\end{equation*}
Since $A_\tau\in\cF_\tau$ is arbitrary, noting that $\hat{\cH}_\tau(\hat{u}_\tau)$ and $\hat{\cH}_\tau(v^\ell)$ are $\cF_\tau$-measurable, we have $\hat{\cH}_\tau(\hat{u}_\tau)\geq\hat{\cH}_\tau(v^\ell)$ a.s. Furthermore, since $\ell\in\bN$ is arbitrary, the continuity of $v\mapsto\hat{\cH}_\tau(v)$ and the density of $\{v^\ell\}_{\ell\in\bN}$ in $U$ yield that
\begin{equation*}
	\hat{\cH}_\tau(\hat{u}_\tau)=\max_{v\in U}\hat{\cH}_\tau(v)\ \ \text{a.s.}
\end{equation*}
The above holds for any $\tau\in[0,T)\setminus\cN$ with $\cN\subset[0,T)$ being null with respect to the Lebesgue measure. Thus, the maximum condition \eqref{MP_eq_maximum-condition} holds. This completes the proof.
\end{proof}

%% Remark

\begin{rem}
\begin{itemize}
\item
The above result is a natural generalization of Peng's global maximum principle for SDEs \cite{Pe90} (see also \cite[Chapter 3]{YoZh99}). Indeed, if $K_b(t)=K_\sigma(t)=I_{n\times n}$, then we can take $\mu(\diff\theta)=\delta_0(\diff\theta)$ and $M_b(\theta)=M_\sigma(\theta)=I_{n\times n}$. In this case, the first and second order adjoint equations \eqref{MP_eq_adeq1} and \eqref{MP_eq_adeq2} (with $\theta=0$ and $(\theta_1,\theta_2)=(0,0)$, respectively) as well as the variational inequality \eqref{MP_eq_variational-inequality} reduce to the counterparts in the SDEs case.
\item
Wang and Yong \cite{WaYo23} obtained a different form of a global maximum principle for non-convolution type SVEs with regular kernels. Their approach relies on the analysis of backward stochastic Volterra integral equations (BSVIEs); the first order adjoint equation is written as a Type-II BSVIE (see \eqref{app_eq_BSVIE1} in \cref{app-BSVIE}), and the second order adjoint equation is written as a coupled system of non-standard BSVIEs (see \eqref{app_eq_BSVIE2} in \cref{app-BSVIE}). Compared with \cite{WaYo23}, our approach does not need the analysis of BSVIEs. The adjoint equations we obtained in the present paper is much simpler than that of \cite{WaYo23}, and is more tractable (especially in the singular kernel case) thanks to the semimartingale structure. As a price, unlike \cite{WaYo23}, we require the structure condition (\cref{MP_assum_kernel}) on the kernels. In \cref{app-BSVIE}, we show some relationships between BSEEs and BSVIEs.
\end{itemize}
\end{rem}

%%%%%%%%%%%%%%%%%%%%%%%%%%%%%%%%%%
%%%%%%%%%%%%%%%%%%%%%%%%%%%%%%%%%%
\section{Well-posedness of general BSEEs on weighted $L^2$ spaces}\label{BSEE}
%%%%%%%%%%%%%%%%%%%%%%%%%%%%%%%%%%
%%%%%%%%%%%%%%%%%%%%%%%%%%%%%%%%%%

In this section, we prove \cref{BSEE_theo_BSEE}. To do so, the following novel estimate for finite dimensional BSDEs, which is interesting by its own right, plays a key role.

%% Lemma

\begin{lemm}\label{BSEE_lemm_BSDE-apriori}
There exists a universal constant $C_0>0$ such that the following holds. Let $E$ be a Euclidean space. Let $\fh:\Omega\to E$ be an $\cF_T$-measurable random variable and $\fg:\Omega\times[0,T]\to E$ be a progressively measurable process such that
\begin{equation}\label{BSEE_eq_BSDE-integrability-generator}
	\bE\left[|\fh|^2+\int^T_0(T-t)^\alpha|\fg_t|^2\,\diff t\right]<\infty,
\end{equation}
where $\alpha\in[0,1)$ is a given constant. Fix a constant $\kappa>0$, and let $(\fp,\fq)$ be the adapted solution of the following BSDE on $E$:
\begin{equation}\label{BSEE_eq_BSDE}
	\begin{dcases}
	\diff\fp_t=\kappa\fp_t\diff t-\fg_t\,\diff t+\fq_t\,\diff W_t,\ \ t\in[0,T],\\
	\fp_T=\fh.
	\end{dcases}
\end{equation}
Then,
\begin{equation}\label{BSEE_eq_BSDE-apriori}
\begin{split}
	&\bE\left[\sup_{t\in[0,T]}|\fp_t|^2+\kappa\int^T_0|\fp_t|^2\,\diff t+\int^T_0|\fq_t|^2\,\diff t\right.\\
	&\left.\hspace{1cm}+\kappa^\alpha\sup_{t\in[0,T]}(T-t)^\alpha|\fp_t|^2+\kappa^{1+\alpha}\int^T_0(T-t)^\alpha|\fp_t|^2\,\diff t+\kappa^\alpha\int^T_0(T-t)^\alpha|\fq_t|^2\,\diff t\right]\\
	&\leq C_0\bE\left[|\fh|^2+\frac{\Gamma(1-\alpha)}{\kappa^{1-\alpha}}\int^T_0(T-t)^\alpha|\fg_t|^2\,\diff t\right],
\end{split}
\end{equation}
where $\Gamma(\beta):=\int^\infty_0t^{\beta-1}e^{-t}\,\diff t$, $\beta>0$, denotes the Gamma function.
\end{lemm}

%% Remark

\begin{rem}
Under the condition \eqref{BSEE_eq_BSDE-integrability-generator} with $\alpha\in[0,1)$, by the Cauchy--Schwarz inequality,
\begin{equation*}
	\bE\left[\left(\int^T_0|\fg_t|\,\diff t\right)^2\right]\leq\frac{T^{1-\alpha}}{1-\alpha}\bE\left[\int^T_0(T-t)^\alpha|\fg_t|^2\,\diff t\right]<\infty.
\end{equation*}
Hence, by \cite[Proposition 4.3.1]{Zh17}, there exists a unique adapted solution $(\fp,\fq)$ to the BSDE \eqref{BSEE_eq_BSDE} such that
\begin{equation*}
	\bE\left[\sup_{t\in[0,T]}|\fp_t|^2+\int^T_0|\fq_t|^2\,\diff t\right]<\infty.
\end{equation*}
The important point in the statement of \cref{BSEE_lemm_BSDE-apriori} is that the constant $C_0$ in \eqref{BSEE_eq_BSDE-apriori} is independent of the parameter $\kappa$.
\end{rem}

%% Proof

\begin{proof}[Proof of \cref{BSEE_lemm_BSDE-apriori}]
In this proof, we denote by $C$ a universal positive constant which varies from line to line.

First, we show that
\begin{equation}\label{BSEE_eq_BSDE-estimate1}
	\bE\left[\sup_{t\in[0,T]}|\fp_t|^2+\kappa\int^T_0|\fp_t|^2\,\diff t+\int^T_0|\fq_t|^2\,\diff t\right]\leq C\bE\left[|\fh|^2+\int^T_0|\fp_t||\fg_t|\,\diff t\right].
\end{equation}
By It\^{o}'s formula, we have
\begin{equation}\label{BSEE_eq_BSDE-Ito}
	|\fp_t|^2+2\kappa\int^T_t|\fp_s|^2\,\diff s+\int^T_t|\fq_s|^2\,\diff s=|\fh|^2+2\int^T_t\langle\fp_s,\fg_s\rangle\,\diff s-2\int^T_t\langle\fp_s,\fq_s\rangle\,\diff W_s,\ \ t\in[0,T].
\end{equation}
On the one hand, taking the expectations in both sides of \eqref{BSEE_eq_BSDE-Ito} with $t=0$, we have
\begin{equation}\label{BSEE_eq_BSDE-estimate1-1}
	\bE\left[2\kappa\int^T_0|\fp_s|^2\,\diff s+\int^T_0|\fq_s|^2\,\diff s\right]\leq\bE\left[|\fh|^2+2\int^T_0|\fp_s||\fg_s|\,\diff s\right].
\end{equation}
On the other hand, taking the supremum in $t\in[0,T]$ and then taking the expectations in both sides of \eqref{BSEE_eq_BSDE-Ito}, the Burkholder--Davis--Gundy inequality and Young's inequality yield that
\begin{align*}
	\bE\left[\sup_{t\in[0,T]}|\fp_t|^2\right]&\leq\bE\left[|\fh|^2+2\int^T_0|\fp_s||\fg_s|\,\diff s\right]+C\bE\left[\left(\int^T_0|\fp_s|^2|\fq_s|^2\,\diff s\right)^{1/2}\right]\\
	&\leq\bE\left[|\fh|^2+2\int^T_0|\fp_s||\fg_s|\,\diff s\right]+C\bE\left[\int^T_0|\fq_s|^2\,\diff s\right]+\frac{1}{2}\bE\left[\sup_{s\in[0,T]}|\fp_s|^2\right].
\end{align*}
Hence, we have
\begin{equation}\label{BSEE_eq_BSDE-estimate1-2}
	\bE\left[\sup_{t\in[0,T]}|\fp_t|^2\right]\leq C\bE\left[|\fh|^2+\int^T_0|\fp_s||\fg_s|\,\diff s\right]+C\bE\left[\int^T_0|\fq_s|^2\,\diff s\right].
\end{equation}
By \eqref{BSEE_eq_BSDE-estimate1-1} and \eqref{BSEE_eq_BSDE-estimate1-2}, we obtain \eqref{BSEE_eq_BSDE-estimate1}. In particular, in the case of $\alpha=0$, Young's inequality yields that
\begin{equation*}
	\bE\left[\sup_{t\in[0,T]}|\fp_t|^2+\kappa\int^T_0|\fp_t|^2\,\diff t+\int^T_0|\fq_t|^2\,\diff t\right]\leq C\bE\left[|\fh|^2+\frac{1}{\kappa}\int^T_0|\fg_s|^2\,\diff s\right]+\frac{\kappa}{2}\bE\left[\int^T_0|\fp_s|^2\,\diff s\right],
\end{equation*}
and hence the desired estimate \eqref{BSEE_eq_BSDE-apriori} holds in this special case. In the rest of this proof, we assume that $\alpha\in(0,1)$.

Next, we show that
\begin{equation}\label{BSEE_eq_BSDE-estimate2}
\begin{split}
	&\bE\left[\sup_{t\in[0,T]}(T-t)^\alpha|\fp_t|^2+\kappa\int^T_0(T-t)^\alpha|\fp_t|^2\,\diff t+\int^T_0(T-t)^\alpha|\fq_t|^2\,\diff t\right]\\
	&\leq C\bE\left[\alpha\int^T_0(T-t)^{\alpha-1}|\fp_t|^2\,\diff s+\int^T_0(T-t)^\alpha|\fp_t||\fg_t|\,\diff t\right].
\end{split}
\end{equation}
From \eqref{BSEE_eq_BSDE-Ito}, we have
\begin{equation*}
	|\fp_t|^2+2\kappa\int^\tau_t|\fp_s|^2\,\diff s+\int^\tau_t|\fq_s|^2\,\diff s=|\fp_\tau|^2+2\int^\tau_t\langle\fp_s,\fg_s\rangle\,\diff s-2\int^\tau_t\langle\fp_s,\fq_s\rangle\,\diff W_s,\ \ 0\leq t\leq\tau\leq T.
\end{equation*}
Noting that $\alpha\in(0,1)$, multiplying $\alpha(T-\tau)^{\alpha-1}$ to both sides above and then integrating with respect to $\tau\in[t,T]$, we have
\begin{equation}\label{BSEE_eq_BSDE-Ito2}
\begin{split}
	&(T-t)^\alpha|\fp_t|^2+2\kappa\int^T_t(T-s)^\alpha|\fp_s|^2\,\diff s+\int^T_t(T-s)^\alpha|\fq_s|^2\,\diff s\\
	&=\alpha\int^T_t(T-s)^{\alpha-1}|\fp_s|^2\,\diff s+2\int^T_t(T-s)^\alpha\langle\fp_s,\fg_s\rangle\,\diff s-2\int^T_t(T-s)^\alpha\langle\fp_s,\fq_s\rangle\,\diff W_s,\ \ t\in[0,T].
\end{split}
\end{equation}
Here, we used $\int^T_t\alpha(T-\tau)^{\alpha-1}\,\diff\tau=(T-t)^\alpha$ and the (stochastic) Fubini theorem. The applications of the (stochastic) Fubini theorem are justified thanks to the facts that $\bE[\sup_{s\in[t,T]}|\fp_s|^2]<\infty$, $\bE[\int^T_t|\fq_s|^2\,\diff s]<\infty$ and $\bE[(\int^T_t|\fg_s|\,\diff s)^2]<\infty$; for example, observe that
\begin{equation*}
	\bE\left[\int^T_t\alpha(T-\tau)^{\alpha-1}\left(\int^\tau_t|\fp_s|^2|\fq_s|^2\,\diff s\right)^{1/2}\,\diff\tau\right]\leq(T-t)^\alpha\bE\left[\sup_{s\in[t,T]}|\fp_s|^2\right]^{1/2}\bE\left[\int^T_t|\fq_s|^2\,\diff s\right]^{1/2}<\infty,
\end{equation*}
and hence \cite[Theorem 2.2]{Ve11} ensures the interchangeability of the stochastic integral and the Lebesgue integral;
\begin{equation*}
	\int^T_t\alpha(T-\tau)^{\alpha-1}\int^\tau_t\langle\fp_s,\fq_s\rangle\,\diff W_s\,\diff\tau=\int^T_t\int^T_s\alpha(T-\tau)^{\alpha-1}\,\diff\tau\langle\fp_s,\fq_s\rangle\,\diff W_s=\int^T_t(T-s)^\alpha\langle\fp_s,\fq_s\rangle\,\diff W_s\ \ \text{a.s.}
\end{equation*}
Other terms can be justified similarly. On the one hand, taking the expectations in both sides of \eqref{BSEE_eq_BSDE-Ito2} with $t=0$, we have
\begin{equation}\label{BSEE_eq_BSDE-estimate2-1}
	\bE\left[2\kappa\int^T_0(T-s)^\alpha|\fp_s|^2\,\diff s+\int^T_0(T-s)^\alpha|\fq_s|^2\,\diff s\right]\leq\bE\left[\alpha\int^T_0(T-s)^{\alpha-1}|\fp_s|^2\,\diff s+2\int^T_0(T-s)^\alpha|\fp_s||\fg_s|\,\diff s\right].
\end{equation}
On the other hand, taking the supremum in $t\in[0,T]$ and then taking the expectations in both sides of \eqref{BSEE_eq_BSDE-Ito2}, the Burkholder--Davis--Gundy inequality and Young's inequality yield that
\begin{align*}
	&\bE\left[\sup_{t\in[0,T]}(T-t)^\alpha|\fp_t|^2\right]\\
	&\leq\bE\left[\alpha\int^T_0(T-s)^{\alpha-1}|\fp_s|^2\,\diff s+2\int^T_0(T-s)^\alpha|\fp_s||\fg_s|\,\diff s\right]+C\bE\left[\left(\int^T_0(T-s)^{2\alpha}|\fp_s|^2|\fq_s|^2\,\diff s\right)^{1/2}\right]\\
	&\leq\bE\left[\alpha\int^T_0(T-s)^{\alpha-1}|\fp_s|^2\,\diff s+2\int^T_0(T-s)^\alpha|\fp_s||\fg_s|\,\diff s\right]+C\bE\left[\int^T_0(T-s)^\alpha|\fq_s|^2\,\diff s\right]\\
	&\hspace{0.5cm}+\frac{1}{2}\bE\left[\sup_{s\in[0,T]}(T-s)^\alpha|\fp_s|^2\right].
\end{align*}
Hence,
\begin{equation}\label{BSEE_eq_BSDE-estimate2-2}
\begin{split}
	&\bE\left[\sup_{t\in[0,T]}(T-t)^\alpha|\fp_t|^2\right]\\
	&\leq C\bE\left[\alpha\int^T_0(T-s)^{\alpha-1}|\fp_s|^2\,\diff s+\int^T_0(T-s)^\alpha|\fp_s||\fg_s|\,\diff s\right]+C\bE\left[\int^T_0(T-s)^\alpha|\fq_s|^2\,\diff s\right].
\end{split}
\end{equation}
By \eqref{BSEE_eq_BSDE-estimate2-1} and \eqref{BSEE_eq_BSDE-estimate2-2}, we obtain \eqref{BSEE_eq_BSDE-estimate2}.

Next, we show that
\begin{equation}\label{BSEE_eq_BSDE-estimate3}
	\bE\left[\int^T_0(T-t)^{\beta-1}|\fp_t|^2\,\diff t\right]\leq(2\kappa)^{-\beta}\Gamma(\beta)\bE\left[|\fh|^2+2\int^T_0|\fp_t||\fg_t|\,\diff t\right]\ \ \text{for any $\beta\in(0,1)$.}
\end{equation}
Let $\beta\in(0,1)$ be fixed. From \eqref{BSEE_eq_BSDE-Ito}, by using It\^{o}'s formula, we have
\begin{equation*}
	|\fp_t|^2+\int^T_te^{-2\kappa(s-t)}|\fq_s|^2\,\diff s=e^{-2\kappa(T-t)}|\fh|^2+2\int^T_te^{-2\kappa(s-t)}\langle\fp_s,\fg_s\rangle\,\diff s-2\int^T_te^{-2\kappa(s-t)}\langle\fp_s,\fq_s\rangle\,\diff W_s,
\end{equation*}
and hence
\begin{equation*}
	\bE\big[|\fp_t|^2\big]\leq\bE\left[e^{-2\kappa(T-t)}|\fh|^2+2\int^T_te^{-2\kappa(s-t)}|\fp_s||\fg_s|\,\diff s\right]
\end{equation*}
for any $t\in[0,T]$. Multiplying $(T-t)^{\beta-1}$ to both sides in the above inequality and then integrating with respect to $t\in[0,T]$, we have, by Tonelli's theorem,
\begin{equation}\label{BSEE_eq_BSDE-estimate3-1}
\begin{split}
	&\bE\left[\int^T_0(T-t)^{\beta-1}|\fp_t|^2\,\diff t\right]\\
	&\leq\bE\left[\left(\int^T_0(T-t)^{\beta-1}e^{-2\kappa(T-t)}\,\diff t\right)\,|\fh|^2+2\int^T_0\left(\int^s_0(T-t)^{\beta-1}e^{-2\kappa(s-t)}\,\diff t\right)\,|\fp_s||\fg_s|\,\diff s\right].
\end{split}
\end{equation}
Noting that $\beta\in(0,1)$, we have
\begin{equation*}
	\int^s_0(T-t)^{\beta-1}e^{-2\kappa(s-t)}\,\diff t\leq\int^s_0(s-t)^{\beta-1}e^{-2\kappa(s-t)}\,\diff t\leq\int^\infty_0t^{\beta-1}e^{-2\kappa t}\,\diff t=(2\kappa)^{-\beta}\Gamma(\beta)
\end{equation*}
for any $s\in[0,T]$. From this estimate and \eqref{BSEE_eq_BSDE-estimate3-1}, we obtain \eqref{BSEE_eq_BSDE-estimate3}.

By \eqref{BSEE_eq_BSDE-estimate1}, \eqref{BSEE_eq_BSDE-estimate2} and \eqref{BSEE_eq_BSDE-estimate3} (applying to $\beta=1-\alpha$ and $\beta=\alpha$), noting that $\alpha\Gamma(\alpha)\leq1\leq\Gamma(1-\alpha)$ and using Young's inequality, we obtain
\begin{align*}
	&\bE\left[\sup_{t\in[0,T]}|\fp_t|^2+\kappa\int^T_0|\fp_t|^2\,\diff t+\int^T_0|\fq_t|^2\,\diff t\right]+\frac{\kappa^{1-\alpha}}{\Gamma(1-\alpha)}\bE\left[\int^T_0(T-t)^{-\alpha}|\fp_t|^2\,\diff t\right]\\
	&\hspace{0.5cm}+\kappa^\alpha\bE\left[\sup_{t\in[0,T]}(T-t)^\alpha|\fp_t|^2+\kappa\int^T_0(T-t)^\alpha|\fp_t|^2\,\diff t+\int^T_0(T-t)^\alpha|\fq_t|^2\,\diff t\right]\\
	&\leq C\bE\left[|\fh|^2+\int^T_0|\fp_t||\fg_t|\,\diff t+\kappa^\alpha\int^T_0(T-t)^\alpha|\fp_t||\fg_t|\,\diff t\right]\\
	&\leq C\bE\left[|\fh|^2+\frac{\Gamma(1-\alpha)}{\kappa^{1-\alpha}}\int^T_0(T-t)^{\alpha}|\fg_t|^2\,\diff t\right]\\
	&\hspace{0.5cm}+\frac{\kappa^{1-\alpha}}{\Gamma(1-\alpha)}\bE\left[\int^T_0(T-t)^{-\alpha}|\fp_t|^2\,\diff t\right]+\frac{\kappa^{1+\alpha}}{2}\bE\left[\int^T_0(T-t)^\alpha|\fp_t|^2\,\diff t\right].
\end{align*}
Hence, the desired estimate \eqref{BSEE_eq_BSDE-apriori} holds. This completes the proof.
\end{proof}

Next, we show the well-posedness and provides a moment estimate for the BSEE \eqref{BSEE_eq_BSEE} assuming that the generator $\cG$ is independent of the solution. Recall the definition \eqref{BSEE_eq_norm} of the norm $\|\cdot\|_{\bH_\beta}$ of the Hilbert space $\bH_\beta$ for $\beta\in\bR$.

%% Lemma

\begin{lemm}\label{BSEE_lemm_trivialBSEE}
Let \cref{BSEE_assum_coefficient} hold, and assume furthermore that $\cG(t,\cP,\cQ)=\cG(t,0,0)=:\cG_t$ for any $(\omega,t,\cP,\cQ)\in\Omega\times[0,T]\times\bH_{1+\alpha}\times\bH_\alpha$. Then, there exists a unique solution $(\cP,\cQ):\Omega\times[0,T]\to\bH_{1+\alpha}\times\bH_\alpha$ to the corresponding BSEE:
\begin{equation}\label{BSEE_eq_trivialBSEE}
	\begin{dcases}
	\diff\cP_t(\vth)=\varpi(\vth)\cP_t(\vth)\,\diff t-\cG_t(\vth)\,\diff t+\cQ_t(\vth)\,\diff W_t,\ \ \vth\in\Theta,\ t\in[0,T],\\
	\cP_T(\vth)=\Phi(\vth),\ \ \vth\in\Theta.
	\end{dcases}
\end{equation}
Furthermore, for any $\lambda\geq1$, it holds that
\begin{equation}\label{BSEE_eq_trivialBSEE-estimate}
\begin{split}
	&\bE\left[\esssup_{t\in[0,T]}e^{2\lambda t}\|\cP_t\|_{\bH_0}^2+\int^T_0e^{2\lambda t}\|\cP_t\|_{\bH_1}^2\,\diff t+\int^T_0e^{2\lambda t}\|\cQ_t\|_{\bH_0}^2\,\diff t\right.\\
	&\left.\hspace{1cm}+\esssup_{t\in[0,T]}(T-t)^\alpha e^{2\lambda t}\|\cP_t\|_{\bH_\alpha}^2+\int^T_0(T-t)^\alpha e^{2\lambda t}\|\cP_t\|_{\bH_{1+\alpha}}^2\,\diff t+\int^T_0(T-t)^\alpha e^{2\lambda t}\|\cQ_t\|_{\bH_\alpha}^2\,\diff t\right]\\
	&\leq C_0\bE\left[e^{2\lambda T}\|\Phi\|_{\bH_0}^2+\frac{\Gamma(1-\alpha)}{\lambda^{1-\alpha}}\int^T_0(T-t)^\alpha e^{2\lambda t}\|\cG_t\|_{\bH_0}^2\,\diff t\right],
\end{split}
\end{equation}
where $C_0>0$ is the constant appearing in \cref{BSEE_lemm_BSDE-apriori}.
\end{lemm}

%% Proof

\begin{proof}
\underline{Uniqueness of the solution of the BSEE \eqref{BSEE_eq_trivialBSEE}.} Let $(\cP^1,\cQ^1)$ and $(\cP^2,\cQ^2)$ be two solutions of the BSEE \eqref{BSEE_eq_trivialBSEE} with respective good representatives $(\bar{p}^1,\bar{q}^1,\bar{g}^1)$ and $(\bar{p}^2,\bar{q}^2,\bar{g}^2)$; see \cref{BSEE_defi_solution}. By the definition, there exists $\Theta_*\in\Sigma$ with $\nu(\Theta\setminus\Theta_*)=0$ such that, for any $\vth\in\Theta_*$, $\bar{p}^1_T(\vth)=\bar{p}^2_T(\vth)$ a.s.\ and $\bar{g}^1_t(\vth)=\bar{g}^2_t(\vth)$ a.s.\ for a.e.\ $t\in[0,T]$. Noting the integrability condition \eqref{BSEE_eq_L^2-solution} in \cref{BSEE_defi_solution}, for any $\vth\in\Theta_*$, the pairs $(\bar{p}^1(\vth),\bar{q}^1(\vth))$ and $(\bar{p}^2(\vth),\bar{q}^2(\vth))$ are adapted solutions of the same BSDE on the Euclidean space $E$. By the uniqueness of the adapted solution (cf.\ \cite[Theorem 4.3.1]{Zh17}), we see that $\bar{p}^1_t(\vth)=\bar{p}^2_t(\vth)$ and $\bar{q}^1_t(\vth)=\bar{q}^2_t(\vth)$ a.s.\ for a.e.\ $t\in[0,T]$ for any $\vth\in\Theta_*$. Since $[\bar{p}^i_t(\cdot)]=\cP^i_t$ and $[\bar{q}^i_t(\cdot)]=\cQ^i_t$ a.s.\ for a.e.\ $t\in[0,T]$ (where $[\psi(\cdot)]$ denotes the $\nu$-equivalence class of a measurable function $\psi:\Theta\to E$), we see that $(\cP^1_t,\cQ^1_t)=(\cP^2_t,\cQ^2_t)$ in $\bH_{1+\alpha}\times\bH_\alpha$ a.s.\ for a.e.\ $t\in[0,T]$. This shows the uniqueness of the solution to the BSEE \eqref{BSEE_eq_trivialBSEE}.

\underline{Existence of the solution of the BSEE \eqref{BSEE_eq_trivialBSEE} satisfying \eqref{BSEE_eq_trivialBSEE-estimate}.} By \cref{app_lemm_measurability} (ii), there exist $\cF_T\otimes\Sigma$-measurable map $\bar{\varphi}:\Omega\times\Theta\to E$ and (jointly) progressively measurable map $\bar{g}:\Omega\times[0,T]\times\Theta\to E$ such that $[\bar{\varphi}(\cdot)]=\Phi$ and $[\bar{g}_t(\cdot)]=\cG_t$ for any $t\in[0,T]$ and $\omega\in\Omega$. By \eqref{BSEE_eq_BSEE-generator} (with $\cG(t,\cP,\cQ)=\cG_t$) and \cref{app_lemm_measurability} (iii), we can find progressively measurable maps $\bar{\eta}:\Omega\times[0,T]\times\Theta\to E$ and $\bar{\zeta}:\Omega\times[0,T]\times[0,T]\times\Theta\to E$ with
\begin{equation*}
	\bE\left[\int_\Theta\int^T_0|\bar{\eta}_t(\vth)|^2\,\diff t\,\nu(\diff\vth)\right]+\bE\left[\int_\Theta\int^T_0(T-s)^\alpha\int^T_0|\bar{\zeta}_t(s,\vth)|^2\,\diff t\,\diff s\,\nu(\diff\vth)\right]<\infty
\end{equation*}
such that, for $\nu$-a.e.\ $\vth\in\Theta$,
\begin{equation}\label{BSEE_eq_etazeta-representation}
\begin{split}
	&\bar{\varphi}(\vth)=\bE\big[\bar{\varphi}(\vth)\big]+\int^T_0\bar{\eta}_t(\vth)\,\diff W_t\ \ \text{a.s.}\ \ \text{and}\\
	&\bar{g}_s(\vth)=\bE\big[\bar{g}_s(\vth)\big]+\int^T_0\bar{\zeta}_t(s,\vth)\,\diff W_t\ \ \text{a.s.\ for a.e.\ $s\in[0,T]$.}
\end{split}
\end{equation}
Hence, there exist jointly measurable versions (still denoted by $\bar{\varphi},\bar{g},\bar{\eta},\bar{\zeta}$) such that the representation formula \eqref{BSEE_eq_etazeta-representation} and the integrability condition
\begin{equation}\label{BSEE_eq_etazeta-integrability}
\begin{split}
	&\bE\left[|\bar{\varphi}(\vth)|^2+\int^T_0(T-s)^\alpha|\bar{g}_s(\vth)|^2\,\diff s+\int^T_0|\bar{\eta}_t(\vth)|^2\,\diff t+\int^T_0(T-s)^\alpha\int^T_0|\bar{\zeta}_t(s,\vth)|^2\,\diff t\,\diff s\right]<\infty
\end{split}
\end{equation}
hold for any $\vth\in\Theta$. Notice that the progressive measurability of $(\bar{g}_s(\vth))_{s\in[0,T]}$ implies $\bar{\zeta}_t(s,\vth)=0$ a.s.\ for a.e.\ $t>s$ for any $\vth\in\Theta$. Define $\bar{p},\bar{q}:\Omega\times[0,T]\times\Theta\to E$ by
\begin{align*}
	&\bar{p}_t(\vth):=\bE_t\left[e^{-\varpi(\vth)(T-t)}\bar{\varphi}(\vth)+\int^T_te^{-\varpi(\vth)(s-t)}\bar{g}_s(\vth)\,\diff s\right],\\
	&\bar{q}_t(\vth):=e^{-\varpi(\vth)(T-t)}\bar{\eta}_t(\vth)+\int^T_te^{-\varpi(\vth)(s-t)}\bar{\zeta}_t(s,\vth)\,\diff s,
\end{align*}
for $(\omega,t,\vth)\in\Omega\times[0,T]\times\Theta$. Here, recall that $\bE_t[\cdot]=\bE[\cdot|\cF_t]$ denotes the conditional expectation operator with respect to $\cF_t$. Clearly, $\bar{p}$ and $\bar{q}$ are progressively measurable. Furthermore, for any $\vth\in\Theta$, $\bar{p}(\vth)=(\bar{p}_t(\vth))_{t\in[0,T]}$ is an $E$-valued continuous semimartingale with $\bar{p}_T(\vth)=\bar{\varphi}(\vth)$ a.s. In particular, $[\bar{p}_T(\cdot)]=\Phi$ a.s. From the integrability condition \eqref{BSEE_eq_etazeta-integrability} and $\alpha\in[0,1)$, one can easily check that
\begin{equation*}
	\bE\left[\sup_{t\in[0,T]}|\bar{p}_t(\vth)|^2+\int^T_0|\bar{q}_t(\vth)|^2\,\diff t+\left(\int^T_0|\bar{g}_t(\vth)|\,\diff t\right)^2\right]<\infty\ \ \text{for any $\vth\in\Theta$.}
\end{equation*}
Let $\vth\in\Theta$ be fixed. Observe that
\begin{align}
	\nonumber
	&e^{-\varpi(\vth)t}\bar{p}_t(\vth)-\bar{p}_0(\vth)+\int^t_0e^{-\varpi(\vth)s}\bar{g}_s(\vth)\,\diff s\\
	\nonumber
	&=e^{-\varpi(\vth)T}\left\{\bE_t\big[\bar{\varphi}(\vth)\big]-\bE\big[\bar{\varphi}(\vth)\big]\right\}+\int^T_0e^{-\varpi(\vth)s}\left\{\bE_t\big[\bar{g}_s(\vth)\big]-\bE\big[\bar{g}_s(\vth)\big]\right\}\,\diff s\\
	\nonumber
	&=e^{-\varpi(\vth)T}\int^t_0\bar{\eta}_s(\vth)\,\diff W_s+\int^T_0e^{-\varpi(\vth)s}\int^{s\wedge t}_0\bar{\zeta}_r(s,\vth)\,\diff W_r\,\diff s\\
	\nonumber
	&=\int^t_0e^{-\varpi(\vth)s}\left\{e^{-\varpi(\vth)(T-s)}\bar{\eta}_s(\vth)+\int^T_se^{-\varpi(\vth)(r-s)}\bar{\zeta}_s(r,\vth)\,\diff r\right\}\,\diff W_s\\
	\label{BSEE_eq_Ito-bar}
	&=\int^t_0e^{-\varpi(\vth)s}\bar{q}_s(\vth)\,\diff W_s\ \ \text{a.s.\ for any $t\in[0,T]$,}
\end{align}
where we used the representation formula \eqref{BSEE_eq_etazeta-representation} in the second equality and the stochastic Fubini theorem in the third equality. The application of the stochastic Fubini theorem can be easily justified by noting the progressive measurability and the integrability \eqref{BSEE_eq_etazeta-integrability} of the map $(\omega,r,s)\mapsto\bar{\zeta}_r(s,\vth)$. By \eqref{BSEE_eq_Ito-bar}, It\^{o}'s formula yields that the $E$-valued continuous semimartingale $\bar{p}(\vth)=(\bar{p}_t(\vth))_{t\in[0,T]}$ satisfies
\begin{equation*}
	\diff\bar{p}_t(\vth)=\varpi(\vth)\bar{p}_t(\vth)-\bar{g}_t(\vth)\,\diff t+\bar{q}_t(\vth)\,\diff W_t,\ \ t\in[0,T].
\end{equation*}

Fix a constant $\lambda\geq1$. Observe that, for any $\vth\in\Theta$, the pair $((e^{\lambda t}\bar{p}_t(\vth))_{t\in[0,T]},(e^{\lambda t}\bar{q}_t(\vth))_{t\in[0,T]})$ is an adapted solution of the BSDE \eqref{BSEE_eq_BSDE} on $E$ with
\begin{equation*}
	\fh:=e^{\lambda T}\bar{\varphi}(\vth),\ \ \fg_t:=e^{\lambda t}\bar{g}_t(\vth),\ \ \text{and}\ \ \kappa=\lambda+\varpi(\vth).
\end{equation*}
Hence, by \cref{BSEE_lemm_BSDE-apriori}, we have
\begin{align*}
	&\bE\left[\sup_{t\in[0,T]}e^{2\lambda t}|\bar{p}_t(\vth)|^2+(\lambda+\varpi(\vth))\int^T_0e^{2\lambda t}|\bar{p}_t(\vth)|^2\,\diff t+\int^T_0e^{2\lambda t}|\bar{q}_t(\vth)|^2\,\diff t\right.\\
	&\left.\hspace{1cm}+(\lambda+\varpi(\vth))^\alpha\sup_{t\in[0,T]}(T-t)^\alpha e^{2\lambda t}|\bar{p}_t(\vth)|^2\right.\\
	&\left.\hspace{1cm}+(\lambda+\varpi(\vth))^{1+\alpha}\int^T_0(T-t)^\alpha e^{2\lambda t}|\bar{p}_t(\vth)|^2\,\diff t+(\lambda+\varpi(\vth))^\alpha \int^T_0(T-t)^\alpha e^{2\lambda t}|\bar{q}_t(\vth)|^2\,\diff t\right]\\
	&\leq C_0\bE\left[e^{2\lambda T}|\bar{\varphi}(\vth)|^2+\frac{\Gamma(1-\alpha)}{(\lambda+\varpi(\vth))^{1-\alpha}}\int^T_0(T-t)^\alpha e^{2\lambda t}|\bar{g}_t(\vth)|^2\,\diff t\right],
\end{align*}
where $C_0>0$ is the constant appearing in \cref{BSEE_lemm_BSDE-apriori}. Noting that $\lambda\geq1$, $\varpi(\vth)\geq0$ and $\alpha\in[0,1)$, we obtain
\begin{align*}
	&\bE\left[\sup_{t\in[0,T]}e^{2\lambda t}|\bar{p}_t(\vth)|^2+(1+\varpi(\vth))\int^T_0e^{2\lambda t}|\bar{p}_t(\vth)|^2\,\diff t+\int^T_0e^{2\lambda t}|\bar{q}_t(\vth)|^2\,\diff t\right.\\
	&\left.\hspace{1cm}+(1+\varpi(\vth))^\alpha\sup_{t\in[0,T]}(T-t)^\alpha e^{2\lambda t}|\bar{p}_t(\vth)|^2\right.\\
	&\left.\hspace{1cm}+(1+\varpi(\vth))^{1+\alpha}\int^T_0(T-t)^\alpha e^{2\lambda t}|\bar{p}_t(\vth)|^2\,\diff t+(1+\varpi(\vth))^\alpha \int^T_0(T-t)^\alpha e^{2\lambda t}|\bar{q}_t(\vth)|^2\,\diff t\right]\\
	&\leq C_0\bE\left[e^{2\lambda T}|\bar{\varphi}(\vth)|^2+\frac{\Gamma(1-\alpha)}{\lambda^{1-\alpha}}\int^T_0(T-t)^\alpha e^{2\lambda t}|\bar{g}_t(\vth)|^2\,\diff t\right].
\end{align*}
Integrating both sides above by $\nu$ with respect to $\vth\in\Theta$ and using Tonelli's theorem (recall that the measure $\nu$ is $\sigma$-finite by the assumption), we get
\begin{equation}\label{BSEE_eq_trivialBSEE-estimate1}
\begin{split}
	&\bE\left[\sup_{t\in[0,T]}e^{2\lambda t}\|\bar{p}_t\|_{\bH_0}^2+\int^T_0e^{2\lambda t}\|\bar{p}_t\|_{\bH_1}^2\,\diff t+\int^T_0e^{2\lambda t}\|\bar{q}_t\|_{\bH_0}^2\,\diff t\right.\\
	&\left.\hspace{1cm}+\sup_{t\in[0,T]}(T-t)^\alpha e^{2\lambda t}\|\bar{p}_t\|_{\bH_\alpha}^2+\int^T_0(T-t)^\alpha e^{2\lambda t}\|\bar{p}_t\|_{\bH_{1+\alpha}}^2\,\diff t+\int^T_0(T-t)^\alpha e^{2\lambda t}\|\bar{q}_t\|_{\bH_\alpha}^2\,\diff t\right]\\
	&\leq C_0\bE\left[e^{2\lambda T}\|\bar{\varphi}\|_{\bH_0}^2+\frac{\Gamma(1-\alpha)}{\lambda^{1-\alpha}}\int^T_0(T-t)^\alpha e^{2\lambda t}\|\bar{g}_t\|_{\bH_0}^2\,\diff t\right].
\end{split}
\end{equation}
The last line above is finite thanks to the assumption \eqref{BSEE_eq_BSEE-generator} (with $\cG(t,\cP,\cQ)=\cG_t$). Now we define $\cP:\Omega\times[0,T]\to\bH_{1+\alpha}$ and $\cQ:\Omega\times[0,T]\to\bH_\alpha$ by
\begin{equation*}
	\cP_t:=
	\begin{dcases}
	[\bar{p}_t(\cdot)]\ \ &\text{if $\|\bar{p}_t\|_{\bH_{1+\alpha}}<\infty$,}\\
	0\ \ &\text{otherwise,}
	\end{dcases}
	\ \ \text{and}\ \ \cQ_t:=
	\begin{dcases}
	[\bar{q}_t(\cdot)]\ \ &\text{if $\|\bar{q}_t\|_{\bH_{\alpha}}<\infty$,}\\
	0\ \ &\text{otherwise,}
	\end{dcases}
\end{equation*}
where $[\psi(\cdot)]$ denotes the $\nu$-equivalence class of a measurable function $\psi:\Theta\to E$. By \cref{app_lemm_measurability} (i), we see that $\cP$ and $\cQ$ are ($\bH_{1+\alpha}$- and $\bH_\alpha$-valued, respectively) progressively measurable processes. The estimate \eqref{BSEE_eq_trivialBSEE-estimate1} implies that $\cP_t=[\bar{p}_t(\cdot)]$ and $\cQ_t=[\bar{q}_t(\cdot)]$ a.s.\ for a.e.\ $t\in[0,T]$ and that $(\cP,\cQ)$ satisfies \eqref{BSEE_eq_trivialBSEE-estimate}. By the construction, $(\cP,\cQ)$ is the solution of the BSEE \eqref{BSEE_eq_trivialBSEE} with the good representative $(\bar{p},\bar{q},\bar{g})$. This completes the proof.
\end{proof}

Now we can prove \cref{BSEE_theo_BSEE}.

%% Proof

\begin{proof}[Proof of \cref{BSEE_theo_BSEE}]
Let \cref{BSEE_assum_coefficient} hold. We denote by $\bS_{T,\alpha}$ the set of all pairs $(\cP,\cQ)$ of the $\bP\otimes\diff t$-equivalence classes of progressively measurable processes $\cP:\Omega\times[0,T]\to\bH_{1+\alpha}$ and $\cQ:\Omega\times[0,T]\to\bH_\alpha$ such that
\begin{equation*}
	\|(\cP,\cQ)\|_{\bS_{T,\alpha}}:=\bE\left[\int^T_0(T-t)^\alpha\|\cP_t\|_{\bH_{1+\alpha}}^2\,\diff t+\int^T_0(T-t)^\alpha\|\cQ_t\|_{\bH_\alpha}^2\,\diff t\right]^{1/2}<\infty.
\end{equation*}
Then, $(\bS_{T,\alpha},\|\cdot\|_{\bS_{T,\alpha}})$ is a Banach space. Furthermore, for each $\lambda\geq1$, define
\begin{equation*}
	\|(\cP,\cQ)\|_{\bS_{T,\alpha,\lambda}}:=\bE\left[\int^T_0(T-t)^\alpha e^{2\lambda t}\|\cP_t\|_{\bH_{1+\alpha}}^2\,\diff t+\int^T_0(T-t)^\alpha e^{2\lambda t}\|\cQ_t\|_{\bH_\alpha}^2\,\diff t\right]^{1/2},\ \ (\cP,\cQ)\in\bS_{T,\alpha}.
\end{equation*}
Clearly, $\|\cdot\|_{\bS_{T,\alpha,\lambda}}$ is another norm on $\bS_{T,\alpha}$ and equivalent to the original one $\|\cdot\|_{\bS_{T,\alpha}}$.

By \cref{BSEE_lemm_trivialBSEE}, every solution of the BSEE \eqref{BSEE_eq_BSEE} satisfying the integrability condition \eqref{BSEE_eq_solution-integrability} must belong to the space $\bS_{T,\alpha}$. Take an arbitrary $(\cP,\cQ)\in\bS_{T,\alpha}$. By \cref{BSEE_assum_coefficient}, we see that
\begin{equation*}
	\bE\left[\|\Phi\|^2_{\bH_0}+\int^T_0(T-t)^\alpha\|\cG(t,\cP_t,\cQ_t)\|^2_{\bH_0}\,\diff t\right]<\infty.
\end{equation*}
Hence, by \cref{BSEE_lemm_trivialBSEE}, there exists a unique solution $(\tilde{\cP},\tilde{\cQ})\in\bS_{T,\alpha}$ to the BSEE
\begin{equation*}
	\begin{dcases}
	\diff\tilde{\cP}_t(\vth)=\varpi(\vth)\tilde{\cP}_t(\vth)\,\diff t-\cG(t,\cP_t,\cQ_t)(\vth)\,\diff t+\tilde{\cQ}_t(\vth)\,\diff W_t,\ \ \vth\in\Theta,\ t\in[0,T],\\
	\tilde{\cP}_T(\vth)=\Phi(\vth),\ \ \vth\in\Theta.
	\end{dcases}
\end{equation*}
Set $\Xi_{T,\alpha}(\cP,\cQ):=(\tilde{\cP},\tilde{\cQ})$. Notice that an element $(\cP,\cQ)\in\bS_{T,\alpha}$ is a solution to the BSEE \eqref{BSEE_eq_BSEE} if and only if it is a fixed point of the map $\Xi_{T,\alpha}:\bS_{T,\alpha}\to\bS_{T,\alpha}$. In the following, we will show that the map $\Xi_{T,\alpha}$ is contractive on the Banach space $\bS_{T,\alpha}$ under the equivalent norm $\|\cdot\|_{\bS_{T,\alpha,\lambda}}$ for a sufficiently large $\lambda\geq1$, showing the existence and uniqueness of the fixed point.

Fix $\lambda\geq1$, which will be determined later. Let $(\cP^1,\cQ^1)$ and $(\cP^2,\cQ^2)$ be two elements of $\bS_{T,\alpha}$. Define $(\tilde{\cP}^i,\tilde{\cQ}^i):=\Xi_{T,\alpha}(\cP^i,\cQ^i)$ for $i=1,2$. Then $(\tilde{\cP}^1-\tilde{\cP}^2,\tilde{\cQ}^1-\tilde{\cQ}^2)\in\bS_{T,\alpha}$ is the unique solution to the BSEE
\begin{equation*}
	\begin{dcases}
	\diff(\tilde{\cP}^1_t-\tilde{\cP}^2_t)(\vth)=\varpi(\vth)(\tilde{\cP}^1_t-\tilde{\cP}^2_t)(\vth)\,\diff t-\big(\cG(t,\cP^1_t,\cQ^1_t)-\cG(t,\cP^2_t,\cQ^2_t)\big)(\vth)\,\diff t+(\tilde{\cQ}^1_t-\tilde{\cQ}^2_t)(\vth)\,\diff W_t,\\
	\hspace{6cm}\vth\in\Theta,\ t\in[0,T],\\
	(\tilde{\cP}^1_T-\tilde{\cP}^2_T)(\vth)=0,\ \ \vth\in\Theta.
	\end{dcases}
\end{equation*}
Hence, by the estimate \eqref{BSEE_eq_trivialBSEE-estimate} in \cref{BSEE_lemm_trivialBSEE}, it holds that
\begin{equation*}
	\big\|(\tilde{\cP}^1-\tilde{\cP}^2,\tilde{\cQ}^1-\tilde{\cQ}^2)\big\|_{\bS_{T,\alpha,\lambda}}^2\leq C_0\frac{\Gamma(1-\alpha)}{\lambda^{1-\alpha}}\bE\left[\int^T_0(T-t)^\alpha e^{2\lambda t}\|\cG(t,\cP^1_t,\cQ^1_t)-\cG(t,\cP^2_t,\cQ^2_t)\|^2_{\bH_0}\,\diff t\right],
\end{equation*}
where $C_0>0$ is the constant appearing in \cref{BSEE_lemm_BSDE-apriori}. By the assumption \eqref{BSEE_eq_BSEE-Lip}, we see that
\begin{align*}
	\big\|(\tilde{\cP}^1-\tilde{\cP}^2,\tilde{\cQ}^1-\tilde{\cQ}^2)\big\|_{\bS_{T,\alpha,\lambda}}^2&\leq C_0\frac{\Gamma(1-\alpha)}{\lambda^{1-\alpha}}\bE\left[\int^T_0(T-t)^\alpha e^{2\lambda t}2L^2\Big\{\|\cP^1_t-\cP^2_t\|_{\bH_{1+\alpha}}^2+\|\cQ^1_t-\cQ^2_t\|_{\bH_\alpha}^2\Big\}\,\diff t\right]\\
	&=2C_0L^2\frac{\Gamma(1-\alpha)}{\lambda^{1-\alpha}}\big\|(\cP^1-\cP^2,\cQ^1-\cQ^2)\big\|_{\bS_{T,\alpha,\lambda}}^2.
\end{align*}
Therefore, noting that $\alpha\in[0,1)$ and taking $\lambda\geq1$ (depending only on $L$ and $\alpha$) such that $2C_0L^2\frac{\Gamma(1-\alpha)}{\lambda^{1-\alpha}}<1$, the map $\Xi_{T,\alpha}$ becomes contractive on the Banach space $\bS_{T,\alpha}$ under the norm $\|\cdot\|_{\bS_{T,\alpha,\lambda}}$. Hence, by Banach's fixed point theorem, the map $\Xi_{T,\alpha}:\bS_{T,\alpha}\to\bS_{T,\alpha}$ has a unique fixed point. This means that the BSEE \eqref{BSEE_eq_BSEE} has a unique solution satisfying \eqref{BSEE_eq_solution-integrability} and that the solution belongs to $\bS_{T,\alpha}$.

Next, we show the estimate \eqref{BSEE_eq_BSEE-estimate}. Let $(\cP,\cQ)$ be the solution to the BSEE \eqref{BSEE_eq_BSEE} satisfying \eqref{BSEE_eq_solution-integrability}. Again by the estimate \eqref{BSEE_eq_trivialBSEE-estimate} in \cref{BSEE_lemm_trivialBSEE} and the assumption \eqref{BSEE_eq_BSEE-Lip}, it holds that
\begin{align*}
	&\bE\left[\esssup_{t\in[0,T]}e^{2\lambda t}\|\cP_t\|_{\bH_0}^2+\int^T_0e^{2\lambda t}\|\cP_t\|_{\bH_1}^2\,\diff t+\int^T_0e^{2\lambda t}\|\cQ_t\|_{\bH_0}^2\,\diff t\right.\\
	&\left.\hspace{1cm}+\esssup_{t\in[0,T]}(T-t)^\alpha e^{2\lambda t}\|\cP_t\|_{\bH_\alpha}^2+\int^T_0(T-t)^\alpha e^{2\lambda t}\|\cP_t\|_{\bH_{1+\alpha}}^2\,\diff t+\int^T_0(T-t)^\alpha e^{2\lambda t}\|\cQ_t\|_{\bH_\alpha}^2\,\diff t\right]\\
	&\leq C_0\bE\left[e^{2\lambda T}\|\Phi\|_{\bH_0}^2+\frac{\Gamma(1-\alpha)}{\lambda^{1-\alpha}}\int^T_0(T-t)^\alpha e^{2\lambda t}\|\cG(t,\cP_t,\cQ_t)\|_{\bH_0}^2\,\diff t\right]\\
	&\leq C_0\bE\left[e^{2\lambda T}\|\Phi\|_{\bH_0}^2+\frac{\Gamma(1-\alpha)}{\lambda^{1-\alpha}}\int^T_0(T-t)^\alpha e^{2\lambda t}3\big\{\|\cG(t,0,0)\|_{\bH_0}^2+L^2\|\cP_t\|_{\bH_{1+\alpha}}^2+L^2\|\cQ_t\|_{\bH_\alpha}^2\big\}\,\diff t\right].
\end{align*}
Hence, noting that $(\cP,\cQ)\in\bS_{T,\alpha}$, taking $\lambda\geq1$ (depending only on $L$ and $\alpha$) such that $3C_0L^2\frac{\Gamma(1-\alpha)}{\lambda^{1-\alpha}}\leq\frac{1}{2}$, we obtain the desired estimate \eqref{BSEE_eq_BSEE-estimate}. This completes the proof.
\end{proof}

%%%%%%%%%%%%%%%%%%%%%%%%%%%%%%%%%%
%%%%%%%%%%%%%%%%%%%%%%%%%%%%%%%%%%
%% Appendix
%%%%%%%%%%%%%%%%%%%%%%%%%%%%%%%%%%
%%%%%%%%%%%%%%%%%%%%%%%%%%%%%%%%%%

\appendix
\setcounter{theo}{0}
\setcounter{equation}{0}

\section*{Appendix}

%%%%%%%%%%%%%%
%% Section
%%%%%%%%%%%%%%

\section{Remarks on measurability issues}\label{app-measurability}

In the study of BSEEs introduced in this paper, we sometimes identify a function-space-valued measurable map with a jointly measurable multi-variate map. Let us make remarks on the measurability issues for the identification of these two objectives.

%% Lemma

\begin{lemm}\label{app_lemm_measurability}
Let $(A,\cA)$ be a measurable space, and let $(B,\cB,\fm)$ be a countably generated $\sigma$-finite measure space. Denote by $(L^2(B;E),\|\cdot\|_{L^2(B;E)},\langle\cdot,\cdot\rangle_{L^2(B;E)})$ the separable Hilbert space of equivalence classes of square-integrable measurable functions from $(B,\cB,\fm)$ to a Euclidean space $E$. For each measurable function $\psi:B\to E$, we denote its $\fm$-equivalence class by $[\psi]$.
\begin{itemize}
\item[(i)]
Let $\bar{f}:A\times B\to E$ be a jointly measurable map (that is, $\cA\otimes\cB$-measurable) such that $\int_B|\bar{f}(a,b)|^2\,\fm(\diff b)<\infty$ for each $a\in A$. Define $f(a):=[\bar{f}(a,\cdot)]$. Then, the map $a\mapsto f(a)$ is a (strongly) measurable map from $(A,\cA)$ to $L^2(B;E)$ equipped with its Borel $\sigma$-field.
\item[(ii)]
Let $f:A\to L^2(B;E)$ be a (strongly) measurable map. Then, there exists a jointly measurable map $\bar{f}:A\times B\to E$ such that $f(a)=[\bar{f}(a,\cdot)]$ for any $a\in A$.
\item[(iii)]
Let $\bar{\varphi}:\Omega\times B\to E$ be an $\cF_T\otimes\cB$-measurable map such that $\bE[\int_B|\bar{\varphi}(b)|^2\fm(\diff b)]<\infty$. Then, there exists a progressively measurable map $\bar{\zeta}:\Omega\times[0,T]\times B\to E$ such that
\begin{equation}\label{app_eq_zeta-integrability}
	\bE\left[\int_B\int^T_0|\bar{\zeta}_t(b)|^2\,\diff t\,\fm(\diff b)\right]<\infty
\end{equation}
and
\begin{equation}\label{app_eq_zeta-representation}
	\bar{\varphi}(b)=\bE[\bar{\varphi}(b)]+\int^T_0\bar{\zeta}_t(b)\,\diff W_t\ \ \text{a.s.\ for $\fm$-a.e.\ $b\in B$.}
\end{equation}
\end{itemize}
\end{lemm}

%% Proof

\begin{proof}
\begin{itemize}
\item[(i)]
For any $\eta\in L^2(B;E)$ (with a measurable representative $\bar{\eta}:B\to E$), Fubini's theorem shows that $A\ni a\mapsto\langle\eta,f(a)\rangle_{L^2(B;E)}=\int_B\langle\bar{\eta}(b),\bar{f}(a,b)\rangle\,\fm(\diff b)\in\bR$ is measurable, which means that $a\mapsto f(a)$ is weakly measurable and equivalently (thanks to the separability of $L^2(B;E)$) strongly measurable.
\item[(ii)]
Let $\{x_k\}_{k\in\bN}$ be a countable dense subset of $L^2(B;E)$ (with measurable representatives $\bar{x}_k:B\to E$, $k\in\bN$). Define a map $\bar{f}_n:A\times B\to E$ for each $n\in\bN$ by
\begin{equation*}
	\bar{f}_n(a,b):=\sum_{k\in\bN}\1_{A^n_k}(a)\bar{x}_k(b),\ \ (a,b)\in A\times B.
\end{equation*}
Here, $A^n_k$, $k\in\bN$, are defined inductively by $A^n_1:=f^{-1}(B(x_1,2^{-n}))$ and $A^n_{k+1}:=f^{-1}(B(x_{k+1},2^{-n}))\setminus(\bigcup^k_{j=1}A^n_j)$ for $k\in\bN$, where $B(x,r)$ denotes the open ball in $L^2(B;E)$ with center $x\in L^2(B;E)$ and radius $r>0$. Since $f:A\to L^2(B;E)$ is strongly measurable, we see that $A^n_k\in\cA$ for any $n,k\in\bN$, and $\{A^n_k\}_{k\in\bN}$ is a disjoint family covering $A$ for each $n\in\bN$. Hence, the map $\bar{f}_n:A\times B\to E$ is well-defined and jointly measurable. Furthermore, $\|f(a)-[\bar{f}_n(a,\cdot)]\|_{L^2(B;E)}\leq2^{-n}$ for any $a\in A$ and $n\in\bN$. By Minkowski's integral inequality, the above implies that the function $b\mapsto F(a,b):=\sum^\infty_{n=1}|\bar{f}_{n+1}(a,b)-\bar{f}_n(a,b)|$ is square integrable with respect to the measure $\fm$ for each $a\in A$. Now we set $\bar{f}(a,b):=\lim_{n\to\infty}\bar{f}_n(a,b)$ if the limit exists in $E$ and $\bar{f}(a,b):=0$ otherwise. Then $\bar{f}:A\times B\to E$ is jointly measurable and such that $\bar{f}_n(a,b)\to\bar{f}(a,b)$ as $n\to\infty$ for $\fm$-a.e.\ $b\in B$ for any $a\in A$. Noting that $|\bar{f}_n(a,b)|\leq|\bar{f}_1(a,b)|+F(a,b)$, by the dominated convergence theorem, $[\bar{f}_n(a,\cdot)]\to[\bar{f}(a,\cdot)]$ in $L^2(B;E)$ as $n\to\infty$ for any $a\in A$. Hence, it holds that $f(a)=[\bar{f}(a,\cdot)]$ in $L^2(B;E)$ for any $a\in A$.
\item[(iii)]
Without loss of generality, we may assume that $\bE[\bar{\varphi}(b)]=0$ for any $b\in B$. Define $\varphi(\omega):=[\bar{\varphi}(\omega,\cdot)]$ if $\int_B|\bar{\varphi}(\omega,b)|^2\,\fm(\diff b)<\infty$ and $\varphi(\omega):=0$ otherwise. By the assertion (i), together with the assumption that $\bE[\int_B|\bar{\varphi}(b)|^2\,\fm(\diff b)]<\infty$, we see that $\varphi:\Omega\to L^2(B;E)$ is strongly $\cF_T$-measurable and $\bE[\|\varphi\|_{L^2(B;E)}^2]<\infty$. Hence, by the martingale representation theorem (applying to the $L^2(B;E)$-valued random variable $\varphi$), there exists a unique progressively measurable process $\zeta:\Omega\times[0,T]\to L^2(B;E)$ such that $\bE[\int^T_0\|\zeta_t\|_{L^2(B;E)}^2\,\diff t]<\infty$ and $\varphi=\int^T_0\zeta_t\,\diff W_t$ in $L^2(B;E)$ a.s. By the assertion (ii), there exists a (jointly) progressively measurable map $\bar{\zeta}:\Omega\times[0,T]\times B\to E$ such that $[\bar{\zeta}_t(\omega,\cdot)]=\zeta_t(\omega)$. Clearly, the integrability condition \eqref{app_eq_zeta-integrability} holds. Furthermore, for any $\eta\in L^2(B;E)$ with a measurable representative $\bar{\eta}:B\to E$, we have
\begin{align*}
	\int_B\langle\bar{\eta}(b),\bar{\varphi}(b)\rangle\,\fm(\diff b)&=\langle\eta,\varphi\rangle_{L^2(B;E)}=\left\langle\eta,\int^T_0\zeta_t\,\diff W_t\right\rangle_{L^2(B;E)}=\int^T_0\langle\eta,\zeta_t\rangle_{L^2(B;E)}\,\diff W_t\\
	&=\int^T_0\int_B\langle\bar{\eta}(b),\bar{\zeta}_t(b)\rangle\,\fm(\diff b)\,\diff W_t=\int_B\left\langle\bar{\eta}(b),\int^T_0\bar{\zeta}_t(b)\,\diff W_t\right\rangle\,\fm(\diff b)\ \ \text{a.s.},
\end{align*}
where we used the stochastic Fubini theorem (cf.\ \cite[Theorem 2.2]{Ve11}) in the last equality. Since $L^2(B;E)$ is separable, there exists a $\bP$-null set $\cN\in\cF$ (independent of the choice of $\eta$) such that the equality $\int_B\langle\bar{\eta}(b),\bar{\varphi}(b,\omega)\rangle\,\fm(\diff b)=\int_B\langle\bar{\eta}(b),(\int^T_0\bar{\zeta}_t(b)\,\diff W_t)(\omega)\rangle\,\fm(\diff b)$ holds for any $\eta=[\bar{\eta}]\in L^2(B;E)$ and any $\omega\in\Omega\setminus\cN$. Hence, the equality \eqref{app_eq_zeta-representation} holds. This completes the proof.
\end{itemize}
\end{proof}

%%%%%%%%%%%%%%
%% Section
%%%%%%%%%%%%%%

\section{Relationships between BSEEs and BSVIEs}\label{app-BSVIE}

In the study of the global maximum principle for SVEs, Wang and Yong \cite{WaYo23} considered (non-convolution type) SVEs with regular kernels. They introduced first and second order adjoint equations in form of the so-called \emph{backward stochastic Volterra integral equations} (BSVIEs). In this section, we investigate a relationship between our adjoint equations \eqref{MP_eq_adeq1} and \eqref{MP_eq_adeq2} in form of BSEEs and that of \cite{WaYo23} in form of BSVIEs.

Suppose that \cref{control_assum_coefficient} and \cref{MP_assum_kernel} hold. Fix a control process $\hat{u}\in\cU$. In this setting, the first order adjoint equation in \cite{WaYo23} can be written as the following (Type-II) BSVIE\footnote{Here, we use a slightly different but equivalent formulation for the Type-II BSVIE from \cite{WaYo23}.}:
\begin{equation}\label{app_eq_BSVIE1}
	\begin{dcases}
	p^\V_1(t)=-\hat{h}_x^\top-\int^T_tq^\V_1(s)\,\diff W_s,\ \ t\in[0,T],\\
	p^\V_2(t)=-\hat{f}_x(t)^\top+\hat{b}_x(t)^\top K_b(T-t)^\top p^\V_1(t)+\hat{\sigma}_x(t)^\top K_\sigma(T-t)^\top q^\V_1(t)\\
	\hspace{2cm}+\int^T_t\Big\{\hat{b}_x(t)^\top K_b(s-t)^\top \bE_t[p^\V_2(s)]+\hat{\sigma}_x(t)^\top K_\sigma(s-t)^\top q^\V_2(s,t)\Big\}\,\diff s.\ \ t\in[0,T].
	\end{dcases}
\end{equation}
Here and below, the super script ``$\V$'' represents ``Volterra''. A tuple
\begin{equation*}
	(p^\V_1,q^\V_1,p^\V_2,q^\V_2)=\Big((p^\V_1(t))_{0\leq t\leq T},(q^\V_1(t))_{0\leq t\leq T},(p^\V_2(t))_{0\leq t\leq T},(q^\V_2(s,t))_{0\leq t\leq s\leq T}\Big)
\end{equation*}
is said to be an \emph{adapted M-solution} to the BSVIE \eqref{app_eq_BSVIE1} if the maps
\begin{align*}
	&\Omega\times[0,T]\ni(\omega,t)\mapsto p^\V_1(t),q^\V_1(t),p^\V_2(t)\in\bR^n\ \ \text{and}\\
	&\Omega\times[0,T]\times[0,T]\ni(\omega,t,s)\mapsto q^\V_2(s,t)\1_{[0,s]}(t)\in\bR^n
\end{align*}
are progressively measurable, square-integrable and satisfy \eqref{app_eq_BSVIE1}, together with the constraint
\begin{equation}\label{app_eq_BSVIE1-M}
	p^\V_2(s)=\bE\big[p^\V_2(s)\big]+\int^s_0q^\V_2(s,t)\,\diff W_t,\ \ s\in[0,T].
\end{equation}
Notice that $(p^\V_1,q^\V_1)$ is determined by the first line in \eqref{app_eq_BSVIE1}; $p^\V_1(t)=-\bE_t[\hat{h}_x]^\top$, and $q^\V_1$ is the martingale integrand of the $\bR^n$-valued martingale $p^\V_1$ defined via the martingale representation theorem:
\begin{equation}\label{app_eq_BSVIE1-trivial}
	p^\V_1(s)=p^\V_1(0)+\int^s_0q^\V_1(t)\,\diff W_t,\ \ s\in[0,T].
\end{equation}
Under suitable integrability conditions on the kernels $K_b$ and $K_\sigma$, the BSVIE \eqref{app_eq_BSVIE1} admits a unique adapted M-solution; see \cite{Yo08}.

The second order adjoint equation introduced in \cite{WaYo23} is written as the following system of BSVIEs:
\begin{equation}\label{app_eq_BSVIE2}
	\begin{dcases}
	P^\V_1(t)=-\hat{h}_{xx}-\int^T_tQ^\V_1(s)\,\diff W_s,\ \ t\in[0,T],\\
	P^\V_2(t)=\hat{b}_x(t)^\top K_b(T-t)^\top P^\V_1(t)+\hat{\sigma}_x(t)^\top K_\sigma(T-t)^\top Q^\V_1(t)\\
	\hspace{2cm}+\int^T_t\Big\{\hat{b}_x(t)^\top K_b(s-t)^\top\bE_t\big[P^\V_2(s)\big]+\hat{\sigma}_x(t)^\top K_\sigma(s-t)^\top Q^\V_2(s,t)\Big\}\,\diff s,\ \ t\in[0,T],\\
	P^\V_3(t)=-\hat{f}_{xx}(t)+\big\langle K_b(T-t)^\top p^\V_1(t),\hat{b}_{xx}(t)\big\rangle+\big\langle K_\sigma(T-t)^\top q^\V_1(t),\hat{\sigma}_{xx}(t)\big\rangle\\
	\hspace{0.5cm}+\int^T_t\Big\{\big\langle K_b(s-t)^\top\bE_t\big[p^\V_2(s)\big],\hat{b}_{xx}(t)\big\rangle+\big\langle K_\sigma(s-t)^\top q^\V_2(s,t),\hat{\sigma}_{xx}(t)\big\rangle\Big\}\,\diff s\\
	\hspace{0.5cm}+\hat{\sigma}_x(t)^\top K_\sigma(T-t)^\top P^\V_1(t)K_\sigma(T-t)\hat{\sigma}_x(t)\\
	\hspace{0.5cm}+\int^T_t\hat{\sigma}_x(t)^\top\Big\{K_\sigma(T-t)^\top\bE_t\big[P^\V_2(s)\big]^\top K_\sigma(s-t)+K_\sigma(s-t)^\top\bE_t\big[P^\V_2(s)\big]K_\sigma(T-t)\Big\}\hat{\sigma}_x(t)\,\diff s\\
	\hspace{0.5cm}+\int^T_t\hat{\sigma}_x(t)^\top K_\sigma(s-t)^\top\bE_t\big[P^\V_3(s)\big]K_\sigma(s-t)\hat{\sigma}_x(t)\,\diff s\\
	\hspace{0.5cm}+\int^T_t\int^T_s\hat{\sigma}_x(t)^\top\Big\{K_\sigma(r-t)^\top\bE_t\big[P^\V_4(r,s)\big]^\top K_\sigma(s-t)+K_\sigma(s-t)^\top\bE_t\big[P^\V_4(r,s)\big]K_\sigma(r-t)\Big\}\hat{\sigma}_x(t)\,\diff r\,\diff s,\\
	\hspace{7cm}t\in[0,T],\\
	P^\V_4(r,t)=\hat{b}_x(t)^\top K_b(T-t)^\top\bE_t\big[P^\V_2(r)\big]^\top+\hat{\sigma}_x(t)^\top K_\sigma(T-t)^\top Q^\V_2(r,t)^\top\\
	\hspace{2cm}+\hat{b}_x(t)^\top K_b(r-t)^\top\bE_t\big[P^\V_3(r)\big]+\hat{\sigma}_x(t)^\top K_\sigma(r-t)^\top Q^\V_3(r,t)\\
	\hspace{2cm}+\int^T_r\Big\{\hat{b}_x(t)^\top K_b(s-t)^\top\bE_t\big[P^\V_4(s,r)\big]^\top+\hat{\sigma}_x(t)^\top K_\sigma(s-t)^\top Q^\V_4(s,r,t)^\top\Big\}\,\diff s\\
	\hspace{2cm}+\int^r_t\Big\{\hat{b}_x(t)^\top K_b(s-t)^\top \bE_t\big[P^\V_4(r,s)\big]+\hat{\sigma}_x(t)^\top K_\sigma(s-t)^\top Q^\V_4(r,s,t)\Big\}\,\diff s,\ \ 0\leq t\leq r\leq T.
	\end{dcases}
\end{equation}
A tuple
\begin{align*}
	&(P^\V_1,Q^\V_1,P^\V_2,Q^\V_2,P^\V_3,Q^\V_3,P^\V_4,Q^\V_4)\\
	&=\Big((P^\V_1(t))_{0\leq t\leq T},(Q^\V_1(t))_{0\leq t\leq T},(P^\V_2(t))_{0\leq t\leq T},(Q^\V_2(s,t))_{0\leq t\leq s\leq T},\\
	&\hspace{2cm}(P^\V_3(t))_{0\leq t\leq T},(Q^\V_3(s,t))_{0\leq t\leq s\leq T},(P^\V_4(r,t))_{0\leq t\leq r\leq T},(Q^\V_4(r,s,t))_{0\leq t\leq s\leq r\leq T}\Big)
\end{align*}
is said to be an adapted M-solution to the BSVIE system \eqref{app_eq_BSVIE2} if the maps
\begin{align*}
	&\Omega\times[0,T]\ni(\omega,t)\mapsto P^\V_1(t),Q^\V_1(t),P^\V_2(t),P^\V_3(t)\in\bR^{n\times n},\\
	&\Omega\times[0,T]\times[0,T]\ni(\omega,t,s)\mapsto Q^\V_2(s,t)\1_{[0,s]}(t),Q^\V_3(s,t)\1_{[0,s]}(t)\in\bR^{n\times n},\\
	&\Omega\times[0,T]\times[0,T]\ni(\omega,t,r)\mapsto P^\V_4(r,t)_{[0,r]}(t)\in\bR^{n\times n},\ \ \text{and}\\
	&\Omega\times[0,T]\times[0,T]\times[0,T]\ni(\omega,t,s,r)\mapsto Q^\V_4(r,s,t)\1_{[0,r]}(s)\1_{[0,s]}(t)\in\bR^{n\times n}
\end{align*}
are progressively measurable, square-integrable and satisfy \eqref{app_eq_BSVIE2}, together with the constraints
\begin{equation}\label{app_eq_BSVIE2-M}
	\begin{dcases}
	P^\V_2(s)=\bE\big[P^\V_2(s)\big]+\int^s_0Q^\V_2(s,t)\,\diff W_t,\ \ 0\leq s\leq T,\\
	P^\V_3(s)=\bE\big[P^\V_3(s)\big]+\int^s_0Q^\V_3(s,t)\,\diff W_t,\ \ 0\leq s\leq T,\\
	P^\V_4(r,s)=\bE\big[P^\V_4(r,s)\big]+\int^s_0Q^\V_4(r,s,t)\,\diff W_t,\ \ 0\leq s\leq r\leq T.
	\end{dcases}
\end{equation}
Notice that the pair $(P^\V_1,Q^\V_1)$ is determined by the first line of \eqref{app_eq_BSVIE2}; $P^\V_1(t)=-\bE_t[\hat{h}_{xx}]$, and $Q^\V_1$ is the martingale integrand of the $\bR^{n\times n}$-valued martingale $P^\V_1$ defined via the martingale representation theorem:
\begin{equation}\label{app_eq_BSVIE2-trivial}
	P^\V_1(s)=P^\V_1(0)+\int^s_0Q^\V_1(t)\,\diff W_t,\ \ s\in[0,T].
\end{equation}
By \cite[Theorem 5.1]{WaYo23}, when $K_b$ and $K_\sigma$ are bounded, the BSVIE system \eqref{app_eq_BSVIE2} has a unique adapted M-solution.

The following proposition shows relationships between BSEEs and BSVIEs.

%% Proposition

\begin{prop}\label{app_prop_BSVIE}
Suppose that \cref{control_assum_coefficient} and \cref{MP_assum_kernel} with $\alpha=0$ hold. Fix a control process $\hat{u}\in\cU$.
\begin{itemize}
\item[(1-A)]
If $(p^\V_1,q^\V_1,p^\V_2,q^\V_2)$ is the adapted M-solution of the BSVIE \eqref{app_eq_BSVIE1}, then the pair $(\hat{p},\hat{q})$ of the $\mu$-equivalence classes of progressively measurable maps $\hat{p},\hat{q}:\Omega\times[0,T]\times\bR_+\to\bR^n$ defined by
\begin{equation}\label{app_eq_pq}
\begin{split}
	&\hat{p}_t(\theta):=e^{-\theta(T-t)}p^\V_1(t)+\int^T_te^{-\theta(s-t)}\bE_t\big[p^\V_2(s)\big]\,\diff s,\\
	&\hat{q}_t(\theta):=e^{-\theta(T-t)}q^\V_1(t)+\int^T_te^{-\theta(s-t)}q^\V_2(s,t)\,\diff s,\ \ (\omega,t,\theta)\in\Omega\times[0,T]\times\bR_+,
\end{split}
\end{equation}
is the solution of the BSEE \eqref{MP_eq_adeq1}.
\item[(1-B)]
Conversely, if $(\hat{p},\hat{q})$ is the solution of the BSEE \eqref{MP_eq_adeq1}, then the tuple $(p^\V_1,q^\V_1,p^\V_2,q^\V_2)$ of progressively measurable maps defined by
\begin{equation}\label{app_eq_p^V}
\begin{split}
	&p^\V_1(t):=-\bE_t\big[\hat{h}_x\big]^\top,\ \ t\in[0,T],\\
	&p^\V_2(t):=\hat{b}_x(t)^\top\mu[M_b^\top\hat{p}_t]+\hat{\sigma}_x(t)^\top\mu[M_\sigma^\top\hat{q}_t]-\hat{f}_x(t)^\top,\ \ t\in[0,T],
\end{split}
\end{equation}
together with the martingale representations \eqref{app_eq_BSVIE1-M} and \eqref{app_eq_BSVIE1-trivial}, is the adapted M-solution of the BSVIE \eqref{app_eq_BSVIE1}.
\item[(2-A)]
Assume that $(P^\V_1,Q^\V_1,P^\V_2,Q^\V_2,P^\V_3,Q^\V_3,P^\V_4,Q^\V_4)$ is the adapted M-solution of the BSVIE system \eqref{app_eq_BSVIE2}. Then, the following hold:
\begin{itemize}
\item[(i)]
The pair $(\hat{P},\hat{Q})$ of the $\mu^{\otimes2}$-equivalence classes of progressively measurable maps $\hat{P},\hat{Q}:\Omega\times[0,T]\times\bR_+^2\to\bR^{n\times n}$ defined by
\begin{equation}\label{app_eq_PQ}
\begin{split}
	&\hat{P}_t(\theta_1,\theta_2)\\
	&:=e^{-(\theta_1+\theta_2)(T-t)}P^\V_1(t)\\
	&\hspace{0.5cm}e^{-\theta_1(T-t)}\int^T_te^{-\theta_2(s-t)}\bE_t\big[P^\V_2(s)\big]^\top\,\diff s+e^{-\theta_2(T-t)}\int^T_te^{-\theta_1(s-t)}\bE_t\big[P^\V_2(s)\big]\,\diff s\\
	&\hspace{0.5cm}+\int^T_te^{-(\theta_1+\theta_2)(r-t)}\bE_t\big[P^\V_3(r)\big]\,\diff r\\
	&\hspace{0.5cm}+\int^T_t\left\{e^{-\theta_1(r-t)}\int^r_te^{-\theta_2(s-t)}\bE_t\big[P^\V_4(r,s)\big]^\top\,\diff s+e^{-\theta_2(r-t)}\int^r_te^{-\theta_1(s-t)}\bE_t\big[P^\V_4(r,s)\big]\,\diff s\right\}\,\diff r,\\
	&\hat{Q}_t(\theta_1,\theta_2)\\
	&:=e^{-(\theta_1+\theta_2)(T-t)}Q^\V_1(t)\\
	&\hspace{0.5cm}+e^{-\theta_1(T-t)}\int^T_te^{-\theta_2(s-t)}Q^\V_2(s,t)^\top\,\diff s+e^{-\theta_2(T-t)}\int^T_te^{-\theta_1(s-t)}Q^\V_2(s,t)\,\diff s\\
	&\hspace{0.5cm}+\int^T_te^{-(\theta_1+\theta_2)(r-t)}Q^\V_3(r,t)\,\diff r\\
	&\hspace{0.5cm}+\int^T_t\left\{e^{-\theta_1(r-t)}\int^r_te^{-\theta_2(s-t)}Q^\V_4(r,s,t)^\top\,\diff s+e^{-\theta_2(r-t)}\int^r_te^{-\theta_1(s-t)}Q^\V_4(r,s,t)\diff s\right\}\,\diff r,\\
	&\hspace{6cm}(\omega,t,\theta_1,\theta_2)\in\Omega\times[0,T]\times\bR_+^2,
\end{split}
\end{equation}
is the solution of the BSEE \eqref{MP_eq_adeq2}.
\item[(ii)]
The pair $(\cP,\cQ)$ of the $\mu$-equivalence classes of progressively measurable maps $\cP,\cQ:\Omega\times[0,T]\times\bR_+\to\bR^{n\times n}$ defined by
\begin{align*}
	&\cP_t(\theta):=e^{-\theta(T-t)}P^\V_1(t)+\int^T_te^{-\theta(s-t)}\bE_t\big[P^\V_2(s)\big]\,\diff s,\\
	&\cQ_t(\theta):=e^{-\theta(T-t)}Q^\V_1(t)+\int^T_te^{-\theta(s-t)}Q^\V_2(s,t)\,\diff s,\ \ (\omega,t,\theta)\in\Omega\times[0,T]\times\bR_+,
\end{align*}
is the solution of the following BSEE:
\begin{equation}\label{app_eq_BSEE-T}
	\begin{dcases}
	\diff\cP_t(\theta)=\theta\cP_t(\theta)\,\diff t-\big\{\hat{b}_x(t)^\top\mu[M_b^\top\cP_t]+\hat{\sigma}_x(t)^\top\mu[M_\sigma^\top\cQ_t]\big\}\,\diff t+\cQ_t(\theta)\,\diff W_t,\ \ \theta\in\bR_+,\ t\in[0,T],\\
	\cP_T(\theta)=-\hat{h}_{xx},\ \ \theta\in\bR_+.
	\end{dcases}
\end{equation}
\item[(iii)]
For each $r\in[0,T]$, the pair $(\sP^r,\sQ^r)$ of the $\mu$-equivalence classes of progressively measurable maps $\sP^r,\sQ^r:\Omega\times[0,r]\times\bR_+\to\bR^{n\times n}$ defined by
\begin{equation}\label{app_eq_PQ-r}
\begin{split}
	&\sP^r_t(\theta):=e^{-\theta(T-t)}\bE_t\big[P^\V_2(r)\big]^\top+e^{-\theta(r-t)}\bE_t\big[P^\V_3(r)\big]+\int^T_re^{-\theta(s-t)}\bE_t\big[P^\V_4(s,r)\big]^\top\diff s\\
	&\hspace{3cm}+\int^r_te^{-\theta(s-t)}\bE_t\big[P^\V_4(r,s)\big]\,\diff s,\\
	&\sQ^r_t(\theta):=e^{-\theta(T-t)}Q^\V_2(r,t)^\top+e^{-\theta(r-t)}Q^\V_3(r,t)+\int^T_re^{-\theta(s-t)}Q^\V_4(s,r,t)^\top\diff s\\
	&\hspace{3cm}+\int^r_te^{-\theta(s-t)}Q^\V_4(r,s,t)\,\diff s,\ \ \ (\omega,t,\theta)\in\Omega\times[0,r]\times\bR_+,
\end{split}
\end{equation}
is the solution of the following BSEE on $[0,r]$:
\begin{equation}\label{app_eq_BSEE-r}
	\begin{dcases}
	\diff\sP^r_t(\theta)=\theta\sP^r_t(\theta)\,\diff t-\big\{\hat{b}_x(t)^\top\mu[M_b^\top\sP^r_t]+\hat{\sigma}_x(t)^\top\mu[M_\sigma^\top\sQ^r_t]\big\}\,\diff t+\sQ^r_t(\theta)\,\diff W_t,\\
	\hspace{7cm}\theta\in\bR_+,\ t\in[0,r],\\
	\sP^r_r(\theta)=\big\langle\mu[M_b^\top\hat{p}_r],\hat{b}_{xx}(r)\big\rangle+\big\langle\mu[M_\sigma^\top\hat{q}_r],\hat{\sigma}_{xx}(r)\big\rangle-\hat{f}_{xx}(r)+\hat{\sigma}_x(r)^\top\mu^{\otimes2}[M_\sigma^\top\hat{P}_rM_\sigma]\hat{\sigma}_x(r)\\
	\hspace{3cm}+\mu[\hat{P}_r(\theta,\cdot)M_b]\hat{b}_x(r)+\mu[\hat{Q}_r(\theta,\cdot)M_\sigma]\hat{\sigma}_x(r),\ \ \theta\in\bR_+.
	\end{dcases}
\end{equation}
\end{itemize}
\item[(2-B)]
Conversely, assume that $(\hat{P},\hat{Q})$ is the solution of the BSEE \eqref{MP_eq_adeq2}, that $(\cP,\cQ)$ is the solution of the BSEE \eqref{app_eq_BSEE-T}, and that $(\sP^r,\sQ^r)$ is the solution of the BSEE \eqref{app_eq_BSEE-r} for each $r\in[0,T]$ such that $(\omega,t,\theta,r)\mapsto\sP^r_t(\theta)\1_{[0,r]}(t),\sQ^r_t(\theta)\1_{[0,r]}(t)\in\bR^{n\times n}$ are progressively measurable. Then, the tuple $(P^\V_1,Q^\V_1,P^\V_2,Q^\V_2,P^\V_3,Q^\V_3,P^\V_4,Q^\V_4)$ of progressively measurable maps defined by
\begin{equation}\label{app_eq_P^V}
\begin{split}
	&P^\V_1(t):=-\bE_t\big[\hat{h}_{xx}\big],\ \ t\in[0,T],\\
	&P^\V_2(t):=\hat{b}_x(t)^\top\mu[M_b^\top\cP_t]+\hat{\sigma}_x(t)^\top\mu[M_\sigma^\top\cQ_t],\ \ t\in[0,T],\\
	&P^\V_3(t):=\big\langle\mu[M_b^\top\hat{p}_t],\hat{b}_{xx}(t)\big\rangle+\big\langle\mu[M_\sigma^\top\hat{q}_t],\hat{\sigma}_{xx}(t)\big\rangle-\hat{f}_{xx}(t)+\hat{\sigma}_x(t)^\top\mu^{\otimes2}[M_\sigma^\top\hat{P}_tM_\sigma]\hat{\sigma}_x(t),\ \ t\in[0,T],\\
	&P^\V_4(r,t):=\hat{b}_x(t)^\top\mu[M_b^\top\sP^r_t]+\hat{\sigma}_x(t)^\top\mu[M_\sigma^\top\sQ^r_t],\ \ 0\leq t\leq r\leq T,
\end{split}
\end{equation}
together with the martingale representations \eqref{app_eq_BSVIE2-M} and \eqref{app_eq_BSVIE2-trivial}, is the adapted M-solution of the BSVIE system \eqref{app_eq_BSVIE2}.
\end{itemize}
\end{prop}

%% Remark

\begin{rem}
The assumption that \eqref{MP_eq_M-integrable} holds with $\alpha=0$ is required in order to ensure the square-integrability condition (with respect to the time parameter) for the adapted M-solution of the BSVIE system \eqref{app_eq_BSVIE2}. Actually, \cref{app_prop_BSVIE} holds true even in the general (singular) case with $\alpha\in(0,1)$ if we replaced the square-integrability requirement in the definition of the adapted M-solution of the BSVIEs to a suitable ``weighted square-integrability condition'' in a similar manner as in the estimate \eqref{BSEE_eq_BSEE-estimate}. However, since the main finding in \cref{app_prop_BSVIE} is the structural relationships between BSEEs and BSVIEs, we do not state such generalization here for the sake of simplicity of presentation.
\end{rem}

%% Proof

\begin{proof}[Proof of \cref{app_prop_BSVIE}]
In this proof, we use the (stochastic) Fubini theorem several times. The required integrability conditions for the applications of the (stochastic) Fubini theorem can be easily checked noting the assumption that \eqref{MP_eq_M-integrable} holds with $\alpha=0$, and we omit the proofs of them. Instead, we focus on the structural relationships.

\underline{(1-A): From $(p^\V_1,q^\V_1,p^\V_2,q^\V_2)$ to $(\hat{p},\hat{q})$.} Let $(p^\V_1,q^\V_1,p^\V_2,q^\V_2)$ be the adapted M-solution of the BSVIE \eqref{app_eq_BSVIE1}, and define $(\hat{p},\hat{q})$ by \eqref{app_eq_pq}. Noting $p^\V_1(t)=\bE_t[p^\V_1(T)]$ and the martingale representations \eqref{app_eq_BSVIE1-M} and \eqref{app_eq_BSVIE1-trivial}, by the same argument as in the proof of \cref{BSEE_lemm_trivialBSEE}, we see that $(\hat{p},\hat{q})$ is the solution to the following BSEE:
\begin{equation}\label{app_eq_BSEE1^V}
	\begin{dcases}
	\diff\hat{p}_t(\theta)=\theta\hat{p}_t(\theta)\,\diff t-p^\V_2(t)\,\diff t+\hat{q}_t(\theta)\,\diff W_t,\ \ \theta\in\bR_+,\ t\in[0,T],\\
	\hat{p}_T(\theta)=p^\V_1(T),\ \ \theta\in\bR_+.
	\end{dcases}
\end{equation}
Hence, noting that $p^\V_1(T)=-\hat{h}_x^\top$, in order to show that $(\hat{p},\hat{q})$ is the solution to the BSEE \eqref{MP_eq_adeq1}, it suffices to show that
\begin{equation}\label{app_eq_p^V_2}
	p^\V_2(t)=\hat{b}_x(t)^\top\mu[M_b^\top\hat{p}_t]+\hat{\sigma}_x(t)^\top\mu[M_\sigma^\top\hat{q}_t]-\hat{f}_x(t)^\top.
\end{equation}
However, thanks to the representation formulas \eqref{MP_eq_kernel} for the kernels $K_b$ and $K_\sigma$, together with the definitions \eqref{app_eq_pq} of $\hat{p}$ and $\hat{q}$, using Fubini's theorem, we have
\begin{equation}\label{app_eq_pq-mu}
\begin{split}
	&\mu[M_b^\top\hat{p}_t]=K_b(T-t)^\top p^\V_1(t)+\int^T_tK_b(s-t)^\top\bE_t\big[p^\V_2(s)\big]\,\diff s,\\
	&\mu[M_\sigma^\top\hat{q}_t]=K_\sigma(T-t)^\top q^\V_1(t)+\int^T_tK_\sigma(s-t)^\top q^\V_2(s,t)\,\diff s.
\end{split}
\end{equation}
From these expressions, together with the second equation in the BSVIE \eqref{app_eq_BSVIE1}, we see that \eqref{app_eq_p^V_2} holds. Hence, $(\hat{p},\hat{q})$ is the solution of the BSEE \eqref{MP_eq_adeq1}.

\underline{(1-B): From $(\hat{p},\hat{q})$ to $(p^\V_1,q^\V_1,p^\V_2,q^\V_2)$.} Let $(\hat{p},\hat{q})$ be the solution of the BSEE \eqref{MP_eq_adeq1}. Define $(p^\V_1,q^\V_1,p^\V_2,q^\V_2)$ by \eqref{app_eq_p^V}, together with the martingale representations \eqref{app_eq_BSVIE1-M} and \eqref{app_eq_BSVIE1-trivial}. By the expression \eqref{app_eq_p^V}, the BSEE \eqref{MP_eq_adeq1} for $(\hat{p},\hat{q})$ can be written as \eqref{app_eq_BSEE1^V}. Using It\^{o}'s formula, for each $\theta\in\bR_+$,
\begin{align*}
	&\hat{p}_t(\theta)+\int^T_te^{-\theta(s-t)}\hat{q}_s(\theta)\,\diff W_s\\
	&=e^{-\theta(T-t)}p^\V_1(T)+\int^T_te^{-\theta(s-t)}p^\V_2(s)\,\diff s\\
	&=e^{-\theta(T-t)}\left\{p^\V_1(t)+\int^T_tq^\V_1(s)\,\diff W_s\right\}+\int^T_te^{-\theta(s-t)}\left\{\bE_t\big[p^\V_2(s)\big]+\int^s_tq^\V_2(s,r)\,\diff W_r\right\}\,\diff s\\
	&=e^{-\theta(T-t)}p^\V_1(t)+\int^T_te^{-\theta(s-t)}\bE_t\big[p^\V_2(s)\big]\,\diff s\\
	&\hspace{1.5cm}+\int^T_te^{-\theta(s-t)}\left\{e^{-\theta(T-s)}q^\V_1(s)+\int^T_se^{-\theta(r-s)}q^\V_2(r,s)\,\diff r\right\}\,\diff W_s,
\end{align*}
where we used the martingale representations \eqref{app_eq_BSVIE1-M} and \eqref{app_eq_BSVIE1-trivial} in the third equality and the stochastic Fubini theorem in the fourth equality. From this, we get the equalities in form of \eqref{app_eq_pq}. Hence, as before, thanks to the representation formulas \eqref{MP_eq_kernel} for the kernels $K_b$ and $K_\sigma$, using Fubini's theorem, we get \eqref{app_eq_pq-mu}. Inserting the expressions in \eqref{app_eq_pq-mu} to the definition \eqref{app_eq_p^V} of $p_2^\V$, we see that $(p^\V_1,q^\V_1,p^\V_2,q^\V_2)$ is the adapted M-solution of the BSVIE \eqref{app_eq_BSVIE1}.

\underline{(2-A): From $(P^\V_1,Q^\V_1,P^\V_2,Q^\V_2,P^\V_3,Q^\V_3,P^\V_4,Q^\V_4)$ to $(\hat{P},\hat{Q})$, $(\cP,\cQ)$ and $(\sP^r,\sQ^r)$.} Suppose that we are given the adapted M-solution $(P^\V_1,Q^\V_1,P^\V_2,Q^\V_2,P^\V_3,Q^\V_3,P^\V_4,Q^\V_4)$ of the BSVIE system \eqref{app_eq_BSVIE2}.

First, we show the assertion (i) in (2-A). Define $(\hat{P},\hat{Q})$ by \eqref{app_eq_PQ}. Notice that
\begin{align*}
	&\hat{P}_t(\theta_1,\theta_2)=e^{-(\theta_1+\theta_2)(T-t)}\bE_t\big[P^\V_1(T)\big]+\int^T_te^{-(\theta_1+\theta_2)(s-t)}\bE_t\big[G^\V_s(\theta_1,\theta_2)\big]\,\diff s,\\
	&\hat{Q}_t(\theta_1,\theta_2)=e^{-(\theta_1+\theta_2)(T-t)}Q^\V_1(t)+\int^T_te^{-(\theta_1+\theta_2)(s-t)}Z^\V_t(s,\theta_1,\theta_2)\,\diff s,
\end{align*}
where
\begin{align*}
	G^\V_s(\theta_1,\theta_2)&=e^{-\theta_1(T-s)}P^\V_2(s)^\top+e^{-\theta_2(T-s)}P^\V_2(s)+P^\V_3(s)\\
	&\hspace{2cm}+\int^T_se^{-\theta_1(r-s)}P^\V_4(r,s)^\top\,\diff r+\int^T_se^{-\theta_2(r-s)}P^\V_4(r,s)\,\diff r
\end{align*}
and
\begin{align*}
	Z^\V_t(s,\theta_1,\theta_2)&=e^{-\theta_1(T-s)}Q^\V_2(s,t)^\top+e^{-\theta_2(T-s)}Q^\V_2(s,t)+Q^\V_3(s,t)\\
	&\hspace{2cm}+\int^T_se^{-\theta_1(r-s)}Q^\V_4(r,s,t)^\top\,\diff r+\int^T_se^{-\theta_2(r-s)}Q^\V_4(r,s,t)\,\diff r.
\end{align*}
By the martingale representations \eqref{app_eq_BSVIE2-M} and \eqref{app_eq_BSVIE2-trivial}, together with the stochastic Fubini theorem, we see that
\begin{equation*}
	P^\V_1(T)=\bE\big[P^\V_1(T)\big]+\int^T_0Q^\V_1(t)\,\diff W_t
\end{equation*}
and
\begin{equation*}
	G^\V_s(\theta_1,\theta_2)=\bE\big[G^\V_s(\theta_1,\theta_2)\big]+\int^s_0Z^\V_t(s,\theta_1,\theta_2)\,\diff W_t.
\end{equation*}
Hence, by the same argument as in the proof of \cref{BSEE_lemm_trivialBSEE}, we see that $(\hat{P},\hat{Q})$ is the solution to the following BSEE:
\begin{equation*}
	\begin{dcases}
	\diff\hat{P}_t(\theta_1,\theta_2)=(\theta_1+\theta_2)\hat{P}_t(\theta_1,\theta_2)\,\diff t-G^\V_t(\theta_1,\theta_2)\,\diff t+\hat{Q}_t(\theta_1,\theta_2)\,\diff W_t,\ \ (\theta_1,\theta_2)\in\bR_+^2,\ t\in[0,T],\\
	\hat{P}_T(\theta_1,\theta_2)=P^\V_1(T),\ \ (\theta_1,\theta_2)\in\bR_+^2.
	\end{dcases}
\end{equation*}
Therefore, noting that $P^\V_1(T)=-\hat{h}_{xx}$, in order to show that $(\hat{P},\hat{Q})$ is the solution to the BSEE \eqref{MP_eq_adeq2}, it suffices to show that
\begin{equation}\label{app_eq_G^V_t}
\begin{split}
	G^\V_t(\theta_1,\theta_2)&=\hat{b}_x(t)^\top\mu[M_b^\top\hat{P}_t(\cdot,\theta_2)]+\mu[\hat{P}_t(\theta_1,\cdot)M_b]\hat{b}_x(t)+\hat{\sigma}_x(t)^\top\mu^{\otimes2}[M_\sigma^\top\hat{P}_t M_\sigma]\hat{\sigma}_x(t)\\
	&\hspace{0.5cm}+\hat{\sigma}_x(t)^\top\mu[ M_\sigma^\top\hat{Q}_t(\cdot,\theta_2)]+\mu[\hat{Q}_t(\theta_1,\cdot)M_\sigma]\hat{\sigma}_x(t)\\
	&\hspace{0.5cm}+\langle\mu[M_b^\top\hat{p}_t],\hat{b}_{xx}(t)\rangle+\langle\mu[M_\sigma^\top\hat{q}_t],\hat{\sigma}_{xx}(t)\rangle-\hat{f}_{xx}(t).
\end{split}
\end{equation}
Thanks to the representation formulas \eqref{MP_eq_kernel} for $K_b$ and $K_\sigma$, together with the definitions \eqref{app_eq_PQ} of $\hat{P}$ and $\hat{Q}$, using Fubini's theorem, we have
\begin{align*}
	&\mu[M_b^\top\hat{P}_t(\cdot,\theta_2)]=e^{-\theta_2(T-t)}\left\{K_b(T-t)^\top P^\V_1(t)+\int^T_tK_b(s-t)^\top\bE_t\big[P^\V_2(s)\big]\,\diff s\right\}\\
	&\hspace{2cm}+\int^T_te^{-\theta_2(r-t)}\left\{K_b(T-t)^\top\bE_t\big[P^\V_2(r)\big]^\top+K_b(r-t)^\top\bE_t\big[P^\V_3(r)\big]\vphantom{\int^T_s}\right.\\
	&\left.\hspace{5cm}+\int^r_tK_b(s-t)^\top\bE_t\big[P^\V_4(r,s)\big]\,\diff s+\int^T_rK_b(s-t)^\top\bE_t\big[P^\V_4(s,r)\big]^\top\diff s\right\}\,\diff r
\end{align*}
and
\begin{align*}
	&\mu[M_\sigma^\top\hat{Q}_t(\cdot,\theta_2)]=e^{-\theta_2(T-t)}\left\{K_\sigma(T-t)^\top Q^\V_1(t)+\int^T_tK_\sigma(s-t)^\top Q^\V_2(s,t)\,\diff s\right\}\\
	&\hspace{2cm}+\int^T_te^{-\theta_2(r-t)}\left\{K_\sigma(T-t)^\top Q^\V_2(r,t)^\top+K_\sigma(r-t)^\top Q^\V_3(r,t)\vphantom{\int^T_s}\right.\\
	&\left.\hspace{5cm}+\int^r_tK_\sigma(s-t)^\top Q^\V_4(r,s,t)\,\diff s+\int^T_rK_\sigma(s-t)^\top Q^\V_4(s,r,t)^\top\diff s\right\}\,\diff r.
\end{align*}
From the above two expressions, together with the second and fourth equations in the BSVIE system \eqref{app_eq_BSVIE2}, we see that
\begin{equation}\label{app_eq_PQ-mu1}
	\hat{b}_x(t)^\top\mu[M_b^\top\hat{P}_t(\cdot,\theta_2)]+\hat{\sigma}_x(t)^\top\mu[M_\sigma^\top\hat{Q}_t(\cdot,\theta_2)]=e^{-\theta_2(T-t)}P^\V_2(t)+\int^T_te^{-\theta_2(r-t)}P^\V_4(r,t)\,\diff r.
\end{equation}
By similar calculations, we can show that
\begin{equation}\label{app_eq_PQ-mu2}
	\mu[\hat{P}_t(\theta_1,\cdot)M_b]\hat{b}_x(t)+\mu[\hat{Q}_t(\theta_1,\cdot)M_\sigma]\hat{\sigma}_x(t)=e^{-\theta_1(T-t)}P^\V_2(t)^\top+\int^T_te^{-\theta_1(r-t)}P^\V_4(r,t)^\top\diff r.
\end{equation}
Furthermore, again by the representation formula \eqref{MP_eq_kernel} for $K_\sigma$ and the definition \eqref{app_eq_PQ} of $\hat{P}$, using Fubini's theorem, we have
\begin{equation}\label{app_eq_P-mumu}
\begin{split}
	&\mu^{\otimes 2}[M_\sigma^\top\hat{P}_tM_\sigma]\\
	&=K_\sigma(T-t)^\top P^\V_1(t)K_\sigma(T-t)\\
	&\hspace{1cm}+\int^T_t\Big\{K_\sigma(T-t)^\top\bE_t\big[P^\V_2(s)\big]^\top K_\sigma(s-t)+K_\sigma(s-t)^\top\bE_t\big[P^\V_2(s)\big]K_\sigma(T-t)\Big\}\,\diff s\\
	&\hspace{1cm}+\int^T_tK_\sigma(s-t)^\top\bE_t\big[P^\V_3(s)\big]K_\sigma(s-t)\,\diff s\\
	&\hspace{1cm}+\int^T_t\int^T_s\Big\{K_\sigma(r-t)^\top\bE_t\big[P^\V_4(r,s)\big]^\top K_\sigma(s-t)+K_\sigma(s-t)^\top\bE_t\big[P^\V_4(r,s)\big]K_\sigma(r-t)\Big\}\,\diff r\,\diff s.
\end{split}
\end{equation}
From \eqref{app_eq_P-mumu} and the relation \eqref{app_eq_pq-mu} between $(\hat{p},\hat{q})$ and $(p^\V_1,q^\V_1,p^\V_2,q^\V_2)$, together with the third equation in the BSVIE system \eqref{app_eq_BSVIE2}, we see that
\begin{equation}\label{app_eq_PQ-mu3}
	\langle\mu[M_b^\top\hat{p}_t],\hat{b}_{xx}(t)\rangle+\langle\mu[M_\sigma^\top\hat{q}_t],\hat{\sigma}_{xx}(t)\rangle-\hat{f}_{xx}(t)+\hat{\sigma}_x(t)^\top\mu^{\otimes2}[M_\sigma^\top\hat{P}_tM_\sigma]\hat{\sigma}_x(t)=P^\V_3(t).
\end{equation}
By \eqref{app_eq_PQ-mu1}, \eqref{app_eq_PQ-mu2} and \eqref{app_eq_PQ-mu3}, we obtain \eqref{app_eq_G^V_t}. Hence, $(\hat{P},\hat{Q})$ is the solution of the BSEE \eqref{MP_eq_adeq2}.

The assertion (ii) in (2-A) can be proved by the same argument as in (1-A) above, and we omit it.

Concerning with the assertion (iii) in (2-A), notice that $(\sP^r,\sQ^r)$ defined by \eqref{app_eq_PQ-r} can be rewritten as
\begin{align*}
	&\sP^r_t(\theta)=e^{-\theta(r-t)}\bE_t\big[\varphi^\V(r,\theta)\big]+\int^r_te^{-\theta(s-t)}\bE_t\big[P^\V_4(r,s)\big]\,\diff s,\\
	&\sQ^r_t(\theta)=e^{-\theta(r-t)}\eta^\V_t(r,\theta)+\int^r_te^{-\theta(s-t)}Q^\V_4(r,s,t)\,\diff s,
\end{align*}
where
\begin{align*}
	&\varphi^\V(r,\theta):=e^{-\theta(T-r)}P^\V_2(r)^\top+P^\V_3(r)+\int^T_re^{-\theta(s-r)}P^\V_4(s,r)^\top\diff s,\\
	&\eta^\V_t(r,\theta):=e^{-\theta(T-r)}Q^\V_2(r,t)^\top+Q^\V_3(r,t)+\int^T_re^{-\theta(s-r)}Q^\V_4(s,r,t)^\top\diff s.
\end{align*}
By the martingale representations \eqref{app_eq_BSVIE2-M} and \eqref{app_eq_BSVIE2-trivial}, together with the stochastic Fubini theorem, we see that
\begin{equation*}
	\varphi^\V(r,\theta)=\bE\big[\varphi^\V(r,\theta)\big]+\int^r_0\eta^\V_t(r,\theta)\,\diff W_t,\ \ P^\V_4(r,s)=\bE\big[P^\V_4(r,s)\big]+\int^s_0Q^\V_4(r,s,t)\,\diff W_t.
\end{equation*}
Hence, by the same argument as in the proof of \cref{BSEE_lemm_trivialBSEE}, we see that $(\sP^r,\sQ^r)$ is the solution to the following BSEE on $[0,r]$:
\begin{equation*}
	\begin{dcases}
	\diff\sP^r_t(\theta)=\theta\sP^r_t(\theta)\,\diff t-P^\V_4(r,t)\,\diff t+\sQ^r_t(\theta)\,\diff W_t,\ \ \theta\in\bR_+,\ t\in[0,r],\\
	\sP^r_r(\theta)=\varphi^\V(r,\theta),\ \ \theta\in\bR_+.
	\end{dcases}
\end{equation*}
Furthermore, by \eqref{app_eq_PQ-mu2} and \eqref{app_eq_PQ-mu3}, we have
\begin{align*}
	\varphi^\V(r,\theta)&=\big\langle\mu[M_b^\top\hat{p}_r],\hat{b}_{xx}(r)\big\rangle+\big\langle\mu[M_\sigma^\top\hat{q}_r],\hat{\sigma}_{xx}(r)\big\rangle-\hat{f}_{xx}(r)+\hat{\sigma}_x(r)^\top\mu^{\otimes2}[M_\sigma^\top\hat{P}_rM_\sigma]\hat{\sigma}_x(r)\\
	&\hspace{0.5cm}+\mu[\hat{P}_r(\theta,\cdot)M_b]\hat{b}_x(r)+\mu[\hat{Q}_r(\theta,\cdot)M_\sigma]\hat{\sigma}_x(r).
\end{align*}

Hence, in order to show that $(\sP^r,\sQ^r)$ is the solution of the BSEE \eqref{app_eq_BSEE-r}, it suffices to show that
\begin{equation}\label{app_eq_P^V_4}
	P^\V_4(r,t)=\hat{b}_x(t)^\top\mu[M_b^\top\sP^r_t]+\hat{\sigma}_x(t)^\top\mu[M_\sigma^\top\sQ^r_t].
\end{equation}
However, thanks to the representation formulas \eqref{MP_eq_kernel} for the kernels $K_b$ and $K_\sigma$, together with the definitions \eqref{app_eq_PQ-r} of $\sP^r$ and $\sQ^r$, using Fubini's theorem, we have
\begin{equation}\label{app_eq_PQ^r-mu}
\begin{split}
	&\mu[M_b^\top\sP^r_t]=K_b(T-t)^\top\bE_t\big[P^\V_2(r)\big]^\top+K_b(r-t)^\top\bE_t\big[P^\V_3(r)\big]+\int^T_rK_b(s-t)^\top\bE_t\big[P^\V_4(s,r)\big]^\top\,\diff s\\
	&\hspace{3cm}+\int^r_tK_b(s-t)^\top\bE_t\big[P^\V_4(r,s)\big]\,\diff s,\\
	&\mu[M_\sigma^\top\sQ^r_t]=K_\sigma(T-t)Q^\V_2(r,t)^\top+K_\sigma(r-t)^\top Q^\V_3(r,t)+\int^T_rK_\sigma(s-t)^\top Q^\V_4(s,r,t)^\top\,\diff s\\
	&\hspace{3cm}+\int^r_tK_\sigma(s-t)^\top Q^\V_4(r,s,t)\,\diff s.
\end{split}
\end{equation}
From these expressions, together with the fourth equation in the BSVIE system \eqref{app_eq_BSVIE2}, we see that \eqref{app_eq_P^V_4} holds. Hence, $(\sP^r,\sQ^r)$ is the solution of the BSEE \eqref{app_eq_BSEE-r}.

\underline{(2-B): From $(\hat{P},\hat{Q})$, $(\cP,\cQ)$ and $(\sP^r,\sQ^r)$ to $(P^\V_1,Q^\V_1,P^\V_2,Q^\V_2,P^\V_3,Q^\V_3,P^\V_4,Q^\V_4)$.} Assume that $(\hat{P},\hat{Q})$ is the solution of the BSEE \eqref{MP_eq_adeq2}, that $(\cP,\cQ)$ is the solution of the BSEE \eqref{app_eq_BSEE-T}, and that $(\sP^r,\sQ^r)$ is the solution of the BSEE \eqref{app_eq_BSEE-r} for each $r\in[0,T]$ such that $(\omega,t,\theta,r)\mapsto\sP^r_t(\theta)\1_{[0,r]}(t),\sQ^r_t(\theta)\1_{[0,r]}(t)\in\bR^{n\times n}$ are progressively measurable. Define $(P^\V_1,Q^\V_1,P^\V_2,Q^\V_2,P^\V_3,Q^\V_3,P^\V_4,Q^\V_4)$ by \eqref{app_eq_P^V}, together with the martingale representations \eqref{app_eq_BSVIE2-M} and \eqref{app_eq_BSVIE2-trivial}. By the same argument as in (1-B) above, we see that $(P^\V_1,Q^\V_1,P^\V_2,Q^V_2)$ solves the first and second equations in the BSVIE system \eqref{app_eq_BSVIE2}. It remains to show that $(P^\V_3,Q^\V_3,P^\V_4,Q^\V_4)$ solves the third and fourth equations in \eqref{app_eq_BSVIE2}.

From the BSEE \eqref{app_eq_BSEE-T} for $(\cP,\cQ)$, using It\^{o}'s formula, we see that, for each $\theta\in\bR_+$,
\begin{align*}
	\cP_t(\theta)&=-e^{-\theta(T-t)}\hat{h}_{xx}+\int^T_te^{-\theta(s-t)}\Big\{\hat{b}_x(s)^\top\mu[M_b^\top\cP_s]+\hat{\sigma}_x(s)^\top\mu[M_\sigma^\top\cQ_s]\Big\}\,\diff s-\int^T_te^{-\theta(s-t)}\cQ_s(\theta)\,\diff W_s,\\
	&\hspace{7cm}t\in[0,T].
\end{align*}
Similarly, from the BSEE \eqref{app_eq_BSEE-r} for $(\sP^r,\sQ^r)$, we have
\begin{equation}\label{app_eq_P^r-Ito}
\begin{split}
	\sP^r_t(\theta)&=e^{-\theta(r-t)}\Big\{\big\langle\mu[M_b^\top\hat{p}_r],\hat{b}_{xx}(r)\big\rangle+\big\langle\mu[M_\sigma^\top\hat{q}_r],\hat{\sigma}_{xx}(r)\big\rangle-\hat{f}_{xx}(r)+\hat{\sigma}_x(r)^\top\mu^{\otimes2}[M_\sigma^\top\hat{P}_rM_\sigma]\hat{\sigma}_x(r)\\
	&\hspace{3cm}+\mu[\hat{P}_r(\theta,\cdot)M_b]\hat{b}_x(r)+\mu[\hat{Q}_r(\theta,\cdot)M_\sigma]\hat{\sigma}_x(r)\Big\}\\
	&\hspace{1cm}+\int^r_te^{-\theta(s-t)}\Big\{\hat{b}_x(s)^\top\mu[M_b^\top\sP^r_s]+\hat{\sigma}_x(s)^\top\mu[M_\sigma^\top\sQ^r_s]\Big\}\,\diff s-\int^r_te^{-\theta(s-t)}\sQ^r_s(\theta)\,\diff W_s,\\
	&\hspace{7cm}0\leq t\leq r\leq T.
\end{split}
\end{equation}
Define $\overline{P},\overline{Q}:\Omega\times[0,T]\times\bR_+^2\to\bR^{n\times n}$ by
\begin{align*}
	&\overline{P}_t(\theta_1,\theta_2):=e^{-\theta_2(T-t)}\cP_t(\theta_1)+\int^T_te^{-\theta_2(r-t)}\sP^r_t(\theta_1)\,\diff r,\\
	&\overline{Q}_t(\theta_1,\theta_2):=e^{-\theta_2(T-t)}\cQ_t(\theta_1)+\int^T_te^{-\theta_2(r-t)}\sQ^r_t(\theta_1)\,\diff r.
\end{align*}
Then, the stochastic Fubini theorem yields that, for each $(\theta_1,\theta_2)\in\bR_+^2$,
\begin{align*}
	&\overline{P}_t(\theta_1,\theta_2)\\
	&=-e^{-(\theta_1+\theta_2)(T-t)}\hat{h}_{xx}\\
	&\hspace{1cm}+\int^T_te^{-(\theta_1+\theta_2)(s-t)}\Big\{\big\langle\mu[M_b^\top\hat{p}_s],\hat{b}_{xx}(s)\big\rangle+\big\langle\mu[M_\sigma^\top\hat{q}_s],\hat{\sigma}_{xx}(s)\big\rangle-\hat{f}_{xx}(s)\\
	&\hspace{3cm}+\hat{\sigma}_x(s)^\top\mu^{\otimes2}[M_\sigma^\top\hat{P}_sM_\sigma]\hat{\sigma}_x(s)+\mu[\hat{P}_s(\theta_1,\cdot)M_b]\hat{b}_x(s)+\mu[\hat{Q}_s(\theta_1,\cdot)M_\sigma]\hat{\sigma}_x(s)\\
	&\hspace{3cm}+\hat{b}_x(s)^\top\mu[M_b^\top\overline{P}_s(\cdot,\theta_2)]+\hat{\sigma}_x(s)^\top\mu[M_\sigma^\top\overline{Q}_s(\cdot,\theta_2)]\Big\}\,\diff s\\
	&\hspace{1cm}-\int^T_te^{-(\theta_1+\theta_2)(s-t)}\overline{Q}_s(\theta_1,\theta_2)\,\diff W_s,\ \ t\in[0,T].
\end{align*}
From this, we see that $(\overline{P},\overline{Q})$ and $(\hat{P},\hat{Q})$ solve the same BSEE \eqref{MP_eq_adeq2}. Thanks to the uniqueness of the solution, we see that $\overline{P}_t(\theta_1,\theta_2)=\hat{P}_t(\theta_1,\theta_2)$ and $\overline{Q}_t(\theta_1,\theta_2)=\hat{Q}_t(\theta_1,\theta_2)$ for $\mu^{\otimes2}$-a.e.\ $(\theta_1,\theta_2)\in\bR_+^2$ a.s.\ for a.e.\ $t\in[0,T]$. In particular, we have
\begin{align*}
	&\mu[M_b^\top\hat{P}_r(\cdot,\theta)]=\mu[M_b^\top\overline{P}_r(\cdot,\theta)]=e^{-\theta(T-r)}\mu[M_b^\top\cP_r]+\int^T_re^{-\theta(s-r)}\mu[M_b^\top\sP^s_r]\diff s,\\
	&\mu[M_\sigma^\top\hat{Q}_r(\cdot,\theta)]=\mu[M_\sigma^\top\overline{Q}_r(\cdot,\theta)]=e^{-\theta(T-r)}\mu[M_\sigma^\top\cQ_r]+\int^T_re^{-\theta(s-r)}\mu[M_\sigma^\top\sQ^s_r]\diff s,
\end{align*}
for $\mu$-a.e.\ $\theta\in\bR_+$ a.s.\ for a.e.\ $r\in[0,T]$. From these relations, together with the definitions \eqref{app_eq_P^V} of $P^\V_2$ and $P^\V_4$, we get the equalities in form of \eqref{app_eq_PQ-mu1} and \eqref{app_eq_PQ-mu2}.

Inserting \eqref{app_eq_PQ-mu2} and the definitions \eqref{app_eq_P^V} of $P^\V_3$ and $P^\V_4$ to \eqref{app_eq_P^r-Ito}, we get
\begin{align*}
	&\sP^r_t(\theta)+\int^r_te^{-\theta(s-t)}\sQ^r_s(\theta)\,\diff W_s\\
	&=e^{-\theta(T-t)}P^\V_2(r)^\top+e^{-\theta(r-t)}P^\V_3(r)+\int^T_re^{-\theta(s-t)}P^\V_4(s,r)^\top\diff s+\int^r_te^{-\theta(s-t)}P^\V_4(r,s)\,\diff s\\
	&=e^{-\theta(T-t)}\bE_t\big[P^\V_2(r)\big]^\top+e^{-\theta(r-t)}\bE_t\big[P^\V_3(r)\big]+\int^T_re^{-\theta(s-t)}\bE_t\big[P^\V_4(s,r)\big]^\top\diff s+\int^r_te^{-\theta(s-t)}\bE_t\big[P^\V_4(r,s)\big]\,\diff s\\
	&\hspace{0.5cm}+\int^r_te^{-\theta(s-t)}\Big\{e^{-\theta(T-s)}Q^\V_2(r,s)^\top+e^{-\theta(r-s)}Q^\V_3(r,s)\\
	&\hspace{4cm}+\int^T_re^{-\theta(\tau-s)}Q^\V_4(\tau,r,s)^\top\,\diff\tau+\int^r_se^{-\theta(\tau-s)}Q^\V_4(r,\tau,s)\,\diff\tau\Big\}\,\diff W_s,
\end{align*}
where we used the martingale representation \eqref{app_eq_BSVIE2-M} and the stochastic Fubini theorem in the second equality. From this, we get the equalities in form of \eqref{app_eq_PQ-r}. Hence, as before, thanks to the representation formulas \eqref{MP_eq_kernel} for the kernels $K_b$ and $K_\sigma$, using Fubini's theorem, we get \eqref{app_eq_PQ^r-mu}. Inserting the expressions \eqref{app_eq_PQ^r-mu} to the definition of $P^\V_4$ in \eqref{app_eq_P^V}, we see that the fourth equation in \eqref{app_eq_BSVIE2} holds.

Furthermore, by the BSEE \eqref{MP_eq_adeq2} for $(\hat{P},\hat{Q})$ and the definitions \eqref{app_eq_P^V} of $P^\V_1$ and $P^\V_3$, together with the formulas \eqref{app_eq_PQ-mu1} and \eqref{app_eq_PQ-mu2}, we have
\begin{align*}
	&\hat{P}_t(\theta_1,\theta_2)=\bE_t\big[\hat{P}_t(\theta_1,\theta_2)\big]\\
	&=\bE_t\Big[-e^{-(\theta_1+\theta_2)(T-t)}\hat{h}_{xx}\vphantom{\int^T_t}\\
	&\hspace{0.5cm}+\int^T_te^{-(\theta_1+\theta_2)(s-t)}\Big\{\big\langle\mu[M_b^\top\hat{p}_s],\hat{b}_{xx}(s)\big\rangle+\big\langle\mu[M_\sigma^\top\hat{q}_s],\hat{\sigma}_{xx}(s)\big\rangle-\hat{f}_{xx}(s)+\hat{\sigma}_x(s)^\top\mu^{\otimes2}[M_\sigma^\top\hat{P}_sM_\sigma]\hat{\sigma}_x(s)\vphantom{\int^T_t}\\
	&\hspace{1cm}+\hat{b}_x(s)^\top\mu[M_b^\top\hat{P}_s(\cdot,\theta_2)]+\hat{\sigma}_x(s)^\top\mu[M_\sigma^\top\hat{Q}_s(\cdot,\theta_2)]+\mu[\hat{P}_s(\theta_1,\cdot)M_b]\hat{b}_x(s)+\mu[\hat{Q}_s(\theta_1,\cdot)M_\sigma]\hat{\sigma}_x(s)\Big\}\,\diff s\Big]\\
	&=e^{-(\theta_1+\theta_2)(T-t)}P^\V_1(t)+\int^T_te^{-(\theta_1+\theta_2)(s-t)}\bE_t\big[P^\V_3(s)\big]\,\diff s\\
	&\hspace{0.5cm}+e^{-\theta_2(T-t)}\int^T_te^{-\theta_1(s-t)}\bE_t\big[P^\V_2(s)\big]\diff s+\int^T_te^{-\theta_2(r-t)}\int^r_te^{-\theta_1(s-t)}\bE_t\big[P^\V_4(r,s)\big]\,\diff s\,\diff r\\
	&\hspace{0.5cm}+e^{-\theta_1(T-t)}\int^T_te^{-\theta_2(s-t)}\bE_t\big[P^\V_2(s)\big]^\top\,\diff s+\int^T_te^{-\theta_1(r-t)}\int^r_te^{-\theta_2(s-t)}\bE_t\big[P^\V_4(r,s)\big]^\top\,\diff s\,\diff r.
\end{align*}
Namely, we have the equality in form of \eqref{app_eq_PQ} for $\hat{P}$. Hence, as before, thanks to the representation formula \eqref{MP_eq_kernel} for $K_\sigma$, using Fubini's theorem, we get \eqref{app_eq_P-mumu}. Inserting the expressions \eqref{app_eq_P-mumu} and \eqref{app_eq_pq-mu} to the definition of $P^\V_3$ in \eqref{app_eq_P^V}, we see that the third equation in \eqref{app_eq_BSVIE2} holds.

Consequently, $(P^\V_1,Q^\V_1,P^\V_2,Q^\V_2,P^\V_3,Q^\V_3,P^\V_4,Q^\V_4)$ is the adapted M-solution of the BSVIE system \eqref{app_eq_BSVIE2}. This completes the proof.
\end{proof}

%%%%%%%%%%%%%%%%%%%%%%%%%%%%
%%%%%% Acknowledgements
%%%%%%%%%%%%%%%%%%%%%%%%%%%%

\section*{Acknowledgments}

The author would like to thank Tianxiao Wang and Jiayin Gong for stimulating discussions and kind hospitality during the author's stay in Chengdu.

%%%%%%%%%%%%%%%%%%%%%%%%%%%%
%%%%%% References
%%%%%%%%%%%%%%%%%%%%%%%%%%%%

\end{document}